\documentclass[11pt]{article}

\usepackage[a4paper,margin=1in]{geometry}
\usepackage{amsmath,amssymb,amsfonts,amsthm}
\usepackage{mathtools}
\usepackage{enumitem}
\usepackage{hyperref}
\usepackage[nameinlink,capitalize]{cleveref}
\usepackage{algorithm}
\usepackage{algpseudocode}
\usepackage{array}
\usepackage{multirow}
\usepackage{graphicx}
\usepackage{subcaption}
\usepackage{epstopdf}
\usepackage{cite}

\hypersetup{colorlinks=true,linkcolor=blue,citecolor=blue,urlcolor=blue}
\newcolumntype{C}[1]{>{\centering\arraybackslash}m{#1}}

\newtheorem{theorem}{Theorem}[section]

\newtheorem{proposition}[theorem]{Proposition}
\newtheorem{corollary}[theorem]{Corollary}
\newtheorem{assumption}[theorem]{Assumption}
\newtheorem{remark}[theorem]{Remark}

\newcommand{\R}{\mathbb{R}}
\newcommand{\ip}[2]{\left\langle #1,#2\right\rangle}
\newcommand{\norm}[1]{\left\|#1\right\|}
\newcommand{\dd}{\mathrm{d}}

\newcommand{\argmin}{\operatorname*{argmin}}
\newcommand{\prox}{\operatorname{prox}}

\newcommand{\sech}{\operatorname{sech}}

\newcommand{\cE}{\mathcal E}

\newcommand{\cH}{\mathcal H}

\newcommand{\bigO}{\mathcal{O}}

\title{A unified continuous-discrete framework for Nesterov acceleration:
transitions between convex and strongly convex regimes\thanks{This work was supported by the Talent Introduction Project of Xihua University (Grant No. Z241102), Sichuan Science and Technology Program (Grant No. 2025ZNSFSC0813) and the National Natural Science Foundation of China (Grant No. 12601606, 12471296).}}

\author{Xin He\thanks{School of Science, Xihua University, Chengdu, Sichuan, China. E-mail: hexinuser@163.com},\and  Ya-Ping Fang \thanks{Department of Mathematics, Sichuan University, Chengdu, Sichuan,  China. E-mail: ypfang@scu.edu.cn}}

\date{\today}

\begin{document}
\maketitle

\begin{abstract}
Classical Nesterov acceleration employs different choices of damping and
inertial parameters in the convex and strongly convex settings, both for
continuous-time dynamics and for discrete algorithms. When the strong
convexity parameter is small, directly using the strongly convex damping or
inertial coefficient may lead to slower early-stage convergence than the
corresponding convex choice, despite its favorable asymptotic exponential or
linear rate. We develop a unified continuous-discrete framework that
encompasses both classical regimes and provides systematic transitions between
them. The resulting coefficient families retain the accelerated convex
behavior at early stages while attaining the strongly convex asymptotic rate.
The continuous-time dynamics arise from a two-state coupling and are analyzed
within a unified Lyapunov framework that yields simultaneous
$\mathcal{O}(1/t^2)$ and exponential convergence estimates, thereby recovering
the classical convex and strongly convex rates. We further derive two classes
of accelerated forward-backward algorithms by discretizing the proposed
dynamics and establish convergence estimates covering the convex, strongly
convex, and intermediate regimes. The framework recovers the classical
Nesterov inertial coefficients and generates hyperbolic, exponential,
algebraic, and polynomial transition families. Numerical experiments
demonstrate the effectiveness of the proposed methods when the strong convexity
parameter is small.
\end{abstract}

\noindent\textbf{Keywords:}
Nesterov acceleration; continuous-time dynamics; accelerated
forward-backward algorithms; transition parameters;
convergence analysis

\section{Introduction}

\subsection{Problem setting and motivation}

Let $\cH$ be a real Hilbert space.  We consider
\begin{equation}\label{ques_main}
    \min_{x\in\cH}\Phi(x),
\end{equation}
where $\Phi$ is $\mu$-strongly convex for some $\mu\ge0$ and
$\operatorname*{argmin}\Phi\neq\varnothing$.  We fix
$x^*\in\operatorname*{argmin}\Phi$ and set $\Phi^*:=\Phi(x^*)$.  When
$\Phi$ is differentiable, $\mu$-strong convexity means that
\begin{equation*}
    \Phi(y)
    \ge
    \Phi(x)
    +
    \ip{\nabla\Phi(x)}{y-x}
    +
    \frac{\mu}{2}\norm{y-x}^2,
    \qquad x,y\in\cH.
\end{equation*}
The cases $\mu=0$ and $\mu>0$ correspond to convex and strongly convex
optimization, respectively.  Nesterov acceleration is one of the fundamental
techniques in first-order optimization.  For the smooth convex problem
\eqref{ques_main}, it achieves an $\bigO(k^{-2})$  convergence rate for the
objective residual \cite{Nesterov1983}, whereas, in the smooth strongly convex
case, a suitable fixed inertial parameter yields a linear convergence rate
\cite{Nesterov2004}. The acceleration techniques have subsequently been extended to composite optimization through accelerated forward-backward schemes
\cite{BeckFista,Tseng2008,Villa13,AujolDossal2015,KimF2018,AttouchC18},
to constrained formulations through primal-dual methods
\cite{XuSIOPT,HeNA,BotMPpd,LuoMC,HeAuto}, and to continuous-time models
\cite{SuJMLR2016,Wibisono16,AttouchMP18,ShiMp,AttouchJEMS,HeSICON}.

 The classical convex and strongly convex Nesterov methods employ different
damping and inertial coefficients. When the convex Nesterov parameters are
applied to $\mu$-strongly convex objectives with $\mu>0$, strong convexity can
improve the standard $\bigO(1/k^2)$ convex convergence rate. Faster polynomial rates
\cite{SuJMLR2016,ApidopoulosMpAL,AujolOptim} and eventual exponential or
linear convergence \cite{LiSYSiopt} have been established in continuous and
discrete settings. These improved rates, however, generally hold only beyond
a $\mu$-dependent time or iteration threshold and need not match the rates
obtained using the classical strongly convex parameters
\cite{Nesterov2004,LinML}. Conversely, when $\mu$ is small, using the
strongly convex parameters from the beginning may lead to slower initial
convergence than using the corresponding convex parameters, despite their
better asymptotic rates.

The central question is therefore how to choose the parameters so as to
retain favorable early-stage convergence without sacrificing the strongly
convex asymptotic rate when $\mu$ is small. To address this question, we
construct time-dependent damping and inertial coefficients for continuous
dynamics and accelerated forward-backward algorithms under a suitable
admissibility condition. The resulting methods retain the convergence
estimates of convex Nesterov acceleration at early stages and attain the
exponential or linear rates of strongly convex Nesterov acceleration at later
stages, while recovering the classical convex and strongly convex Nesterov
parameters as endpoint cases.

\subsection{Nesterov acceleration for convex and strongly convex objectives}

For an $L$-smooth objective $\Phi$, Nesterov's accelerated gradient (NAG)
method can be written as
\begin{equation}\label{al_NAG}
\left\{
\begin{aligned}
    y_k &= x_k+\beta_k(x_k-x_{k-1}),\\
    x_{k+1} &= y_k-s\nabla\Phi(y_k),
\end{aligned}
\right.
\end{equation}
where $s\le 1/L$. In the convex case $\mu=0$, the classical inertial
coefficient is
$
    \beta_k=\frac{t_k-1}{t_{k+1}},
$
where
$t_1=1$ and
$t_{k+1}=\frac{1+\sqrt{1+4t_k^2}}{2}$.
This choice yields
$
    \Phi(x_k)-\Phi^*=\bigO(k^{-2})
$
\cite{Nesterov1983}.
More generally, an explicit family of inertial coefficients is given by
\begin{equation}\label{eq_para_conv}
    \beta_k=\frac{k-1}{k+\gamma-1},
    \qquad \gamma\ge3,
\end{equation}
which yields the same $\bigO(k^{-2})$ rate for the objective residual
\cite{ChambolleJota,SuJMLR2016,AttouchMP18}. The continuous-time
counterpart of this family is the following dynamical system with vanishing
damping \cite{SuJMLR2016}:
\begin{equation}\label{dy_NAG_C}
    \ddot x(t)+\frac{\gamma}{t}\dot x(t)+\nabla\Phi(x(t))=0.
\end{equation}
For $\gamma\ge3$, its trajectories satisfy
$
    \Phi(x(t))-\Phi^*=\bigO(t^{-2}),
$
and the case $\gamma=3$ corresponds to the formal continuous-time limit of
the classical convex Nesterov method \cite{SuJMLR2016,Nesterov1983}. Further asymptotic properties of the continuous dynamics, including improved
decay rates and convergence of trajectories, have been studied in
\cite{AttouchSiopt16,AttouchMP18}, while recent work has established
convergence of the iterates for accelerated discrete algorithms
\cite{JangArxiv,BotArxiv}. More general time-dependent damping and stabilization terms have been considered in \cite{AttouchJEMS,ShiMp,BotMp,AttouchCabot2017}, while inertial algorithms and their continuous-time counterparts have been analyzed in \cite{Apidopoulos20,HeCOAP,AttouchC18,LuoChen2022}.

In the strongly convex case $\mu>0$, Nesterov's accelerated gradient method
uses a fixed momentum coefficient that depends on the curvature parameters.
In particular, when $s=1/L$, the classical strongly convex coefficient in
\eqref{al_NAG} is
\begin{equation}\label{eq_para_SC}
    \beta_k
    =
    \frac{1-\sqrt{\mu/L}}{1+\sqrt{\mu/L}},
\end{equation}
which yields
$
    \Phi(x_k)-\Phi^*
    =
    \bigO\left((1-\sqrt{\mu/L})^k\right)
$
\cite{Nesterov2004,LinML}. A standard continuous-time counterpart is
\begin{equation}\label{dy_NAG_SC}
    \ddot x(t)+2\sqrt{\mu}\dot x(t)+\nabla\Phi(x(t))=0,
\end{equation}
for which
$
    \Phi(x(t))-\Phi^*
    =
    \bigO\left(e^{-\sqrt{\mu}t}\right)
$
\cite{Siegel2019,LuoChen2022,WilsonJmlr}. Thus, convex and strongly convex
acceleration use different coefficients in both discrete and continuous
time. The convex case uses a time-dependent inertial coefficient and
vanishing damping, whereas the strongly convex case uses a fixed inertial
coefficient and constant damping.

For composite objectives, Nesterov acceleration leads to accelerated
forward-backward methods. Consider
\begin{equation}\label{ques_comp}
    \min_{x\in\cH}\Phi(x):=f(x)+g(x),
\end{equation}
where $f:\cH\to\R$ is $L$-smooth and $\mu$-strongly convex for some
$\mu\ge0$, and $g:\cH\to\R\cup\{+\infty\}$ is proper, lower
semicontinuous, and convex. Then $\Phi$ is $\mu$-strongly convex. The
gradient step in \eqref{al_NAG} is replaced by a forward-backward step, in
which $f$ is treated explicitly and $g$ is handled through its proximal
map. Two standard forms of Nesterov acceleration are used for problem
\eqref{ques_comp}: one-sequence inertial schemes, including the fast
iterative shrinkage-thresholding algorithm (FISTA) \cite{BeckFista} and its
variants \cite{AujolDossal2015,Villa13,KimF2018,AttouchSiopt16}, and
two-sequence schemes of Tseng's type \cite{Tseng2008}, which are based on
the coupling of two iterate sequences
\cite{Nesterov12,LinML,LuoChen2022}. With suitable parameter choices, both
forms attain an $\bigO(k^{-2})$ rate for the objective residual when
$\mu=0$ and a linear rate of the form
$
    \bigO\left((1-\sqrt{\mu/L})^k\right)
$
when $\mu>0$. In the smooth case $g\equiv0$, the two formulations can be reduced to
equivalent iterations after an appropriate identification of their variables, but generally
produce different iterations when $g\neq 0$ is nonsmooth \cite{LinML}.
Continuous-time models and numerical discretizations provide a complementary
approach to the design and Lyapunov analysis of accelerated
forward-backward schemes
\cite{HeCOAP,AttouchMP18,LuoChen2022,SuJMLR2016,ShiMp,AttouchC18}.

Existing results provide separate parameter choices for the convex and
strongly convex cases. The present work studies coefficient families that
connect these choices in both continuous and discrete time, with particular
attention to the regime in which $\mu$ is small.

 \subsection{Unified formulations and existing limitations}

The preceding discussion raises the central question of this paper: can the
classical convex and strongly convex Nesterov parameters be embedded into a
unified continuous-discrete framework by means of coefficient families that
connect the two regimes? When $\mu>0$ is small, such a transition should
retain the favorable initial behavior of the convex parameters while
recovering the faster strongly convex rate asymptotically. In continuous
time, this amounts to constructing a family of damping coefficients
$\delta_\mu(t)$ such that
\[
\begin{aligned}
    \delta_\mu(t)&\to\frac{\gamma}{t}
    &&\text{as }\mu\downarrow0,\qquad \gamma\ge3,\\
    \delta_\mu(t)&\to2\sqrt{\mu}
    &&\text{as }t\to\infty.
\end{aligned}
\]
The first limit ensures consistency with the convex accelerated damping law
for each fixed $t>0$ as the strong convexity parameter vanishes, whereas the
second ensures that, for each fixed $\mu>0$, the damping approaches the
classical strongly convex damping at large times. The discrete counterpart
is to construct inertial coefficients that reduce to the convex parameters
in \eqref{eq_para_conv} as $\mu\downarrow0$ and approach the fixed strongly
convex parameter in \eqref{eq_para_SC} as $k\to\infty$. The objective is to
obtain the convex $\bigO(k^{-2})$ behavior at early stages and the accelerated
strongly convex linear rate asymptotically within a unified Lyapunov
framework.

Several works have investigated general mechanisms underlying acceleration.
Variational formulations of accelerated dynamics were developed in
\cite{Wibisono16}, while Lyapunov interpretations were provided in
\cite{WilsonJmlr}. High-resolution differential equations and related
differential-equation and geometric descriptions of accelerated methods were
considered in
\cite{ShiMp,LuoChen2022,AujolOptim,HeCOAP,Muehlebach19}. These works provide
important insights into the mechanisms of acceleration, but they do not
construct coefficient families connecting the convex and strongly convex
parameters through different admissible transition coefficients.

The formulations in
\cite{LuoChen2022,ChenL21,KimYang2023,ChambolleActa}
are more closely related to the present work. In the scaled NAG formulation
of \cite{LuoChen2022}, the dynamics take the form
\begin{equation}\label{dy:LuoC}
    \ddot x(t)
    +
    \frac{1}{\lambda(t)}
    \left(\mu+\lambda(t)^2-\dot\lambda(t)\right)\dot x(t)
    +
    \nabla\Phi(x(t))
    =
    0,
\end{equation}
where the positive scaling function $\lambda(t)$ satisfies
\begin{equation}\label{eq_lamass}
    2\dot\lambda(t)\le \mu-\lambda(t)^2.
\end{equation}
This condition leads to a unified Lyapunov analysis, but it also restricts
the admissible damping coefficients. Indeed, setting
$\theta(t)=1/\lambda(t)$ transforms \eqref{eq_lamass} into
\[
    2\dot\theta(t)+\mu\theta(t)^2\ge1.
\]
When $\mu=0$, this implies $\dot\theta(t)\ge1/2$. Hence, for an
asymptotically linear scaling satisfying
$\theta(t)\sim ct$ and $\dot\theta(t)\to c$, one necessarily has
$c\ge1/2$. The induced damping
\[
    \delta_\mu(t)
    =
    \frac{1+\dot\theta(t)}{\theta(t)}+\mu\theta(t)
\]
therefore satisfies, when $\mu=0$,
\[
    \delta_0(t)
    \sim
    \frac{\gamma}{t},
    \qquad
    \gamma=\frac{1+c}{c}\le3.
\]
The scaled NAG dynamic in \eqref{dy:LuoC} therefore includes the critical
convex damping $3/t$ and damping laws with smaller leading coefficients, but
exclude the supercritical family $\gamma/t$ with $\gamma>3$.

At the discrete level, the algorithms derived from the scaled dynamics in
\cite{LuoChen2022,ChenL21} yield estimates of the form
\[
    \Phi(x_k)-\Phi^*
    =
    \bigO\left(
        \min\left\{
            \frac{1}{k^2},
            (1+\sqrt{\mu/L})^{-k}
        \right\}
    \right)
\]
in both smooth and composite settings. This estimate combines the convex
sublinear rate with a strongly convex linear rate. The latter is, however,
slower than the classical accelerated strongly convex rate. Indeed, for
$\sqrt{\mu/L}\in(0,1]$, one has
$
    \frac{1}{1+\sqrt{\mu/L}}>1-\sqrt{\mu/L},
$
and hence $(1+\sqrt{\mu/L})^{-k}$ decays more slowly than $(1-\sqrt{\mu/L})^k$.

The equality case of \eqref{eq_lamass} underlies the transition studied in
\cite{KimYang2023}. Under the singular asymptotic condition
$\lambda(t)\sim2/t$ as $t\downarrow0$, the equation
\[
    2\dot\lambda(t)=\mu-\lambda(t)^2
\]
has the explicit solution
\[
    \lambda(t)
    =
    \sqrt{\mu}\coth\left(\frac{\sqrt{\mu}}{2}t\right).
\]
The induced damping transitions from the critical convex value $3/t$ to the
strongly convex value $2\sqrt{\mu}$. The associated discrete method for
smooth objectives similarly connects the convex $\bigO(k^{-2})$ rate with
the accelerated linear rate $(1-\sqrt{\mu/L})^k$. This construction,
however, corresponds only to the critical choice $\gamma=3$ and to one
specific hyperbolic transition coefficient.

A closely related discrete transition already appears in 
\cite[Algorithm~5]{ChambolleActa}. Specializing that algorithm to the setting in which
the smooth term is $\mu$-strongly convex and the nonsmooth term is convex,
and taking $s\mu\in[0,1)$, yields the parameter rule
\begin{equation}\label{eq_canonical_cp}
\begin{aligned}
    \beta_k
    =
    \frac{
        (t_k-1)(1-s\mu t_{k+1})
    }{
        (1-s\mu)t_{k+1}
    },
\end{aligned}
\end{equation}
where
\[
    t_{k+1}
    =
    \frac{
        1-s\mu t_k^2+
        \sqrt{(1-s\mu t_k^2)^2+4t_k^2}
    }{2}.
\]
With the standard choice $s=1/L$, the resulting method satisfies
\[
    \Phi(x_k)-\Phi^*
    =
    \bigO\left(
        \min\left\{
            \frac{1}{k^2},
            \left(1-\sqrt{\frac{\mu}{L}}\right)^k
        \right\}
    \right).
\]
When $\mu=0$, \eqref{eq_canonical_cp} reduces to the classical convex
Nesterov and FISTA parameter rule
\cite{Nesterov1983,BeckFista}. Its momentum coefficient is asymptotically
equivalent, up to an index shift, to the critical Nesterov coefficient
$
    \frac{k-1}{k+2}.
$
For every fixed $s\mu\in(0,1)$, one has $t_k\to \frac{1}{\sqrt{s\mu}}$,
\[
    \beta_k\to
    \frac{1-\sqrt{s\mu}}{1+\sqrt{s\mu}}
    \qquad
    \text{as }k\to\infty.
\]
Thus, this parameter rule connects the critical convex momentum with the
classical fixed momentum for strongly convex acceleration. \cite[Algorithm~5]{ChambolleActa} therefore provides an important discrete benchmark and
already unifies the convex and strongly convex parameter regimes.
Nevertheless, it specifies only one particular equality-based parameter
rule. Its convex limit is tied to the critical Nesterov asymptotic and does
not cover the supercritical family associated with the inertial coefficient
in \eqref{eq_para_conv} for $\gamma>3$. Moreover, although this equality rule is sufficient to attain the standard accelerated convergence rates, this
does not imply that it is the unique admissible choice or that it provides
the best transient numerical behavior when $\mu$ is small. In particular,
the discrete derivation does not provide a continuous-time mechanism that
explains the full evolution of the coefficients or generates alternative
transitions with the same endpoints.

The restriction to the critical convex parameter is significant because the
supercritical damping $\gamma/t$ with $\gamma>3$ yields faster asymptotic
properties for several convex inertial dynamics
\cite{May17,AttouchCabot2017}. The corresponding discrete inertial
parameters also enjoy improved
objective decay
\cite{AttouchSiopt16,ChambolleJota}. It is therefore natural to ask whether,
starting from any accelerated convex damping $\gamma/t$ with $\gamma\ge3$,
one can construct different continuous and discrete transitions toward the
strongly convex damping $2\sqrt{\mu}$ and the corresponding fixed momentum
parameter. Such a construction would provide a unified extension of the
models in
\cite{SuJMLR2016,AttouchCabot2017,ChambolleJota,Siegel2019,KimF2018,
LuoChen2022,AttouchMP18},
while allowing the transient behavior to be adjusted without sacrificing
the accelerated strongly convex asymptotic rate.

\subsection{Contributions and organization}

This paper develops a two-state Nesterov construction that provides a unified
continuous-discrete framework for convex acceleration, strongly convex
acceleration, and transitions between the two regimes. At the continuous-time
level, eliminating the auxiliary state variable yields the second-order
dynamic
\begin{equation}\label{dy_main}
    \ddot x(t)
    +
    \delta_\mu(t)\dot x(t)
    +
    \nabla\Phi(x(t))
    =
    0,
    \qquad
    \delta_\mu(t)
    =
    \frac{1+\dot\theta(t)}{\theta(t)}
    +
    \mu\theta(t),
\end{equation}
where $t\ge t_0>0$. A second-order equation of the same algebraic form was
obtained in \cite{LuoChen2022} through a scaled NAG formulation and a
time-scaling technique. The present work differs, however, in both its
derivation and its admissibility condition. Specifically, our dynamics are
derived directly from a two-state system and analyzed under the following
assumption.

\begin{assumption}\label{ass_coeff}
Suppose that $\Phi$ is $\mu$-strongly convex for some $\mu\ge0$. Let
$\theta:[t_0,+\infty)\to(0,+\infty)$ be continuously differentiable and
nondecreasing, and assume that
\[
    0
    <
    2\dot\theta(t)+\mu\theta(t)^2
    \le
    1
    \qquad
  \forall t\ge t_0.
\]
\end{assumption}

By contrast, after the change of variables
$\theta(t)=1/\lambda(t)$, the condition imposed in
\cite{LuoChen2022} becomes
\[
    2\dot\theta(t)+\mu\theta(t)^2
    \ge
    1.
\]
Thus, although the two approaches lead to second-order equations of the same
algebraic form, they impose oppositely directed admissibility conditions and
generate different classes of damping coefficients. Their common boundary
is given by
\[
    2\dot\theta(t)+\mu\theta(t)^2=1,
\]
which is associated with the hyperbolic transition considered in
\cite{LuoChen2022,KimYang2023}.

Under Assumption~\ref{ass_coeff}, we establish a Lyapunov estimate associated
with the scale
\[
    A(t)
    =
    \theta(t)^2
    e^{
        \int_{t_0}^{t}\mu\theta(s)\,\dd s}
\]
and show that
\[
    \Phi(x(t))-\Phi^*
    \le
    \frac{C}{A(t)}
\]
for some constant $C>0$ depending on the initial data. The proposed
framework recovers the classical convex damping $\gamma/t$ for
$\gamma\ge3$ \cite{SuJMLR2016,AttouchCabot2017} and the strongly convex
damping $2\sqrt{\mu}$ \cite{Siegel2019,WilsonJmlr} as endpoint cases. It
further generates hyperbolic, exponential, algebraic, and polynomial
transition coefficients between these two regimes. For small $\mu$, the
resulting dynamics retain convex-type accelerated behavior at early stages
while recovering the strongly convex exponential rate asymptotically. The
admissible coefficient families and their convergence properties are studied
in detail in Section~\ref{sec:dyna}.

Through suitable discretizations, the same two-state construction yields
one-sequence and two-sequence accelerated forward-backward algorithms for
the composite optimization problem \eqref{ques_comp}. These two classes
respectively extend the one-sequence inertial schemes studied in
\cite{Nesterov2004,BeckFista,ChambolleJota} and the two-sequence schemes
studied in \cite{Nesterov12,Tseng2008,LinML}. Both classes are governed by
the same parameter sequence, which induces the inertial coefficient
\[
    \beta_k
    =
    \frac{
        (\theta_k-\sqrt{s})
        (1-\mu\sqrt{s}\theta_{k+1})
    }{
        (1-s\mu)\theta_{k+1}
    }.
\]
Under the corresponding discrete admissibility conditions, we introduce the
Lyapunov scale
\[
    A_k
    =
    \theta_k^2
    \prod_{i=1}^{k}
    \frac{1}{1-\sqrt{s}\mu\theta_i},
    \qquad
    s\le\frac{1}{L},
\]
and establish estimates of the form
\[
    \Phi(x_k)-\Phi^*
    \le
    \frac{C}{A_k}.
\]
The resulting coefficients recover, as special cases, the classical convex
and strongly convex Nesterov parameters and their forward-backward variants
\cite{Nesterov1983,Nesterov2004,BeckFista,ChambolleJota,AttouchC18,
LinML}. Moreover, under the identification
$\theta_k=\sqrt{s}\,t_k$, the proposed inertial coefficient recovers the
 transition rule \eqref{eq_canonical_cp} from
\cite[Algorithm~5]{ChambolleActa}. Beyond this equality rule, the discrete
framework generates hyperbolic, exponential, algebraic, and polynomial
transition families that parallel their continuous-time counterparts.
Consequently, the continuous dynamics and the discrete algorithms are
derived and analyzed through the same underlying two-state mechanism at the
levels of their endpoint parameters, transition coefficients, Lyapunov
structures, and convergence estimates.

The main contributions of this paper are summarized as follows.
\begin{enumerate}
    \item[$\bullet$]
    We derive the second-order dynamic \eqref{dy_main} directly from a
    two-state Nesterov system and establish a unified Lyapunov analysis for
    convex, strongly convex, and transition regimes \cite{SuJMLR2016,Siegel2019,AttouchCabot2017,KimYang2023,AttouchMP18} . Although
    \eqref{dy_main} has the same algebraic form as the scaled NAG dynamic in
    \cite{LuoChen2022}, our derivation is based on a two-state coupling and
    the oppositely directed admissibility condition in
    Assumption~\ref{ass_coeff}. This difference yields a broader class of
    admissible damping coefficients, including the critical convex damping
    $3/t$, the supercritical family $\gamma/t$ with $\gamma>3$, and
    hyperbolic, exponential, algebraic, and polynomial transitions toward
    the strongly convex damping $2\sqrt{\mu}$. The corresponding Lyapunov
    scale provides convergence estimates that recover the classical convex
    and strongly convex rates at the endpoint regimes.

    \item[$\bullet$]
    We derive one-sequence and two-sequence accelerated forward-backward
    algorithms by discretizing the same two-state mechanism and establish
    convergence estimates for convex, strongly convex, and transition
    regimes. The discrete framework recovers the classical accelerated
    parameter choices
    \cite{Nesterov1983,Nesterov2004,BeckFista,Tseng2008,LinML,
    ChambolleJota,AttouchC18}, includes the transition rule
    \eqref{eq_canonical_cp} from \cite[Algorithm~5]{ChambolleActa}, and
    generates hyperbolic, exponential, algebraic, and polynomial transition
    families paralleling their continuous-time counterparts. Thus, the
    continuous dynamics and discrete algorithms share the same underlying
    parameter mechanism, Lyapunov structure, and endpoint behavior.

  \item[$\bullet$]
    We conduct numerical experiments for both the continuous dynamics and
    the discrete algorithms. On the tested problems with a small strong
    convexity parameter, suitable choices from the proposed transition
    families exhibit more favorable transient behavior than the classical
    convex and strongly convex endpoint choices and the equality-based
    transition, while retaining the accelerated strongly convex asymptotic
    rate.
\end{enumerate}
The remainder of the paper is organized as follows.
Section~\ref{sec:dyna} develops the continuous dynamics, establishes the
Lyapunov estimates, and studies the classical and transition damping
coefficients. Section~\ref{sec_discrete} derives the accelerated
forward-backward algorithms and proves their convergence properties.
Section~\ref{sec_numerics} presents numerical experiments for both the
continuous dynamics and the discrete algorithms. The final section
concludes the paper.

\section{Unified continuous-time dynamics}\label{sec:dyna}

This section develops the continuous-time component of the unified framework.
We first derive the dynamic \eqref{dy_main} from a two-state Nesterov
coupling and clarify the roles of the coefficient functions. We then
establish a unified Lyapunov estimate under the proposed admissibility
condition. The resulting framework recovers the classical convex and
strongly convex dynamics as endpoint cases and, more importantly, generates
several new families of coefficients that transition between the two
regimes, including hyperbolic, exponential, algebraic, and polynomial
coefficients. For each coefficient family, we derive the corresponding
convergence rate for the objective residual.

\subsection{Derivation from a two-state Nesterov coupling}

We begin with a two-state dynamical system involving two coefficient
functions. This formulation separates the inertial time scale, the feedback
term, and the gradient  coefficient. Consider
\begin{equation}\label{dy_coup}
\begin{aligned}
    \dot x(t)
    &=
    \frac{1}{\theta(t)}
    \bigl(z(t)-x(t)\bigr),\\
    \dot z(t)
    &=
    -\mu\theta(t)\bigl(z(t)-x(t)\bigr)
    -\beta(t)\nabla\Phi(x(t)).
\end{aligned}
\end{equation}
Here, $x(t)$ is the primal state and $z(t)$ is the auxiliary accelerated
state. The coefficient $\theta(t)>0$ determines the inertial time scale in
the first equation, $\mu\theta(t)$ controls the feedback acting on
$z(t)-x(t)$, and $\beta(t)>0$ scales the gradient.

The first equation in \eqref{dy_coup} gives
\[
    z(t)=x(t)+\theta(t)\dot x(t).
\]
Differentiating this identity yields
\[
    \dot z(t)
    =
    \bigl(1+\dot\theta(t)\bigr)\dot x(t)
    +
    \theta(t)\ddot x(t).
\]
On the other hand, since
$z(t)-x(t)=\theta(t)\dot x(t)$, the second equation in
\eqref{dy_coup} gives
\[
    \dot z(t)
    =
    -\mu\theta(t)^2\dot x(t)
    -
    \beta(t)\nabla\Phi(x(t)).
\]
Combining these two expressions, we obtain
\[
    \theta(t)\ddot x(t)
    +
    \bigl(1+\dot\theta(t)+\mu\theta(t)^2\bigr)\dot x(t)
    +
    \beta(t)\nabla\Phi(x(t))
    =
    0.
\]
Dividing by $\theta(t)$ leads to
\begin{equation}\label{eq_gen_dy}
    \ddot x(t)
    +
    \left(
        \frac{1+\dot\theta(t)}{\theta(t)}
        +
        \mu\theta(t)
    \right)\dot x(t)
    +
    \frac{\beta(t)}{\theta(t)}
    \nabla\Phi(x(t))
    =
    0.
\end{equation}
Thus, the coefficient of the gradient force is
$\beta(t)/\theta(t)$. To obtain a unit coefficient in front of
$\nabla\Phi(x(t))$, we impose the normalization
$
    \beta(t)=\theta(t).
$
Equation \eqref{eq_gen_dy} then reduces to the unified dynamic
\eqref{dy_main}.

This derivation shows that the damping coefficient in \eqref{dy_main} is
generated by the two-state coupling rather than prescribed independently.
In particular, the term
$
    \frac{1+\dot\theta(t)}{\theta(t)}
$
arises from the inertial relation
$z(t)=x(t)+\theta(t)\dot x(t)$, whereas the term
$\mu\theta(t)$ is induced by the feedback in the second equation of
\eqref{dy_coup}. The normalization $\beta(t)=\theta(t)$ fixes the gradient
scale without otherwise restricting the admissible choices of
$\theta(t)$. This two-state representation provides the basis for the
coefficient conditions, convergence analysis, and transition coefficients
developed below.

\subsection{Convergence analysis}

Throughout this subsection, $\Phi:\cH\to\R$ is continuously Fr\'echet
differentiable and $\mu$-strongly convex for some $\mu\ge0$, with
$\operatorname*{argmin}\Phi\neq\varnothing$. We work under
Assumption~\ref{ass_coeff}.

The upper bound in Assumption~\ref{ass_coeff} is the compatibility condition
required in the Lyapunov estimate below, whereas the strict lower bound
ensures that the time-dependent weight introduced below is strictly
increasing.  Under Assumption~\ref{ass_coeff}, the damping coefficient in
\eqref{dy_main} is strictly positive. Indeed,
\begin{equation}\label{eq_deltaP}
	\delta_\mu(t) = \frac{1+\dot\theta(t)}{\theta(t)}+\mu\theta(t) = \frac{1}{\theta(t)} + \frac{2\dot\theta(t)+\mu \theta(t)^2}{2\theta(t)} + \frac{\mu\theta(t)}{2}> \frac{1}{\theta(t)}+\frac{\mu\theta(t)}{2}>0.
\end{equation}
Thus, \eqref{dy_main} is a dissipative second-order system with a positive
time-dependent damping coefficient. We first establish global well-posedness.

\begin{proposition}\label{prop:wellposed}
Suppose that Assumption~\ref{ass_coeff} holds and that $\nabla\Phi$ is
locally Lipschitz continuous on $\cH$. Then, for every initial condition
$(x_0,v_0)\in\cH\times\cH$, the dynamic \eqref{dy_main} admits a unique
global solution
$
    x\in C^2([t_0,+\infty),\cH)
$
satisfying $x(t_0)=x_0$ and $\dot x(t_0)=v_0.$
\end{proposition}

\begin{proof}
The dynamic \eqref{dy_main} can be written as the first-order system
\[
    \frac{\dd}{\dd t}
    \begin{pmatrix}
        x(t)\\
        \dot x(t)
    \end{pmatrix}
    =
    \begin{pmatrix}
        \dot x(t)\\
        -\delta_\mu(t)\dot x(t)-\nabla\Phi(x(t))
    \end{pmatrix}.
\]
Since $\delta_\mu$ is continuous and $\nabla\Phi$ is locally Lipschitz
continuous, the Cauchy-Lipschitz theorem yields, for every initial
condition, a unique maximal solution
$
    x\in C^2([t_0,T),\cH)
$
for some $T\in(t_0,+\infty]$.

Consider the mechanical energy
\[
    \mathcal W(t)
    :=
    \frac{1}{2}\norm{\dot x(t)}^2
    +
    \Phi(x(t))-\Phi^*.
\]
Using \eqref{dy_main} and \eqref{eq_deltaP}, we obtain
\[
\begin{aligned}
    \dot{\mathcal W}(t)
    &=
    \ip{\dot x(t)}{\ddot x(t)}
    +
    \ip{\nabla\Phi(x(t))}{\dot x(t)}\\
    &=
    -\delta_\mu(t)\norm{\dot x(t)}^2
    \le
    0.
\end{aligned}
\]
Since $\Phi(x(t))-\Phi^*\ge0$, it follows that
\[
    \sup_{t\in[t_0,T)}
    \norm{\dot x(t)}
    \le \sup_{t\in[t_0,T)}\sqrt{2\mathcal W(t)}\le \sqrt{2\mathcal W(t_0)}<+\infty.
\]

Suppose, by contradiction, that $T<+\infty$. The boundedness of $\dot x$
implies that $x$ is Lipschitz continuous on $[t_0,T)$, and hence there
exists $x_T\in\cH$ such that
\[
    x(t)\to x_T
    \qquad
    \text{as }t\uparrow T.
\]
Since $\delta_\mu$ is continuous, it is bounded on $[t_0,T]$. Moreover,
the continuity of $\nabla\Phi$ and the convergence $x(t)\to x_T$ imply
\[
    \sup_{t\in[t_0,T)}
    \norm{\nabla\Phi(x(t))}
    <+\infty.
\]
It then follows from \eqref{dy_main} that
\[
    \sup_{t\in[t_0,T)}
    \norm{\ddot x(t)}
    <+\infty.
\]
Consequently, $\dot x$ is Lipschitz continuous on $[t_0,T)$, and there
exists $v_T\in\cH$ such that
\[
    \dot x(t)\to v_T
    \qquad
    \text{as }t\uparrow T.
\]
Applying the local existence theorem at the state $(T,x_T,v_T)$ extends the
solution beyond $T$, contradicting the maximality of $[t_0,T)$. Therefore,
$T=+\infty$, which proves the proposition.
\end{proof}

We now introduce the  time-dependent weight associated with $\theta(t)$:
\begin{equation}\label{eq_Acont}
    A(t)
    :=
    \theta(t)^2
    \ e^{
        \int_{t_0}^{t}\mu\theta(s)\,\dd s}.
\end{equation}
Its derivative is
\begin{equation}\label{eq_Adot}
    \dot A(t)
    =
    \left(
        2\frac{\dot\theta(t)}{\theta(t)}
        +
        \mu\theta(t)
    \right)A(t).
\end{equation}
It follows from Assumption~\ref{ass_coeff} that $A(t)$ is strictly increasing.
The particular form of $A(t)$ is chosen so that the mixed terms cancel when
the Lyapunov function is differentiated.

Let $x(t)$ be a global solution of dynamic \eqref{dy_main}. For
$x^*\in\operatorname*{argmin}\Phi$, define
\begin{equation}\label{eq_Econt}
    \mathcal E(t)
    :=
    A(t)\bigl(\Phi(x(t))-\Phi^*\bigr)
    +
    \frac{A(t)}{2\theta(t)^2}
    \norm{
        x(t)-x^*+\theta(t)\dot x(t)
    }^2.
\end{equation}
The following theorem provides a general convergence estimate. The rates for
the classical endpoint dynamics and the transition dynamics will be obtained
by evaluating $A(t)$ for the corresponding choices of $\theta(t)$.

\begin{theorem}\label{thm_cont_main}
Let $x\in C^2([t_0,+\infty),\cH)$ be a global solution of
\eqref{dy_main}, let
$x^*\in\operatorname*{argmin}\Phi$, and suppose that
Assumption~\ref{ass_coeff} holds. Then, for every $t\ge t_0$,
\[
    \Phi(x(t))-\Phi^*
    \le
    \frac{\mathcal E(t_0)}{A(t)}.
\]
\end{theorem}

\begin{proof}
Set
\[
    v(t)
    :=
    x(t)-x^*+\theta(t)\dot x(t).
\]
Using \eqref{dy_main}, we obtain
\[
\begin{aligned}
    \dot v(t)
    &=
    \bigl(1+\dot\theta(t)\bigr)\dot x(t)
    +
    \theta(t)\ddot x(t)\\
    &=
    \bigl(
        1+\dot\theta(t)-\theta(t)\delta_\mu(t)
    \bigr)\dot x(t)
    -
    \theta(t)\nabla\Phi(x(t))\\
    &=
    -\mu\theta(t)^2\dot x(t)
    -
    \theta(t)\nabla\Phi(x(t)).
\end{aligned}
\]
Differentiating \eqref{eq_Econt} gives
\begin{equation}\label{eq_dotE1}
\begin{split}
    \dot\cE(t)
    ={}&
    \dot A(t)(\Phi(x(t))-\Phi^*)
    +A(t)\ip{\nabla\Phi(x(t))}{\dot x(t)}\\
    &+
    \left(
        \frac{\dot A(t)}{2\theta(t)^2}
        -\frac{A(t)\dot\theta(t)}{\theta(t)^3}
    \right)\norm{v(t)}^2
    +\frac{A(t)}{\theta(t)^2}\ip{v(t)}{\dot v(t)}\\
    ={}&
    \dot A(t)(\Phi(x(t))-\Phi^*)
    +\frac{\mu A(t)}{2\theta(t)}\norm{v(t)}^2
    -\mu A(t)\ip{x(t)-x^*}{\dot x(t)}\\
    &-\mu A(t)\theta(t)\norm{\dot x(t)}^2
    -\frac{A(t)}{\theta(t)}
    \ip{x(t)-x^*}{\nabla\Phi(x(t))},
\end{split}	
\end{equation}
where the second equality follows from \eqref{eq_Adot}, which gives
\[
    \frac{\dot A(t)}{2\theta(t)^2}
    -
    \frac{A(t)\dot\theta(t)}{\theta(t)^3}
    =
    \frac{\mu A(t)}{2\theta(t)}.
\]
Since
\[
    \norm{v(t)}^2
    =
    \norm{x(t)-x^*}^2
    +2\theta(t)\ip{x(t)-x^*}{\dot x(t)}
    +\theta(t)^2\norm{\dot x(t)}^2,
\]
equation \eqref{eq_dotE1} reduces to
\begin{equation}\label{eq_Edot}
\begin{aligned}
    \dot{\mathcal E}(t)
    ={}&
    \dot A(t)
    \bigl(\Phi(x(t))-\Phi^*\bigr)
    +
    \frac{\mu A(t)}{2\theta(t)}
    \norm{x(t)-x^*}^2\\
    &-
    \frac{\mu A(t)\theta(t)}{2}
    \norm{\dot x(t)}^2
    -
    \frac{A(t)}{\theta(t)}
    \ip{x(t)-x^*}{\nabla\Phi(x(t))}.
\end{aligned}
\end{equation}
By the $\mu$-strong convexity of $\Phi$,
\[
    \ip{x(t)-x^*}{\nabla\Phi(x(t))}
    \ge
    \Phi(x(t))-\Phi^*
    +
    \frac{\mu}{2}
    \norm{x(t)-x^*}^2.
\]
Combining this inequality with \eqref{eq_Adot} and \eqref{eq_Edot}, we
obtain
\[
    \dot{\mathcal E}(t)
    \le
    -\frac{A(t)}{\theta(t)}
    \bigl(
        1-2\dot\theta(t)-\mu\theta(t)^2
    \bigr)
    \bigl(\Phi(x(t))-\Phi^*\bigr)-
    \frac{\mu A(t)\theta(t)}{2}
    \norm{\dot x(t)}^2
    \le
    0,
\]
where the last inequality follows from
Assumption~\ref{ass_coeff}. Therefore, $\mathcal E$ is nonnegative and
nonincreasing. Hence,
\[
    A(t)\bigl(\Phi(x(t))-\Phi^*\bigr)
    \le
    \mathcal E(t)
    \le
    \mathcal E(t_0),
\]
which proves the result.
\end{proof}

  \medskip
\noindent\textbf{Relation to the Luo-Chen  \cite{LuoChen2022} and Kim-Yang \cite{KimYang2023} dynamics.} As discussed in the introduction, the dynamic \eqref{dy_main} has the same
algebraic form as the scaled NAG equation studied in
\cite{LuoChen2022}, but the two formulations impose oppositely directed
conditions on their coefficient functions. We record here the precise
transformation and analyze their common equality branch.

Set
\[
    \lambda(t):=\frac{1}{\theta(t)}.
\]
Then
\[
   \delta_{\mu}(t)= \frac{1+\dot\theta(t)}{\theta(t)}
    +
    \mu\theta(t)
    =
    \frac{1}{\lambda(t)}
    \left(
        \mu+\lambda(t)^2-\dot\lambda(t)
    \right),
\]
and \eqref{dy_main} corresponds algebraically to the scaled NAG equation
\eqref{dy:LuoC}; see \cite[Eq.~(71)]{LuoChen2022}. The coefficient
condition imposed in that work is
\begin{equation}\label{eq:luoCass}
    2\dot\lambda(t)
    \le
    \mu-\lambda(t)^2.
\end{equation}
Under the identification $\lambda(t)=1/\theta(t)$,
condition \eqref{eq:luoCass} is equivalent to
\[
    2\dot\theta(t)+\mu\theta(t)^2
    \ge
    1,
\]
whereas Assumption~\ref{ass_coeff} requires
\[
    0
    <
    2\dot\theta(t)+\mu\theta(t)^2
    \le
    1.
\]
Thus, the two admissibility conditions meet precisely on the equality branch
\begin{equation}\label{eq_equ_ass}
    2\dot\theta(t)+\mu\theta(t)^2=1.
\end{equation}

The difference between the two conditions is already visible in the convex
case. For the power-law choice
$
    \lambda(t)=\frac{c}{t}$ with $c>0$,
condition \eqref{eq:luoCass} with $\mu=0$ requires $c\le2$. The
corresponding damping coefficient is
\[
    \delta_0(t)
    =
    \lambda(t)
    -
    \frac{\dot\lambda(t)}{\lambda(t)}
    =
    \frac{c+1}{t}.
\]
Hence, the scaled NAG condition covers the range
\[
    \delta_0(t)=\frac{\gamma}{t},
    \qquad
    1<\gamma\le3.
\]
By contrast, Assumption~\ref{ass_coeff}, expressed in terms of
$\theta(t)=t/c$, requires $c\ge2$ and therefore covers
\[
    \delta_0(t)=\frac{\gamma}{t},
    \qquad
    \gamma\ge3.
\]
The two ranges meet at $c=2$, which corresponds to the critical damping
$\gamma=3$. In particular, the condition considered here includes the
supercritical range $\gamma>3$, for which the convex dynamic
\eqref{dy_NAG_C} satisfies the improved estimate
$
    \Phi(x(t))-\Phi^*
    =
    o(t^{-2})
$
under the standard assumptions in \cite{AttouchSiopt16}.

On the equality branch \eqref{eq_equ_ass}, the transformed coefficient
$\lambda=1/\theta$ satisfies
\[
    2\dot\lambda(t)=\mu-\lambda(t)^2.
\]
The singular branch  corresponding to $\gamma_0=0$  in \cite{LuoChen2022,KimYang2023} is
\[
    \lambda(t)
    =
    \sqrt{\mu}
    \frac{e^{\sqrt{\mu}t}+1}{e^{\sqrt{\mu}t}-1}.
\]
Consequently,
\[
    \theta(t)
    =
    \frac{1}{\lambda(t)}
    =
    \frac{1}{\sqrt{\mu}}
    \frac{e^{\sqrt{\mu}t}-1}{e^{\sqrt{\mu}t}+1}
    =
    \frac{1}{\sqrt{\mu}}
    \tanh\left(\frac{\sqrt{\mu}}{2}t\right).
\]
The corresponding damping coefficient is
\[
     \delta_\mu(t)
    =
    \frac{\mu+3\lambda(t)^2}{2\lambda(t)}=
    2\sqrt{\mu}
    \frac{
        e^{2\sqrt{\mu}t}+e^{\sqrt{\mu}t}+1
    }{
        e^{2\sqrt{\mu}t}-1
    }.
\]
  For every fixed $t>0$,
\[
    \delta_\mu(t)
    \to
    \frac{3}{t}
    \qquad
    \text{as }\mu\downarrow0,
\]
whereas, for every fixed $\mu>0$,
\[
    \delta_\mu(t)
    \to
    2\sqrt{\mu}
    \qquad
    \text{as }t\to+\infty.
\]
More precisely,
\[
     \delta_\mu(t)-2\sqrt{\mu}
    =
    2\sqrt{\mu}
    \frac{e^{\sqrt{\mu}t}+2}{e^{2\sqrt{\mu}t}-1}
    =
    \mathcal O\left(e^{-\sqrt{\mu}t}\right).
\]
Therefore, the equality branch  \eqref{eq_equ_ass}  generates the damping transition
\[
    \frac{3}{t}
    \longrightarrow
    2\sqrt{\mu}.
\]
This branch also corresponds to the Kim-Yang unified
dynamic in
\cite[Eq.~(8)]{KimYang2023}.  For this choice of $\theta(t)$, the function $A(t)$ defined in
\eqref{eq_Acont} can be written as
\[    A(t)= 
    \frac{e^{\sqrt{\mu}t_0}}{
        \mu\left(e^{\sqrt{\mu}t_0}+1\right)^2
    }
    \left(
        e^{\frac{\sqrt{\mu}t}{2}}
        -
        e^{-\frac{\sqrt{\mu}t}{2}}
    \right)^2
    =
    \frac{e^{\sqrt{\mu}t_0}}{
        \mu\left(e^{\sqrt{\mu}t_0}+1\right)^2
    }
    e^{\sqrt{\mu}t}
    \left(1-e^{-\sqrt{\mu}t}\right)^2.
\]
Thus $A(t)$ is increasing.  Using the elementary inequality
\[
    e^{r/2}-e^{-r/2}\ge r,
    \qquad r\ge0,
\]
we obtain
\[
    A(t)
    \ge
    \frac{e^{\sqrt{\mu}t_0}}{
        \left(e^{\sqrt{\mu}t_0}+1\right)^2
    }
    t^2.
\]
Moreover, since $t\ge t_0$,
$
    1-e^{-\sqrt{\mu}t}
    \ge
    1-e^{-\sqrt{\mu}t_0},
$
and hence
\[
    A(t)
    \ge
    \frac{
        e^{\sqrt{\mu}t_0}
        \left(
            1-e^{-\sqrt{\mu}t_0}
        \right)^2
    }{
        \mu
        \left(
            e^{\sqrt{\mu}t_0}+1
        \right)^2
    }
    e^{\sqrt{\mu}t}.
\]
Consequently, Theorem~\ref{thm_cont_main} yields
\[
\begin{aligned}
    \Phi(x(t))-\Phi^*
    \le
    \mathcal E(t_0)
    \min\left\{
        \frac{
            \left(
                e^{\sqrt{\mu}t_0}+1
            \right)^2
        }{
            e^{\sqrt{\mu}t_0}t^2
        },
        \frac{
            \mu
            \left(
                e^{\sqrt{\mu}t_0}+1
            \right)^2
        }{
            e^{\sqrt{\mu}t_0}
            \left(
                1-e^{-\sqrt{\mu}t_0}
            \right)^2
        }
        e^{-\sqrt{\mu}t}
    \right\}.
\end{aligned}
\]
Therefore,
\[
    \Phi(x(t))-\Phi^*
    =
    \bigO\left(
        \min\left\{
            \frac{1}{t^2},
            e^{-\sqrt{\mu}t}
        \right\}
    \right).
\]

\begin{remark}
On the equality branch \eqref{eq_equ_ass}, the scaled NAG estimate of
Luo-Chen \cite[Section~3.2]{LuoChen2022}, the Kim-Yang estimate
\cite[Theorem~4.1]{KimYang2023}, and Theorem~\ref{thm_cont_main} yield the
same convergence rate. The polynomial and exponential estimates for the
objective residual follow from two different lower bounds on $A(t)$. In
particular, the bound
$
    \Phi(x(t))-\Phi^*
    =
    \bigO(t^{-2})
$
does not rely on the exponential decay induced by strong convexity. As $\mu\downarrow0$, one has
$e^{-\sqrt{\mu}t}\to1$ for every fixed $t>0$, so the exponential estimate
loses its decay. By contrast, the prefactor in the $\bigO(1/t^2)$ estimate
remains bounded, and
$
    A(t)\to\frac{t^2}{4}.
$
At the same time,
\[
    \delta_\mu(t)
    \to
    \frac{3}{t}
    \qquad
    \text{as }\mu\downarrow0
\]
for every fixed $t>0$. Hence, the equality-based transition dynamics reduce
to the convex NAG flow \eqref{dy_NAG_C} with $\gamma=3$ in the limit $\mu\downarrow0$, while the
$\bigO(t^{-2})$ estimate for the objective residual remains valid.
\end{remark}

Assumption~\ref{ass_coeff} is not restricted to the equality branch. In the
convex case $\mu=0$, it includes both the critical and supercritical damping
laws
\[
    \delta_0(t)=\frac{\gamma}{t},
    \qquad
    \gamma\ge3.
\]
For $\mu>0$, it also permits additional coefficient families that transition
from the convex damping $\gamma/t$ to the strongly convex damping
$2\sqrt{\mu}$. These families are not covered by
assumption \eqref{eq:luoCass} in \cite{LuoChen2022} and allow the duration and shape of the transition between
the two endpoint regimes to be adjusted. This
flexibility is particularly relevant when $\mu$ is small, since the damping
can retain convex-type behavior over a longer initial period before
approaching the strongly convex value $2\sqrt{\mu}$. The corresponding
hyperbolic, exponential, algebraic, and polynomial damping coefficients are
developed in the following subsections.

 \subsection{Classical  and transition damping coefficients}\label{sec23}

The general estimate in Theorem \ref{thm_cont_main} reduces the convergence analysis
of \eqref{dy_main} to the choice of an admissible coefficient function
$\theta(t)$ and the estimation of the function $A(t)$ defined in
\eqref{eq_Acont}. We first recover the classical convex and strongly convex
Nesterov dynamics as endpoint cases and then construct several families of
damping coefficients that transition between them. The transition functions
introduced below vanish at the origin, so the resulting dynamics are
considered on $[t_0,+\infty)$ with $t_0>0$.

\paragraph{Classical convex and strongly convex dynamics.} We first consider the convex case $\mu=0$. Assumption~\ref{ass_coeff} then
reduces to
$
    0<2\dot\theta(t)\le1.
$
The damping coefficient becomes
$
    \delta_0(t)=\frac{1+\dot\theta(t)}{\theta(t)},
$
or, equivalently,
\[
    \dot\theta(t)=\delta_0(t)\theta(t)-1.
\]
Conversely, for a prescribed positive damping coefficient $\delta_0$, define
\[
    p(t):=e^{\int_{t_0}^t\delta_0(s)\,\dd s},
    \qquad
    \Gamma_\delta(t):=p(t)\int_t^{+\infty}\frac{\dd s}{p(s)},
\]
whenever the integral is finite. A direct calculation gives
\[
    \dot\Gamma_\delta(t)=\delta_0(t)\Gamma_\delta(t)-1,
    \qquad
    \delta_0(t)=\frac{1+\dot\Gamma_\delta(t)}{\Gamma_\delta(t)}.
\]
Thus, the choice $\theta(t)=\Gamma_\delta(t)$ recovers the time-dependent
viscosity dynamic considered in
\cite[Proposition~2.1]{AttouchCabot2017}. Moreover,
\[
    0<2\dot\Gamma_\delta(t)\le1
    \quad\Longleftrightarrow\quad
    1<\delta_0(t)\Gamma_\delta(t)\le\frac32.
\]
Consequently, every damping coefficient satisfying this condition gives an
admissible convex choice in the present framework. Since
$A(t)=\theta(t)^2=\Gamma_\delta(t)^2$, Theorem \ref{thm_cont_main} yields
\[
    \Phi(x(t))-\Phi^*=\bigO\left(\Gamma_\delta(t)^{-2}\right),
\]
which recovers the convergence estimate established in
\cite[Corollary~3.4]{AttouchCabot2017}.

In particular, for
\[
    \delta_0(t)=\frac{\gamma}{t},
    \qquad
    \gamma\ge3,
\]
one has
\[
    \theta(t)=\Gamma_\delta(t)=\frac{t}{\gamma-1}.
\]
The dynamic \eqref{dy_main} then reduces to the convex NAG flow
\eqref{dy_NAG_C}. Since
\[
    A(t)=\theta(t)^2=\frac{t^2}{(\gamma-1)^2},
\]
Theorem \ref{thm_cont_main} gives
\[
    \Phi(x(t))-\Phi^*=\bigO(t^{-2}),
\]
recovering the classical convex NAG estimate
\cite{SuJMLR2016,AttouchMP18}. The admissible range $\gamma\ge3$ includes
both the critical choice $\gamma=3$ and the supercritical choices
$\gamma>3$.

We next consider the strongly convex case $\mu>0$. The constant choice
\[
    \theta(t)=\frac{1}{\sqrt{\mu}}
\]
satisfies Assumption~\ref{ass_coeff}, with
\[
    2\dot\theta(t)+\mu\theta(t)^2=1.
\]
The corresponding damping coefficient is
\[
    \delta_\mu(t)=\frac{1+\dot\theta(t)}{\theta(t)}+\mu\theta(t)
    =2\sqrt{\mu}.
\]
Hence, \eqref{dy_main} becomes the classical strongly convex NAG dynamic
\eqref{dy_NAG_SC}. In this case,
\[
    A(t)=\frac{1}{\mu}e^{\sqrt{\mu}(t-t_0)},
\]
and Theorem \ref{thm_cont_main} therefore gives
\[
    \Phi(x(t))-\Phi^*=\bigO\left(e^{-\sqrt{\mu}t}\right).
\]
This recovers the accelerated exponential rate of the classical strongly
convex NAG dynamic
\cite{LuoChen2022,Siegel2019,WilsonJmlr}.

These two endpoint choices determine the limiting behavior required of a
transition family. Writing $\theta_\mu$ when its dependence on $\mu$ needs
to be emphasized, we seek, for every fixed $t>0$,
\[
    \theta_\mu(t)\to\frac{t}{\gamma-1},
    \qquad
    \delta_\mu(t)\to\frac{\gamma}{t}
    \qquad
    \text{as }\mu\downarrow0,
\]
where $\gamma\ge3$. For every fixed $\mu>0$, the strongly convex asymptotic
regime requires
\[
    \theta_\mu(t)\to\frac{1}{\sqrt{\mu}},
    \qquad
    \delta_\mu(t)\to2\sqrt{\mu}
    \qquad
    \text{as }t\to+\infty.
\]
The following constructions realize these two endpoint behaviors within
admissible families of damping coefficients.

\paragraph{A hyperbolic family of transition damping coefficients.}
We first construct a hyperbolic family that extends the equality-branch
damping transition associated with the critical convex damping $3/t$ to
every convex damping $\gamma/t$ with $\gamma\ge3$. For completeness, we use the standard notation
\[
    \sinh r:=\frac{e^r-e^{-r}}{2},
    \qquad
    \cosh r:=\frac{e^r+e^{-r}}{2},
    \qquad
    \tanh r:=\frac{\sinh r}{\cosh r},
\]
and
\[
    \sech r:=\frac{1}{\cosh r},
    \qquad
    \coth r:=\frac{\cosh r}{\sinh r}
    \quad (r\ne0).
\]
We also use the identities
\[
    \frac{\dd}{\dd r}\tanh r=\sech^2r,
    \qquad
    \sech^2r+\tanh^2r=1.
\]
These identities allow us to verify the admissibility condition explicitly
and to derive the damping coefficient generated by the following
hyperbolic choice of $\theta(t)$.

\begin{theorem}\label{cor:hyp_transition}
Let $\mu>0$, $\gamma\ge3$, and
$
    \kappa:=\frac{1}{\gamma-1}\in(0,1/2].
$
Define
\[
    \theta(t)=\frac{1}{\sqrt{\mu}}
    \tanh\left(\kappa\sqrt{\mu}\,t\right).
\]
Then $\theta(t)$ satisfies Assumption~\ref{ass_coeff}, and the corresponding
damping coefficient is
\begin{equation}\label{eq_hyp_damp}
    \delta_\mu(t)
    =
    \sqrt{\mu}
    \left[
        (1+\kappa)\coth\left(\kappa\sqrt{\mu}\,t\right)
        +(1-\kappa)\tanh\left(\kappa\sqrt{\mu}\,t\right)
    \right].
\end{equation}
For every fixed $t>0$,
\[
    \delta_\mu(t)\to\frac{\gamma}{t}
    \qquad
    \text{as }\mu\downarrow0,
\]
whereas, for every fixed $\mu>0$,
\[
    \delta_\mu(t)\to2\sqrt{\mu}
    \qquad
    \text{as }t\to+\infty.
\]
Moreover, every global solution of \eqref{dy_main} satisfies, for every
$t\ge t_0>0$,
\[
    \Phi(x(t))-\Phi^*
    \le
    \mathcal E(t_0)
    \min\Bigg\{
        \frac{
            \cosh^{\gamma-1}(\kappa\sqrt{\mu}\,t_0)
        }{
            \kappa^2t^2
        },
        \frac{
            2^{\gamma-1}\mu
            \cosh^{\gamma-1}(\kappa\sqrt{\mu}\,t_0)
        }{
            \left(1-e^{-2\kappa\sqrt{\mu}\,t_0}\right)^2
        }
        e^{-\sqrt{\mu}\,t}
    \Bigg\}.
\]
In particular,
\[
    \Phi(x(t))-\Phi^*
    =
    \bigO\left(
        \min\left\{
            \frac{1}{t^2},
            e^{-\sqrt{\mu}t}
        \right\}
    \right).
\]
\end{theorem}

\begin{proof}
Differentiating $\theta(t)$ gives
\[
    \dot\theta(t)
    =
    \kappa\left(
        1-\tanh^2(\kappa\sqrt{\mu}\,t)
    \right).
\]
Consequently,
\[
\begin{aligned}
    2\dot\theta(t)+\mu\theta(t)^2
    &=
    2\kappa\left(
        1-\tanh^2(\kappa\sqrt{\mu}\,t)
    \right)
    +
    \tanh^2(\kappa\sqrt{\mu}\,t)\\
    &=
    2\kappa
    +
    (1-2\kappa)
    \tanh^2(\kappa\sqrt{\mu}\,t)
    \le1.
\end{aligned}
\]
The strict lower bound and the monotonicity of $\theta(t)$ follow immediately.
Thus, Assumption~\ref{ass_coeff} holds.

Substituting $\theta(t)$ and $\dot\theta(t)$ into the definition of
$\delta_\mu(t)$ gives \eqref{eq_hyp_damp}. Set
$
    r_\mu:=\kappa\sqrt{\mu}\,t.
$
For every fixed $t>0$, one has $r_\mu\to0$ as $\mu\downarrow0$. By
l'H\^{o}pital's rule,
\[
    \lim_{r\downarrow0}r\coth r
    =
    \lim_{r\downarrow0}\frac{r\cosh r}{\sinh r}
    =
    \lim_{r\downarrow0}
    \frac{\cosh r+r\sinh r}{\cosh r}
    =
    1.
\]
Moreover, since $|\tanh r|\le1$,
\[
    \sqrt{\mu}\tanh(r_\mu)\to0
    \qquad
    \text{as }\mu\downarrow0.
\]
Therefore,
\[
\begin{aligned}
    \delta_\mu(t)
    &=
    \sqrt{\mu}
    \left[
        (1+\kappa)\coth(r_\mu)
        +(1-\kappa)\tanh(r_\mu)
    \right]\\
    &=
    \frac{1+\kappa}{\kappa t}
    r_\mu\coth(r_\mu)
    +
    (1-\kappa)\sqrt{\mu}\tanh(r_\mu)
    \to
    \frac{1+\kappa}{\kappa t}.
\end{aligned}
\]
Since $\kappa=1/(\gamma-1)$, it follows that
\[
    \delta_\mu(t)
    \to
    \left(1+\frac{1}{\kappa}\right)\frac{1}{t}
    =
    \frac{\gamma}{t}
    \qquad
    \text{as }\mu\downarrow0.
\]

For every fixed $\mu>0$, one has
$
    r_\mu=\kappa\sqrt{\mu}\,t\to+\infty
$
as $t\to+\infty$. Since
\[
    \tanh r\to1
    \qquad\text{and}\qquad
    \coth r\to1
    \qquad
    \text{as }r\to+\infty,
\]
we obtain
\[
    \delta_\mu(t)
    \to
    \sqrt{\mu}
    \bigl((1+\kappa)+(1-\kappa)\bigr)
    =
    2\sqrt{\mu}
    \qquad
    \text{as }t\to+\infty.
\]

For the function $A(t)$ defined in \eqref{eq_Acont}, we have
\[
    \int_{t_0}^t\mu\theta(s)\,\dd s
    =
    \sqrt{\mu}
    \int_{t_0}^t
    \tanh\left(\kappa\sqrt{\mu}\,s\right)\,\dd s=
    \frac{1}{\kappa}
    \ln
    \frac{
        \cosh(\kappa\sqrt{\mu}\,t)
    }{
        \cosh(\kappa\sqrt{\mu}\,t_0)
    }.
\]
Since $1/\kappa=\gamma-1$, it follows that
\[
    A(t)
    =
    \frac{
        \sinh^2(\kappa\sqrt{\mu}\,t)
        \cosh^{\gamma-3}(\kappa\sqrt{\mu}\,t)
    }{
        \mu
        \cosh^{\gamma-1}(\kappa\sqrt{\mu}\,t_0)
    }.
\]
Since $\gamma\ge 1+1/\kappa\ge3$, $\sinh r\ge r$, and $\cosh r\ge1$ for $r\ge0$,
\[
    A(t)
    \ge
    \frac{
        \kappa^2t^2
    }{
        \cosh^{\gamma-1}(\kappa\sqrt{\mu}\,t_0)
    }.
\]
On the other hand,
\[
\sinh^2(\kappa\sqrt{\mu}\,t)
    \cosh^{\gamma-3}(\kappa\sqrt{\mu}\,t)
    =
    \frac{
        e^{(\gamma-1)\kappa\sqrt{\mu}\,t}
    }{
        2^{\gamma-1}
    }
    \left(
        1-e^{-2\kappa\sqrt{\mu}\,t}
    \right)^2
    \left(
        1+e^{-2\kappa\sqrt{\mu}\,t}
    \right)^{\gamma-3}.
\]
Since $(\gamma-1)\kappa=1$, for every $t\ge t_0$,
\[
    A(t)
    \ge
    \frac{
        \left(
            1-e^{-2\kappa\sqrt{\mu}\,t_0}
        \right)^2
    }{
        2^{\gamma-1}\mu
        \cosh^{\gamma-1}(\kappa\sqrt{\mu}\,t_0)
    }
    e^{\sqrt{\mu}\,t}.
\]
The conclusion now follows from Theorem \ref{thm_cont_main}.
\end{proof}

\begin{remark}
When $\gamma=3$, equivalently $\kappa=1/2$, the hyperbolic family reduces to
the equality-branch damping transition associated with the Luo-Chen scaled
NAG dynamic \cite[Section~3.2]{LuoChen2022} and the Kim-Yang unified
dynamic \cite[Theorem~4.1]{KimYang2023}. When $\gamma>3$, it produces a
damping transition from the supercritical convex damping $\gamma/t$ to the
strongly convex damping $2\sqrt{\mu}$. Thus, the construction extends the
critical equality-branch transition to the full accelerated convex range
$\gamma\ge3$, while preserving the $\bigO(t^{-2})$ and
$\bigO(e^{-\sqrt{\mu}t})$ estimates for the objective residual.
The coefficient
$
    \cosh^{\gamma-1}(\kappa\sqrt{\mu}\,t_0)/\kappa^2
$
in the polynomial estimate converges to $1/\kappa^2$ as
$\mu\downarrow0$. Equivalently, the corresponding quadratic lower bound on
$A(t)$ remains nondegenerate in the convex limit.
\end{remark}

\begin{remark}
The $\bigO(t^{-2})$ and
$\bigO(e^{-\sqrt{\mu}t})$ estimates describe different stages of the
damping transition. This distinction is particularly relevant when the
strong convexity parameter $\mu$ is small. In that case, the damping
coefficient remains close to the convex damping $\gamma/t$ over a relatively
long initial interval, and the exponential factor
$e^{-\sqrt{\mu}t}$ decreases slowly. After the constants in the two
estimates are taken into account, the $\bigO(t^{-2})$ estimate may therefore
provide the sharper bound during the initial and intermediate stages, in
agreement with the convex NAG-type behavior of the dynamics.
For every fixed $\mu>0$, the damping coefficient eventually approaches the
strongly convex value $2\sqrt{\mu}$, and the exponential estimate becomes
sharper for sufficiently large $t$. Thus, when $\mu$ is small, the dynamics
may exhibit a prolonged convex-type phase before entering the strongly
convex asymptotic regime. The location of this transition depends on $\mu$
and on the constants appearing in the two estimates.
\end{remark}

 \paragraph{A general construction of transition damping coefficients.}

The hyperbolic form is not essential. We next give general conditions on a
transition function that produce the same endpoint damping coefficients and
convergence rates. Let
\[
    \theta(t)=\frac{1}{\sqrt{\mu}}S(\sqrt{\mu}t),
\]
where $S:[0,+\infty)\to[0,1)$ is continuously differentiable, increasing,
and concave, with $S(0)=0$. Under this parametrization,
Assumption~\ref{ass_coeff} becomes
\begin{equation}\label{eq_S_cond}
    0<2S'(r)+S(r)^2\le1,
    \qquad
    \forall r\ge 0.
\end{equation}
The damping coefficient and the function $A(t)$ are given by
\begin{equation}\label{eq_S_sca}
\begin{aligned}
    \delta_\mu(t)
    &=
    \sqrt{\mu}
    \left(
        \frac{1+S'(\sqrt{\mu}t)}{S(\sqrt{\mu}t)}
        +
        S(\sqrt{\mu}t)
    \right),\\
    A(t)
    &=
    \frac{1}{\mu}
    S(\sqrt{\mu}t)^2
    e^{
        \int_{\sqrt{\mu}t_0}^{\sqrt{\mu}t}
        S(r)\,\dd r
    }.
\end{aligned}
\end{equation}

Assume that
\begin{equation}\label{eq_S_asy}
\begin{gathered}
    S'(0)=\kappa \quad\text{with}\ \kappa\in(0,\frac{1}{2}],\\
    S(r)\to1,
    \qquad
    S'(r)\to0
    \qquad
    \text{as }r\to+\infty,\\
    \int_0^{+\infty}\bigl(1-S(r)\bigr)\,\dd r<+\infty.
\end{gathered}
\end{equation}
The behavior of $S$ near the origin determines the limiting convex damping,
whereas its behavior at infinity determines the strongly convex asymptotic
regime. The integrability condition in \eqref{eq_S_asy} ensures that the
approach of $S$ to its limiting value is sufficiently fast to yield an
exponential convergence estimate. The following theorem makes these
properties precise.

\begin{theorem}\label{thm:general_transition}
Let $\mu>0$, and let $S:[0,+\infty)\to[0,1)$ be continuously
differentiable, increasing, and concave, with $S(0)=0$. Suppose that
\eqref{eq_S_cond} and \eqref{eq_S_asy} hold. For the coefficient choice
\[
    \theta(t)=\frac{1}{\sqrt{\mu}}S(\sqrt{\mu}t),
\]
the corresponding damping coefficient satisfies, for every fixed $t>0$,
\[
    \delta_\mu(t)\to\frac{\gamma}{t}
    \qquad
    \text{as }\mu\downarrow0,
    \qquad
    \gamma:=1+\frac{1}{\kappa}\ge3,
\]
and, for every fixed $\mu>0$,
\[
    \delta_\mu(t)\to2\sqrt{\mu}
    \qquad
    \text{as }t\to+\infty.
\]
Moreover, the corresponding trajectory  of dynamic
\eqref{dy_main} satisfies, for every $t\ge t_0>0$, 
\[    \Phi(x(t))-\Phi^*
    \le
    \mathcal E(t_0)
    \min\left\{
        \frac{1}{C_{\mu,t_0}t^2},
        \frac{
            \mu e^{\sqrt{\mu}t_0+\xi}
        }{
            S(\sqrt{\mu}t_0)^2
        }
        e^{-\sqrt{\mu}t}
    \right\},
\]
where $C_{\mu,t_0}$ is defined in
\eqref{eq_Squad_cons} and $\xi=\int_0^{+\infty}\bigl(1-S(r)\bigr)\,\dd r<+\infty$. In particular,
\[
    \Phi(x(t))-\Phi^*
    =
    \bigO\left(
        \min\left\{
            \frac{1}{t^2},
            e^{-\sqrt{\mu}t}
        \right\}
    \right).
\]
\end{theorem}

\begin{proof}
Since $S$ is continuously differentiable, $S(0)=0$, and
$S'(0)=\kappa$, we have
\[
    S(r)=\kappa r+o(r),
    \qquad
    S'(r)=\kappa+o(1)
    \qquad
    \text{as }r\downarrow0.
\]
It follows from \eqref{eq_S_sca} that, for every fixed $t>0$,
\[
    \delta_\mu(t)
    \to
    \left(
        1+\frac{1}{\kappa}
    \right)\frac{1}{t}
    =
    \frac{\gamma}{t}
    \qquad
    \text{as }\mu\downarrow0.
\]
The assumptions $S(r)\to1$ and $S'(r)\to0$ similarly give
\[
    \delta_\mu(t)\to2\sqrt{\mu}
    \qquad
    \text{as }t\to+\infty.
\]

We next derive lower bounds for $A(t)$. Set
\[
    \xi:=\int_0^{+\infty}\bigl(1-S(r)\bigr)\,\dd r<+\infty.
\]
For every $t\ge t_0$,
\[
\begin{aligned}
    \int_{\sqrt{\mu}t_0}^{\sqrt{\mu}t}S(r)\,\dd r
    &=
    \sqrt{\mu}(t-t_0)
    -
    \int_{\sqrt{\mu}t_0}^{\sqrt{\mu}t}
    \bigl(1-S(r)\bigr)\,\dd r\\
    &\ge
    \sqrt{\mu}(t-t_0)-\xi.
\end{aligned}
\]
Since $S$ is increasing,
$
    S(\sqrt{\mu}t)\ge S(\sqrt{\mu}t_0),
$
and hence
\begin{equation}\label{eq_S_exp}
    A(t)
    \ge
    \frac{S(\sqrt{\mu}t_0)^2}{\mu}
    e^{-\sqrt{\mu}t_0-\xi}
    e^{\sqrt{\mu}t},
    \qquad
    t\ge t_0.
\end{equation}

To obtain a quadratic lower bound, observe that the concavity of $S$ and
$S(0)=0$ imply
\[
    S(r)
    \ge
    (1-r)S(0)+rS(1)
    =
    rS(1),
    \qquad
    0\le r\le1.
\]
The monotonicity of $S$ also gives
\[
    S(r)\ge S(1),
    \qquad
    r\ge1.
\]
Therefore,
\begin{equation}\label{eq_S_low}
    S(r)\ge S(1)\min\{r,1\},
    \qquad
    r\ge0.
\end{equation}

Suppose first that $\sqrt{\mu}\,t_0\le1$. If
$\sqrt{\mu}\,t\le1$, then the integral in \eqref{eq_S_sca} is
nonnegative. Hence, \eqref{eq_S_low} gives
\[
    A(t)\ge\frac{1}{\mu}S(\sqrt{\mu}t)^2\ge S(1)^2t^2.
\]
If $\sqrt{\mu}\,t\ge1$, then $S(\sqrt{\mu}t)\ge S(1)$, 
\[
    \int_{\sqrt{\mu}t_0}^{\sqrt{\mu}t}S(r)\,\dd r
    \ge
    \int_1^{\sqrt{\mu}t}S(r)\,\dd r
    \ge
    S(1)\bigl(\sqrt{\mu}t-1\bigr).
\]
Consequently,  we obtain
\[
    A(t)
    \ge
    \frac{S(1)^2}{\mu}
    e^{S(1)(\sqrt{\mu}t-1)}
    =
    S(1)^2
    \frac{
        e^{S(1)(\sqrt{\mu}t-1)}
    }{
        (\sqrt{\mu}t)^2
    }
    t^2.
\]
For every $y\ge1$, the function
$
    y\mapsto e^{S(1)(y-1)}/y^2
$
attains its minimum at $y=2/S(1)$. Therefore,
\[
    \frac{e^{S(1)(y-1)}}{y^2}
    \ge
    \frac{S(1)^2}{4}e^{2-S(1)},
    \qquad
    y\ge1.
\]
Taking $y=\sqrt{\mu}t$ yields
\[
    A(t)
    \ge
    \frac{S(1)^4}{4}e^{2-S(1)}t^2
    \qquad
    \text{whenever }\sqrt{\mu}t\ge1.
\]
Consequently, if $\sqrt{\mu}\,t_0\le1$, then
\begin{equation*}\label{eq_Aest1}
	 A(t)
    \ge
    \min\left\{
        S(1)^2,
        \frac{S(1)^4}{4}e^{2-S(1)}
    \right\}t^2,
    \qquad
    t\ge t_0.
\end{equation*}

Suppose next that $\sqrt{\mu}\,t_0>1$. For every $t\ge t_0$, the
monotonicity of $S$ gives
\[
    S(\sqrt{\mu}t)\ge S(\sqrt{\mu}t_0),
    \qquad
    \int_{\sqrt{\mu}t_0}^{\sqrt{\mu}t}S(r)\,\dd r
    \ge
    S(\sqrt{\mu}t_0)\sqrt{\mu}(t-t_0).
\]
Consequently,
\[
    A(t)\ge
    \frac{S(\sqrt{\mu}t_0)^2}{\mu}
    e^{S(\sqrt{\mu}t_0)\sqrt{\mu}(t-t_0)}.\]
Arguing as in the preceding case, with $S(1)$ replaced by $S(\sqrt{\mu}t_0),$
we obtain
\[
    \frac{
        e^{
            S(\sqrt{\mu}t_0)
            (\sqrt{\mu}t-\sqrt{\mu}t_0)
        }
    }{
        (\sqrt{\mu}t)^2
    }
    \ge
    \frac{S(\sqrt{\mu}t_0)^2}{4}
    e^{
        2-S(\sqrt{\mu}t_0)\sqrt{\mu}t_0
    }.
\]
Therefore,
\[
    A(t)
    \ge
    \frac{S(\sqrt{\mu}t_0)^4}{4}
    e^{
        2-S(\sqrt{\mu}t_0)\sqrt{\mu}t_0
    }
    t^2,
    \qquad
    t\ge t_0.
\]

Combining these  two cases, define
\begin{equation}\label{eq_Squad_cons}
    C_{\mu,t_0}
    :=
    \begin{cases}
    \displaystyle
    \min\left\{
        S(1)^2,
        \frac{S(1)^4}{4}e^{2-S(1)}
    \right\},
    & \sqrt{\mu}\,t_0\le1,\\[1.2em]
    \displaystyle
    \frac{S(\sqrt{\mu}t_0)^4}{4}
    e^{2-S(\sqrt{\mu}t_0)\sqrt{\mu}t_0},
    & \sqrt{\mu}\,t_0>1.
    \end{cases}
\end{equation}
Then
\begin{equation}\label{eq_S_low_quad}
    A(t)\ge C_{\mu,t_0}t^2,
    \qquad
    t\ge t_0.
\end{equation}
The asserted estimate follows from \eqref{eq_S_exp},
\eqref{eq_S_low_quad}, and Theorem \ref{thm_cont_main}.
\end{proof}

Thus, the two-regime convergence behavior established for the hyperbolic
family extends to every transition function satisfying
\eqref{eq_S_cond} and \eqref{eq_S_asy}. Moreover, for every fixed
$t_0>0$ and $0<\mu\le t_0^{-2}$, the constant $C_{\mu,t_0}$ in
\eqref{eq_Squad_cons} is positive and independent of $\mu$, so the
quadratic lower bound on $A(t)$ remains nondegenerate as
$\mu\downarrow0$.

We now give three non-hyperbolic transition functions satisfying
\eqref{eq_S_cond} and \eqref{eq_S_asy}. They generate the same endpoint
damping coefficients and convergence rates, but differ in the rate at which
the corresponding damping coefficients approach $2\sqrt{\mu}$.

\begin{corollary}\label{cor_disT}
Let $0<\kappa\le1/2$. The following transition functions satisfy
\eqref{eq_S_cond} and \eqref{eq_S_asy}:
\begin{equation}\label{eq_three_choices}
\begin{aligned}
    &\text{\rm Exponential:}\qquad
    S_{\rm exp}(r)
    =
    1-e^{-\kappa r},
    &&r\ge0,\\
    &\text{\rm Algebraic:}\qquad
    S_{\rm alg}(r)
    =
    \frac{\kappa r}{\sqrt{1+\kappa^2r^2}},
    &&r\ge0,\\
    &\text{\rm Polynomial:}\qquad
    S_{\rm poly}(r)
    =
    1-\left(1+\frac{\kappa r}{p}\right)^{-p},
    &&r\ge0,\quad p>1.
\end{aligned}
\end{equation}
Consequently, for each of these choices, the corresponding damping
coefficient satisfies
\[
    \delta_\mu(t)\to\frac{\gamma}{t}
    \quad\text{as }\mu\downarrow0,
    \qquad
    \delta_\mu(t)\to2\sqrt{\mu}
    \quad\text{as }t\to+\infty,
\]
where
$
    \gamma=1+\frac{1}{\kappa}\ge3,
$
and every global solution of \eqref{dy_main} satisfies
\[
    \Phi(x(t))-\Phi^*
    =
    \bigO\left(
        \min\left\{
            \frac{1}{t^2},
            e^{-\sqrt{\mu}t}
        \right\}
    \right).
\]
\end{corollary}

\begin{proof}
The first derivatives are
\[
\begin{aligned}
    S_{\rm exp}'(r)
    &=
    \kappa\bigl(1-S_{\rm exp}(r)\bigr),\\
    S_{\rm alg}'(r)
    &=
    \kappa\bigl(1-S_{\rm alg}(r)^2\bigr)^{3/2},\\
    S_{\rm poly}'(r)
    &=
    \kappa\bigl(1-S_{\rm poly}(r)\bigr)^{1+1/p}.
\end{aligned}
\]
Their second derivatives are
\[
\begin{aligned}
    S_{\rm exp}''(r)
    &=
    -\kappa^2\bigl(1-S_{\rm exp}(r)\bigr),\\
    S_{\rm alg}''(r)
    &=
    -3\kappa^2S_{\rm alg}(r)
    \bigl(1-S_{\rm alg}(r)^2\bigr)^2,\\
    S_{\rm poly}''(r)
    &=
    -\kappa^2\left(1+\frac1p\right)
    \bigl(1-S_{\rm poly}(r)\bigr)^{1+2/p}.
\end{aligned}
\]
Thus, all three transition functions are increasing and concave. They also
satisfy
\[
    S(0)=0,
    \qquad
    S'(0)=\kappa,
    \qquad
    S(r)\to1,
    \qquad
    S'(r)\to0
    \quad
    \text{as }r\to+\infty.
\]

Moreover,
\[
\begin{aligned}
    2S_{\rm exp}'+S_{\rm exp}^2
    &=
    1-(1-S_{\rm exp})
    (1+S_{\rm exp}-2\kappa),\\
    2S_{\rm alg}'+S_{\rm alg}^2
    &=
    1-(1-S_{\rm alg}^2)
    \left(
        1-2\kappa\sqrt{1-S_{\rm alg}^2}
    \right),\\
    2S_{\rm poly}'+S_{\rm poly}^2
    &=
    1-(1-S_{\rm poly})
    \left(
        1+S_{\rm poly}
        -2\kappa(1-S_{\rm poly})^{1/p}
    \right).
\end{aligned}
\]
Since $0<\kappa\le1/2$, $0\le S(r)<1$, and $S'(r)>0$, all three
quantities belong to $(0,1]$. Hence, \eqref{eq_S_cond} holds.

Finally,
\[
\begin{aligned}
    \int_0^{+\infty}
    \bigl(1-S_{\rm exp}(r)\bigr)\,\dd r
    &=
    \frac1\kappa,\\
    \int_0^{+\infty}
    \bigl(1-S_{\rm alg}(r)\bigr)\,\dd r
    &=
    \frac1\kappa,\\
    \int_0^{+\infty}
    \bigl(1-S_{\rm poly}(r)\bigr)\,\dd r
    &=
    \frac{p}{\kappa(p-1)}.
\end{aligned}
\]
Thus, \eqref{eq_S_asy} holds, and the conclusion follows from
Theorem \ref{thm:general_transition}.
\end{proof}

To compare the different constructions, we also include the hyperbolic
transition function
\[
    S_{\rm hyp}(r):=\tanh(\kappa r).
\]
Together with the exponential, algebraic, and polynomial choices in
\eqref{eq_three_choices}, these four transition functions
generate damping coefficients satisfying the two endpoint requirements
\[
\begin{aligned}
    \delta_\mu(t)&\to\frac{\gamma}{t}
    &&\text{as }\mu\downarrow0,\qquad \gamma=1+\frac1\kappa\ge3,\\
    \delta_\mu(t)&\to2\sqrt{\mu}
    &&\text{as }t\to+\infty.
\end{aligned}
\]
They therefore provide different admissible realizations of the same
transition from convex to strongly convex damping.

For a general transition function $S$, with
$r=\sqrt{\mu}t$, one has
\[
    \frac{\delta_\mu(t)}{\sqrt{\mu}}-2
    =
    \frac{S'(r)+(1-S(r))^2}{S(r)}.
\]
Then, the behavior near the strongly convex endpoint is,
\begin{equation}\label{eq_toSC}
	\begin{aligned}
    \delta_\mu^{\rm hyp}(t)-2\sqrt{\mu}
    &\sim
    4\kappa\sqrt{\mu}\,e^{-2\kappa r},\\
    \delta_\mu^{\rm exp}(t)-2\sqrt{\mu}
    &\sim
    \kappa\sqrt{\mu}\,e^{-\kappa r},\\
    \delta_\mu^{\rm alg}(t)-2\sqrt{\mu}
    &\sim
    \frac{\sqrt{\mu}}{\kappa^2r^3},\\
    \delta_\mu^{\rm poly}(t)-2\sqrt{\mu}
    &\sim
    \frac{p^{p+1}\sqrt{\mu}}{\kappa^p r^{p+1}}
    \qquad
    \text{as }r\to+\infty.
\end{aligned}
\end{equation}
Thus, the hyperbolic choice approaches the strongly convex damping most
rapidly among the four constructions considered here. The exponential,
algebraic, and polynomial choices provide more gradual transitions, with the
parameter $p$ controlling the asymptotic transition rate within the
polynomial family. Moreover,
\[
    S_{\rm poly}(r)\to S_{\rm exp}(r)
    \qquad
    \text{as }p\to+\infty
\]
for every fixed $r\ge0$.

Near the convex endpoint, the four transition functions satisfy
\[
\begin{aligned}
    S_{\rm hyp}(r)
    &=
    \kappa r-\frac{\kappa^3}{3}r^3+\bigO(r^5),\\
    S_{\rm exp}(r)
    &=
    \kappa r-\frac{\kappa^2}{2}r^2+\bigO(r^3),\\
    S_{\rm alg}(r)
    &=
    \kappa r-\frac{\kappa^3}{2}r^3+\bigO(r^5),\\
    S_{\rm poly}(r)
    &=
    \kappa r
    -
    \frac{p+1}{2p}\kappa^2r^2
    +
    \bigO(r^3).
\end{aligned}
\]
In particular, all four choices have the same leading-order convex behavior
$S(r)\sim\kappa r$. The hyperbolic and algebraic choices deviate from this
linear behavior only at cubic order, whereas the exponential and polynomial
choices deviate at quadratic order. These different local and asymptotic
behaviors allow the duration and shape of the transition between the two
damping regimes to be adjusted.

In the critical case $\kappa=1/2$, the hyperbolic choice reduces to the
equality-branch transition associated with the Kim-Yang dynamic
\cite{KimYang2023}. Among the four choices considered here, from \eqref{eq_toSC}, this transition
approaches the strongly convex damping most rapidly, whereas the
exponential, algebraic, and polynomial choices provide more gradual
transitions. In the final numerical section, we compare these transition
families with the classical convex and strongly convex flows and with the
equality-branch transition. The experiments will show that, on the tested
problems with a small strong convexity parameter, suitable gradual
transitions can provide more favorable transient objective decrease while
retaining the accelerated strongly convex asymptotic rate, with the
improvement becoming particularly pronounced for very small values of
$\mu$.

 \section{Discrete accelerated forward-backward schemes}
\label{sec_discrete}

In this section, we construct discrete accelerated forward-backward
counterparts of the continuous-time transition framework. We consider the
composite optimization problem
\[
    \min_{x\in\cH}\Phi(x):=f(x)+g(x),
\]
where $\cH$ is a real Hilbert space,
$f:\cH\to\mathbb R$ is continuously Fr\'echet differentiable and
$\mu$-strongly convex for some $\mu\ge0$, and $\nabla f$ is
$L$-Lipschitz continuous for some $L>0$. The function
$g:\cH\to(-\infty,+\infty]$ is assumed to be proper, lower
semicontinuous, and convex. We also assume that
$\operatorname{argmin}\Phi\ne\emptyset$. The strong convexity and
Lipschitz continuity of $\nabla f$ imply
$
    0\le\mu\le L.
$ Throughout this section, we can assume $\mu< L$, otherwise $f$ is quadratic and the problem is trivial.

We construct discrete composite counterparts of the second-order dynamic
\eqref{dy_main} through its two-state representation. More precisely, setting
$\beta(t)=\theta(t)$ in \eqref{dy_coup} gives the two-state formulation
corresponding to \eqref{dy_main}. Replacing $\nabla\Phi$ by the composite
operator $\nabla f+\partial g$ then leads to
\begin{equation}\label{eq_discrete_dynamics}
\begin{aligned}
    \dot x(t)
    &=
    \frac{1}{\theta(t)}
    \bigl(z(t)-x(t)\bigr),\\
    \dot z(t)
    &\in
    -\mu\theta(t)\bigl(z(t)-x(t)\bigr)
    -\theta(t)
    \bigl(\nabla f(x(t))+\partial g(x(t))\bigr).
\end{aligned}
\end{equation}
The first equation couples the primal state $x$ with the auxiliary state
$z$, while the second combines the strong-convexity feedback with the
composite first-order operator. Rather than discretizing the second-order
dynamics directly, we discretize the equivalent two-state formulation
\eqref{eq_discrete_dynamics}. This makes it possible to evaluate the smooth
and nonsmooth terms at suitable discrete points and thereby obtain
implementable accelerated forward-backward schemes.

 For $\lambda>0$, the proximal mapping of $g$ is defined by
\[
    \prox_{\lambda g}(u)
    :=
    \argmin_{v\in\cH}
    \left\{
        g(v)+\frac{1}{2\lambda}\norm{v-u}^{2}
    \right\}.
\]
Equivalently,
\[
    v=\prox_{\lambda g}(u)
    \quad\Longleftrightarrow\quad
    \frac{u-v}{\lambda}\in\partial g(v).
\]

We construct two discretizations of \eqref{eq_discrete_dynamics}. The first
eliminates the auxiliary state and produces an accelerated
forward-backward iteration involving only the sequence $\{x_k\}$. The
second retains $\{z_k\}$ and gives an explicit two-sequence representation.
Although their implementations differ, both schemes are generated by the
same inertial coefficient.

 \subsection{One-sequence and two-sequence discretizations}
\label{subsec_discrete_schemes}

Throughout this part, let $0<s\le 1/L$. Then
$
    0\le s\mu\le \frac{\mu}{L}<1,
$
so the factors $1-s\mu$ and $1-\sqrt{s\mu}$ appearing in the coefficient
formulas below are strictly positive. Let $\{\theta_k\}_{k\ge1}$ be a
sequence of positive coefficients. Both discretizations use the step size
$\sqrt{s}$ and are generated by the same sequence $\{\theta_k\}$. For consistency between the one-sequence and two-sequence formulations, we
initialize $x_0=x_1=z_1\in\operatorname{dom} g$.

\paragraph{One-sequence discretization.}

Using $\sqrt{s}$ as the step size in both equations, we discretize
\eqref{eq_discrete_dynamics} as
\begin{equation}\label{eq_one_system}
\begin{aligned}
    \frac{x_{k+1}-x_k}{\sqrt{s}}
    &=
    \frac{1}{\theta_{k+1}}
    \bigl(z_{k+1}-x_k\bigr),\\
    \frac{z_{k+1}-z_k}{\sqrt{s}}
    &\in
    -\mu\theta_{k+1}(z_k-y_k)
    -\theta_{k+1}
    \bigl(\nabla f(y_k)+\partial g(x_{k+1})\bigr),
\end{aligned}
\end{equation}
where
\[
    y_k:=x_k+\beta_k(x_k-x_{k-1})
\]
is the extrapolated point. The smooth gradient is evaluated explicitly at
$y_k$, whereas the nonsmooth term is evaluated implicitly at the new
iterate $x_{k+1}$. These evaluation points are chosen so that eliminating
the auxiliary state produces the standard proximal-gradient step
\begin{equation}\label{eq_one_prox}
    x_{k+1}
    =
    \prox_{sg}\bigl(y_k-s\nabla f(y_k)\bigr).
\end{equation}
By the characterization of the proximal mapping, \eqref{eq_one_prox} is
equivalent to
\[
    \frac{y_k-x_{k+1}}{s}
    \in
    \nabla f(y_k)+\partial g(x_{k+1}).
\]

We next use the first equation of \eqref{eq_one_system} to express $z_k$
in terms of $x_k$ and $x_{k-1}$.
The first equation of \eqref{eq_one_system} gives
\[
    z_{k+1}
    =
    x_k+
    \frac{\theta_{k+1}}{\sqrt{s}}
    (x_{k+1}-x_k).
\]
Applying the same relation at the preceding iteration yields
\[
    z_k
    =
    x_{k-1}
    +
    \frac{\theta_k}{\sqrt{s}}
    (x_k-x_{k-1})
    =
    x_k+
    \frac{\theta_k-\sqrt{s}}{\sqrt{s}}
    (x_k-x_{k-1}).
\]
Therefore,
\[
    z_k-y_k
    =
    \left(
        \frac{\theta_k-\sqrt{s}}{\sqrt{s}}-\beta_k
    \right)
    (x_k-x_{k-1}),
\]
and
\[
    \frac{z_{k+1}-z_k}{\sqrt{s}}
    =
    \frac{\theta_{k+1}}{s}(x_{k+1}-x_k)
    -
    \frac{\theta_k-\sqrt{s}}{s}
    (x_k-x_{k-1}).
\]
Substituting these identities into the second equation of
\eqref{eq_one_system} and multiplying by $s/\theta_{k+1}$ gives
\[
\begin{aligned}
    x_{k+1}-x_k
    -
    \frac{\theta_k-\sqrt{s}}{\theta_{k+1}}
    (x_k-x_{k-1})
    \in{}&
    -s\mu
    \left(
        \frac{\theta_k-\sqrt{s}}{\sqrt{s}}-\beta_k
    \right)
    (x_k-x_{k-1})\\
    &-
    s\bigl(\nabla f(y_k)+\partial g(x_{k+1})\bigr).
\end{aligned}
\]
To obtain the optimality condition associated with
\eqref{eq_one_prox}, the coefficient of $x_k-x_{k-1}$ must agree with the
definition of $y_k$. This requires
\[
    \beta_k
    =
    \frac{\theta_k-\sqrt{s}}{\theta_{k+1}}
    -
    s\mu
    \left(
        \frac{\theta_k-\sqrt{s}}{\sqrt{s}}-\beta_k
    \right).
\]
Since $1-s\mu>0$, solving for $\beta_k$ gives
\begin{equation}\label{eq_beta}
    \beta_k
    =
    \frac{
        (\theta_k-\sqrt{s})
        \bigl(1-\mu\sqrt{s}\theta_{k+1}\bigr)
    }{
        (1-s\mu)\theta_{k+1}
    }.
\end{equation}
The factor $1-\mu\sqrt{s}\theta_{k+1}$ encodes the discrete
strong-convexity correction, while the ratio involving $\theta_k$ and
$\theta_{k+1}$ determines the inertial extrapolation.

With the choice \eqref{eq_beta}, the discretization
\eqref{eq_one_system} reduces to the one-sequence accelerated
forward-backward scheme
\begin{equation}\label{eq_one_algorithm}
\left\{
\begin{aligned}
    \beta_k
    &=
    \frac{
        (\theta_k-\sqrt{s})
        \bigl(1-\mu\sqrt{s}\theta_{k+1}\bigr)
    }{
        (1-s\mu)\theta_{k+1}
    },\\[4pt]
    y_k
    &=
    x_k+\beta_k(x_k-x_{k-1}),\\[5pt]
    x_{k+1}
    &=
    \prox_{sg}\bigl(y_k-s\nabla f(y_k)\bigr).
\end{aligned}
\right.
\end{equation}
Thus, its implementation requires only the two most recent primal iterates.
The auxiliary state can be reconstructed, when needed for the analysis,
from
\[
    z_k
    =
    x_k+
    \frac{\theta_k-\sqrt{s}}{\sqrt{s}}
    (x_k-x_{k-1}).
\]

\paragraph{Two-sequence discretization.}

We next retain the auxiliary state and evaluate the implicit nonsmooth term
at $z_{k+1}$ rather than at $x_{k+1}$. This gives
\begin{equation}\label{eq_two_discrete}
\begin{aligned}
    \frac{x_{k+1}-x_k}{\sqrt{s}}
    &=
    \frac{1}{\theta_{k+1}}
    \bigl(z_{k+1}-x_k\bigr),\\
    \frac{z_{k+1}-z_k}{\sqrt{s}}
    &\in
    -\mu\theta_{k+1}(z_k-y_k)
    -\theta_{k+1}
    \bigl(\nabla f(y_k)+\partial g(z_{k+1})\bigr),
\end{aligned}
\end{equation}
where
\[
    y_k:=x_k+\beta_k(x_k-x_{k-1}),
\]
and $\beta_k$ is again given by \eqref{eq_beta}. Thus, the extrapolated
point and the strong-convexity correction are unchanged; only the point at
which the nonsmooth term is evaluated differs from the one-sequence
discretization.

The first equation of \eqref{eq_two_discrete} gives
\[
    x_{k+1}
    =
    \left(
        1-\frac{\sqrt{s}}{\theta_{k+1}}
    \right)x_k
    +
    \frac{\sqrt{s}}{\theta_{k+1}}z_{k+1}.
\]
Thus, $z_{k+1}$ is first computed through a proximal step, after which
$x_{k+1}$ is obtained through an affine coupling of $x_k$ and $z_{k+1}$.
The second equation of \eqref{eq_two_discrete} is equivalent to
\[
    z_{k+1}
    =
    \prox_{\sqrt{s}\theta_{k+1}g}
    \left(
        z_k
        -\mu\sqrt{s}\theta_{k+1}(z_k-y_k)
        -\sqrt{s}\theta_{k+1}\nabla f(y_k)
    \right).
\]
Consequently, the two-sequence scheme takes the form
\begin{equation}\label{eq_two_algorithm}
\left\{
\begin{aligned}
    \beta_k
    &=
    \frac{
        (\theta_k-\sqrt{s})
        \bigl(1-\mu\sqrt{s}\theta_{k+1}\bigr)
    }{
        (1-s\mu)\theta_{k+1}
    },\\[4pt]
    y_k
    &=
    x_k+\beta_k(x_k-x_{k-1}),\\[4pt]
    z_{k+1}
    &=
    \prox_{\sqrt{s}\theta_{k+1}g}
    \left(
        z_k
        -\mu\sqrt{s}\theta_{k+1}(z_k-y_k)
        -\sqrt{s}\theta_{k+1}\nabla f(y_k)
    \right),\\[4pt]
    x_{k+1}
    &=
    \left(
        1-\frac{\sqrt{s}}{\theta_{k+1}}
    \right)x_k
    +
    \frac{\sqrt{s}}{\theta_{k+1}}z_{k+1}.
\end{aligned}
\right.
\end{equation}
Using \eqref{eq_beta} and the first equation of
\eqref{eq_two_discrete}, the same coefficient matching as in the
one-sequence case gives
\[    \frac{y_k-x_{k+1}}{s}
    \in
    \nabla f(y_k)+\partial g(z_{k+1}).
\]
This relation is the two-sequence counterpart of the optimality condition
associated with \eqref{eq_one_prox}. The subgradient is evaluated at
$z_{k+1}$, whereas the displacement on the left-hand side is measured
between $y_k$ and $x_{k+1}$.

The two-sequence scheme can also be written without the displacement
$x_k-x_{k-1}$. Indeed, the relation
$
    z_k=x_k+\frac{\theta_k-\sqrt{s}}{\sqrt{s}}(x_k-x_{k-1})
$
and \eqref{eq_beta} give
\[
    y_k
    =
    \frac{\theta_{k+1}-\sqrt{s}}
         {\theta_{k+1}(1-s\mu)}x_k
    +
    \frac{
        \sqrt{s}
        \bigl(1-\mu\sqrt{s}\theta_{k+1}\bigr)
    }{
        \theta_{k+1}(1-s\mu)
    }z_k.
\]
Under Assumption~\ref{ass_discrete}, the two coefficients in the
representation of $y_k$ are nonnegative and sum to one. Hence, $y_k$ is
a convex combination of $x_k$ and $z_k$. A direct calculation also gives
\[
    z_k-\mu\sqrt{s}\theta_{k+1}(z_k-y_k)
    =
    \frac{\theta_{k+1}}{\sqrt{s}}y_k
    -
    \frac{\theta_{k+1}-\sqrt{s}}{\sqrt{s}}x_k.
\]
Substituting this identity into \eqref{eq_two_algorithm} yields the
equivalent affine representation
\begin{equation}\label{eq_two_affine}
\left\{
\begin{aligned}
    y_k
    &=
    \frac{\theta_{k+1}-\sqrt{s}}
         {\theta_{k+1}(1-s\mu)}x_k
    +
    \frac{
        \sqrt{s}
        \bigl(1-\mu\sqrt{s}\theta_{k+1}\bigr)
    }{
        \theta_{k+1}(1-s\mu)
    }z_k,
    \\[4pt]
    z_{k+1}
    &=
    \prox_{\sqrt{s}\theta_{k+1}g}
    \left(
        \frac{\theta_{k+1}}{\sqrt{s}}y_k
        -
        \frac{\theta_{k+1}-\sqrt{s}}{\sqrt{s}}x_k
        -
        \sqrt{s}\theta_{k+1}\nabla f(y_k)
    \right),
    \\[4pt]
    x_{k+1}
    &=
    \left(
        1-\frac{\sqrt{s}}{\theta_{k+1}}
    \right)x_k
    +
    \frac{\sqrt{s}}{\theta_{k+1}}z_{k+1}.
\end{aligned}
\right.
\end{equation}
This representation emphasizes the direct interaction among the primal
state $x_k$, the auxiliary state $z_k$, and the extrapolated point $y_k$.
It is also convenient for comparison with classical two-sequence
accelerated forward-backward methods.

The two discretizations use the same extrapolated point and the same
inertial coefficient. Their difference lies in the treatment of the
auxiliary state. The one-sequence formulation performs the proximal step at
$x_{k+1}$ and eliminates $z_k$ from the implementation, whereas the
two-sequence formulation updates $z_{k+1}$ proximally and then obtains
$x_{k+1}$ through an affine coupling. They therefore provide complementary
discrete realizations of the same composite two-state dynamic
\eqref{eq_discrete_dynamics}.

In the convex case $\mu=0$, setting
$
    \theta_k=\sqrt{s}\,t_k
$
gives
\[
    \beta_k=\frac{t_k-1}{t_{k+1}}.
\]
When $\{t_k\}$ is chosen according to a standard Nesterov acceleration \cite{Nesterov1983}, the
one-sequence formulation \eqref{eq_one_algorithm} reduces to the FISTA
accelerated forward-backward scheme
\cite{BeckFista,ChambolleJota}. Under the corresponding parametrization, the
affine representation \eqref{eq_two_affine} recovers the 
Tseng-type accelerated forward-backward scheme
\cite{Tseng2008,Nesterov12}. In the strongly convex case $\mu>0$, the constant choice
$
    \theta_k=1/\sqrt{\mu}
$
gives
\[
    \beta_k=\frac{1-\sqrt{s\mu}}{1+\sqrt{s\mu}}.
\]
The one- and two-sequence formulations then reduce, respectively, to the
first- and second-type strongly convex Nesterov accelerated schemes, in the
terminology of \cite{LinML}; see also \cite{Nesterov2004}. The present
framework extends these convex and strongly convex endpoint schemes by
allowing the inertial coefficient to transition between the corresponding
parameter regimes. A detailed comparison with existing accelerated methods
is given after the convergence rate analysis.

 \subsection{A unified energy estimate and convergence rates}\label{sec_32}

We next establish a unified energy estimate for the one- and two-sequence
schemes. Although their proximal optimality conditions evaluate
$\partial g$ at different points, both methods satisfy the same affine
coupling relation. These common relations allow the
two algorithms to be analyzed within a single framework.

Fix $x^*\in\operatorname{argmin}\Phi$ and set
$\Phi^*:=\Phi(x^*)$. For $k\ge1$, define
\begin{equation}\label{eq_tau}
    \tau_k:=\frac{\sqrt{s}}{\theta_{k+1}}.
\end{equation}
The first equation in either discretization can then be written as
\begin{equation}\label{eq_affine}
    x_{k+1}
    =
    (1-\tau_k)x_k+\tau_k z_{k+1}.
\end{equation}
The following assumption specifies the admissible discrete coefficients.

\begin{assumption}\label{ass_discrete}
For every $k\ge1$, suppose that
\begin{equation}\label{eq_theta}
    \theta_k\ge\sqrt{s},
    \qquad
    \mu\theta_k^2\le1,
\end{equation}
and
\begin{equation}\label{eq_budget}
    \bigl(1-\sqrt{s}\mu\theta_{k+1}\bigr)\theta_k^2
    \ge
    \left(
        1-\frac{\sqrt{s}}{\theta_{k+1}}
    \right)\theta_{k+1}^2.
\end{equation}
\end{assumption}

The first condition in \eqref{eq_theta} ensures that
$
    0<\tau_k\le1,
$
so \eqref{eq_affine} expresses $x_{k+1}$ as a convex combination of
$x_k$ and $z_{k+1}$. Moreover, since $s\mu<1$, the conditions in
\eqref{eq_theta} imply
\begin{equation}\label{eq_musth}
    0
    \le
    \sqrt{s}\mu\theta_k
    =
    \sqrt{s\mu}\,\sqrt{\mu}\theta_k
    \le
    \sqrt{s\mu}
    <1.
\end{equation}
Hence, every factor $1-\sqrt{s}\mu\theta_k$ is strictly positive.

For $k\ge1$, define the discrete weights
\begin{equation}\label{eq_scales}
\begin{aligned}
    A_k
    &:=
    \theta_k^2
    \prod_{i=1}^k
    \frac{1}{1-\sqrt{s}\mu\theta_i},
    \\[3pt]
    M_k
    &:=
    \frac{A_k}{\theta_k^2}
    =
    \prod_{i=1}^k
    \frac{1}{1-\sqrt{s}\mu\theta_i}.
\end{aligned}
\end{equation}
Both sequences are therefore well defined and positive. The factor
$\theta_k^2$ produces the accelerated convex contribution, while the
product in $M_k$ accumulates the contribution of strong convexity. The definition of $A_k$ gives
\[
    \frac{A_k}{A_{k+1}}
    =
    \frac{
        \bigl(1-\sqrt{s}\mu\theta_{k+1}\bigr)\theta_k^2
    }{
        \theta_{k+1}^2
    }.
\]
Consequently, \eqref{eq_budget} is equivalent to
\[
    \frac{A_k}{A_{k+1}}
    \ge
    1-\frac{\sqrt{s}}{\theta_{k+1}}
    =
    1-\tau_k,
\]
or, equivalently,
\begin{equation}\label{eq_scale_budget}
    A_k\ge(1-\tau_k)A_{k+1}.
\end{equation}
Thus, \eqref{eq_budget} controls the admissible growth of the weight
$A_k$.

For either \eqref{eq_one_algorithm} or \eqref{eq_two_algorithm}, let
$(x_k,y_k,z_k)$ denote the associated iterates. In the one-sequence
formulation, the auxiliary state is reconstructed through
\[
    z_{k+1}
    =
    x_k+
    \frac{\theta_{k+1}}{\sqrt{s}}
    (x_{k+1}-x_k),
\]
whereas in the two-sequence formulation it is updated explicitly. Define
the discrete energy
\begin{equation}\label{eq_energy}
    \mathcal E_k
    :=
    A_k\bigl(\Phi(x_k)-\Phi^*\bigr)
    +
    \frac{M_k}{2}\norm{z_k-x^*}^2.
\end{equation}

The common affine coupling  will be used to
show that this energy decreases along the iterates of both discretizations,
leading directly to a unified objective-residual estimate.

\begin{theorem}\label{thm_dis}
Suppose that Assumption \ref{ass_discrete} holds. Then, for the iterates generated by
either \eqref{eq_one_algorithm} or \eqref{eq_two_algorithm}, the energy
sequence defined in \eqref{eq_energy} is nonincreasing.
Consequently,
\[
    \Phi(x_k)-\Phi^*
    \le    \frac{\mathcal E_1}{A_k},
    \qquad
    k\ge1,
\]
where $\mathcal E_1= A_1\bigl(\Phi(x_1)-\Phi^*\bigr)
    +
    \frac{M_1}{2}\norm{z_1-x^*}^2$.
\end{theorem}

\begin{proof}
Define the common residual
\begin{equation}\label{eq_residual}
    p_{k+1}
    :=
    \frac{y_k-x_{k+1}}{s}.
\end{equation}
For the one-sequence algorithm \eqref{eq_one_algorithm}, the proximal
optimality condition gives
\[
    p_{k+1}
    \in
    \nabla f(y_k)+\partial g(x_{k+1}),
\]
whereas the two-sequence algorithm \eqref{eq_two_algorithm} satisfies
\[
    p_{k+1}
    \in
    \nabla f(y_k)+\partial g(z_{k+1}).
\]
Thus, the two algorithms differ only in the point at which the nonsmooth
subgradient is evaluated.

Using
\eqref{eq_residual} and the corresponding proximal optimality condition,
the second equation of either \eqref{eq_one_system} or
\eqref{eq_two_discrete} can be written as
\[
    \frac{z_{k+1}-z_k}{\sqrt{s}}
    =
    -\mu\theta_{k+1}(z_k-y_k)
    -\theta_{k+1}p_{k+1}.
\]
Therefore,
\begin{equation}\label{eq_z_recursion}
\begin{aligned}
    z_{k+1}
    &=
    z_k
    -\sqrt{s}\mu\theta_{k+1}(z_k-y_k)
    -\sqrt{s}\theta_{k+1}p_{k+1}\\
    &=
    \bigl(1-\sqrt{s}\mu\theta_{k+1}\bigr)z_k
    +
    \sqrt{s}\mu\theta_{k+1}y_k
    -
    \frac{s}{\tau_k}p_{k+1},
\end{aligned}
\end{equation}
where the last equality follows from \eqref{eq_tau}.

\emph{Objective estimate.}
Since $s\le1/L$, the $L$-smoothness of $f$ and
\eqref{eq_residual} give
\[
\begin{aligned}
    f(x_{k+1})
    &\le
    f(y_k)
    +
    \ip{\nabla f(y_k)}{x_{k+1}-y_k}
    +
    \frac{L}{2}\norm{x_{k+1}-y_k}^2\\
    &\le
    f(y_k)
    +
    \ip{\nabla f(y_k)}{x_{k+1}-y_k}
    +
    \frac{s}{2}\norm{p_{k+1}}^2.
\end{aligned}
\]
By the $\mu$-strong convexity of $f$ and $0<\tau_k\le1$,
\[
\begin{aligned}
    (1-\tau_k)f(x_k)+\tau_kf(x^*)= {}& f(y_k)+ (1-\tau_k)(f(x_k)-f(y_k))+ \tau_k(f(x^*)-f(y_k)) \\
    \ge{}&
    f(y_k)
    +
    \ip{\nabla f(y_k)}
    {(1-\tau_k)x_k+\tau_kx^*-y_k}\\
    &+
    \frac{\mu}{2}
    \left[
        (1-\tau_k)\norm{y_k-x_k}^2
        +
        \tau_k\norm{y_k-x^*}^2
    \right].
\end{aligned}
\]
Combining these inequalities and using \eqref{eq_affine}, we obtain
\begin{equation}\label{eq_estf}
\begin{aligned}
    &f(x_{k+1})
    -(1-\tau_k)f(x_k)
    -\tau_k f(x^*)
    \\[2pt]
    &\quad\le
    \tau_k
    \ip{\nabla f(y_k)}{z_{k+1}-x^*}
    +
    \frac{s}{2}\norm{p_{k+1}}^2-
    \frac{\mu}{2}
    \left[
        (1-\tau_k)\norm{y_k-x_k}^2
        +
        \tau_k\norm{y_k-x^*}^2
    \right].
\end{aligned}	
\end{equation}

We next estimate the nonsmooth term. For the one-sequence algorithm,
\[
    p_{k+1}-\nabla f(y_k)\in\partial g(x_{k+1}).
\]
Applying the subgradient inequality with comparison points $x_k$ and
$x^*$, taking the weighted sum with coefficients $1-\tau_k$ and
$\tau_k$, and using \eqref{eq_affine}, we obtain
\[
\begin{aligned}
  g(x_{k+1})
    -(1-\tau_k)g(x_k)
    -\tau_k g(x^*)  &= -(1-\tau_k)(g(x_k)-g(x_{k+1}))
    -\tau_k (g(x^*)-g(x_{k+1}))\\
    &\quad\le
    \ip{p_{k+1}-\nabla f(y_k)}
       {x_{k+1}-(1-\tau_k)x_k-\tau_k x^*}\\
    &\quad=
    \tau_k
    \ip{p_{k+1}-\nabla f(y_k)}
       {z_{k+1}-x^*}.
\end{aligned}
\]

For the two-sequence algorithm,
\[
    p_{k+1}-\nabla f(y_k)\in\partial g(z_{k+1}).
\]
Since $0<\tau_k\le1$, the convexity of $g$ and
\eqref{eq_affine} yield
\[
\begin{aligned}
    g(x_{k+1})
    &\le
    (1-\tau_k)g(x_k)+\tau_k g(z_{k+1})\\
    &\le
    (1-\tau_k)g(x_k)
    +
    \tau_k
    \left(
        g(x^*)
        +
        \ip{p_{k+1}-\nabla f(y_k)}
           {z_{k+1}-x^*}
    \right).
\end{aligned}
\]
Thus, both algorithms \eqref{eq_one_algorithm} and \eqref{eq_two_algorithm} satisfy
\begin{equation}\label{eq_estg}
    g(x_{k+1})
    -(1-\tau_k)g(x_k)
    -\tau_k g(x^*)
    \le
    \tau_k
    \ip{p_{k+1}-\nabla f(y_k)}
       {z_{k+1}-x^*}.
\end{equation}

Adding \eqref{eq_estf} and \eqref{eq_estg} gives
\begin{equation}\label{eq_objective}
\begin{aligned}
    \Phi(x_{k+1})
    -(1-\tau_k)\Phi(x_k)
    -\tau_k\Phi^*
    &\le
    \tau_k
    \ip{p_{k+1}}{z_{k+1}-x^*}
    +
    \frac{s}{2}\norm{p_{k+1}}^2\\
    &\quad
    -
    \frac{\mu}{2}
    \left[
        (1-\tau_k)\norm{y_k-x_k}^2
        +
        \tau_k\norm{y_k-x^*}^2
    \right].
\end{aligned}
\end{equation}

\emph{Estimate of the mixed term.}
Set
\begin{equation}\label{eq_w}
    w_k
    :=
    \bigl(1-\sqrt{s}\mu\theta_{k+1}\bigr)z_k
    +
    \sqrt{s}\mu\theta_{k+1}y_k.
\end{equation}
It follows from \eqref{eq_z_recursion} that
\begin{equation}\label{eq_w_relation}
    w_k-z_{k+1}
    =
    \frac{s}{\tau_k}p_{k+1}.
\end{equation}
Furthermore, \eqref{eq_tau} and \eqref{eq_scales} give
\begin{equation}\label{eq_scale_identity}
    \tau_k^2A_{k+1}=sM_{k+1}.
\end{equation}
Consequently, from \eqref{eq_w_relation},
\[
    \tau_kA_{k+1}
    \ip{p_{k+1}}{z_{k+1}-x^*}
    =
    M_{k+1}
    \ip{w_k-z_{k+1}}{z_{k+1}-x^*}.
\]
Using the identity
\[
    \ip{w_k-z_{k+1}}{z_{k+1}-x^*}
    =
    \frac12\norm{w_k-x^*}^2
    -
    \frac12\norm{z_{k+1}-x^*}^2
    -
    \frac12\norm{w_k-z_{k+1}}^2,
\]
we estimate each term on the right-hand side.

By \eqref{eq_musth},
$
    0\le\sqrt{s}\mu\theta_{k+1}<1.
$
It follows from \eqref{eq_w} and the
 convexity of $\|\cdot\|^2$ that
\[
    \norm{w_k-x^*}^2
    \le
    \bigl(1-\sqrt{s}\mu\theta_{k+1}\bigr)
    \norm{z_k-x^*}^2
    +
    \sqrt{s}\mu\theta_{k+1}
    \norm{y_k-x^*}^2.
\]
The product definition of $M_k$ also gives
\[
    M_k
    =
    \bigl(1-\sqrt{s}\mu\theta_{k+1}\bigr)M_{k+1},
\]
while \eqref{eq_w_relation} and \eqref{eq_scale_identity} imply
\[
    \frac{M_{k+1}}{2}
    \norm{w_k-z_{k+1}}^2
    =
    \frac{sA_{k+1}}{2}
    \norm{p_{k+1}}^2.
\]
Combining these relations yields
\begin{equation}\label{eq_mixed}
\begin{aligned}
  \tau_kA_{k+1}
    \ip{p_{k+1}}{z_{k+1}-x^*}  &\le
    \frac{M_k}{2}\norm{z_k-x^*}^2
    -
    \frac{M_{k+1}}{2}\norm{z_{k+1}-x^*}^2\\
    &\qquad
    +
    \frac{\sqrt{s}\mu\theta_{k+1}M_{k+1}}{2}
    \norm{y_k-x^*}^2
    -
    \frac{sA_{k+1}}{2}\norm{p_{k+1}}^2.
\end{aligned}
\end{equation}

\emph{Energy descent.}
Multiplying \eqref{eq_objective} by $A_{k+1}$ and applying
\eqref{eq_mixed}, the terms involving
$\norm{p_{k+1}}^2$ cancel. Moreover,
\[
    \sqrt{s}\mu\theta_{k+1}M_{k+1}
    =
    \mu\tau_kA_{k+1},
\]
so the terms involving $\norm{y_k-x^*}^2$ cancel as well. We obtain
\[
\begin{aligned}
A_{k+1}\bigl(\Phi(x_{k+1})-\Phi^*\bigr)
    +
    \frac{M_{k+1}}{2}\norm{z_{k+1}-x^*}^2
    &\le
    (1-\tau_k)A_{k+1}
    \bigl(\Phi(x_k)-\Phi^*\bigr)
    +
    \frac{M_k}{2}\norm{z_k-x^*}^2\\
    &\quad
    -
    \frac{\mu}{2}
    (1-\tau_k)A_{k+1}
    \norm{y_k-x_k}^2.
\end{aligned}
\]
Therefore,
\[
    \mathcal E_{k+1}-\mathcal E_k
    \le
    \bigl((1-\tau_k)A_{k+1}-A_k\bigr)
    \bigl(\Phi(x_k)-\Phi^*\bigr)
    -
    \frac{\mu}{2}
    (1-\tau_k)A_{k+1}
    \norm{y_k-x_k}^2
    \le0,
\]
where the last inequality follows from
\eqref{eq_scale_budget}, $0<\tau_k\le1$, and
$\Phi(x_k)-\Phi^*\ge0$. Hence, $\{\mathcal E_k\}$ is nonnegative and
nonincreasing. It follows that
\[
    A_k\bigl(\Phi(x_k)-\Phi^*\bigr)
    \le
    \mathcal E_k
    \le
    \mathcal E_1,
\]
which proves the result.
\end{proof}

Theorem~\ref{thm_dis} gives the same energy estimate for the one- and
two-sequence schemes, despite the different points at which their
nonsmooth subgradients are evaluated. When $\mu=0$, one has
$
    M_k=1
$
and
$
    A_k=\theta_k^2,
$
which recovers the usual weight associated with convex acceleration. When
$\mu>0$, the product in \eqref{eq_scales} incorporates the contribution of
strong convexity and can produce geometric growth of $A_k$ for suitable
coefficient choices.

 \paragraph{Relation to the continuous-time coefficient condition.}

The discrete condition \eqref{eq_budget} mirrors the upper coefficient
condition in the continuous-time framework. Indeed, it is equivalent to
\[
    \frac{\theta_{k+1}^2-\theta_k^2}
         {\sqrt{s}\,\theta_{k+1}}
    +
    \mu\theta_k^2
    \le
    1,
\]
or, equivalently,
\begin{equation}\label{eq_eqvi}
    \frac{\theta_{k+1}-\theta_k}{\sqrt{s}}
    \frac{\theta_{k+1}+\theta_k}{\theta_{k+1}}
    +
    \mu\theta_k^2
    \le
    1.
\end{equation}
Thus the first term is a discrete counterpart of $2\dot\theta(t)$, while the
second term corresponds to $\mu\theta(t)^2$.  Let
$
    r_k=t_0+k\sqrt{s},
$
and consider $r_k$ in a fixed finite interval $[t_0,T]$. Suppose that
$\theta_k=\theta(r_k)$. Then, as $s\to0$,
\[
    \frac{\theta_{k+1}-\theta_k}{\sqrt{s}}
    =\dot{\theta}(r_k)+O(\sqrt{s}),
\qquad
    \frac{\theta_{k+1}+\theta_k}{\theta_{k+1}}
    =2+O(\sqrt{s}),
\]
so that the discrete condition \eqref{eq_budget} formally converges to
\[
    2\dot{\theta}(t)+\mu\theta(t)^2\le1.
\]

Moreover,
\[
\begin{aligned}
    \ln A_k
    &=
    2\ln\theta_k
    -\sum_{i=1}^k
    \ln\bigl(1-\sqrt{s}\mu\theta_i\bigr) \\
    &=
    2\ln\theta_k
    +
    \sum_{i=1}^k\sqrt{s}\mu\theta_i
    +
    O(ks).
\end{aligned}
\]
Since $r_k\le T$, we have $k\sqrt{s}=O(1)$ and hence
$
    O(ks)=O(\sqrt{s}).
$
Therefore, by the Riemann-sum approximation,
\[
    \ln A_k
    =
    2\ln\theta(r_k)
    +
    \int_{t_0}^{r_k}\mu\theta(r)\,\dd r
    +
    O(\sqrt{s}),
\]
which shows that the discrete weight $A_k$ is consistent with the
continuous weight
\[
    A(t)
    =
    \theta(t)^2
    e^{\int_{t_0}^t\mu\theta(r)\,\dd r}.
\]
  This correspondence explains why the
continuous and discrete objective-residual bounds take the parallel forms
\[
    \Phi(x(t))-\Phi^*
    =
    \bigO\left(\frac{1}{A(t)}\right)
\]
and
\[
    \Phi(x_k)-\Phi^*
    =
    \bigO\left(\frac{1}{A_k}\right).
\]
Specific admissible choices of $\{\theta_k\}$ and their corresponding
convergence rates are discussed in the next subsection.

\subsection{Classical schemes and discrete transition coefficients}\label{subsec_parameters}

The coefficient sequence ${\theta_k}$ determines both accelerated
forward-backward schemes through the inertial coefficient
\eqref{eq_beta}. We first show how the classical convex and strongly convex
schemes, together with the known equality-branch transition, arise from the
 discrete coefficient condition. We then introduce sampled transition
coefficients that provide additional flexibility in the passage between the
two endpoint regimes.

 \paragraph{Classical convex and strongly convex schemes and their
equality-branch transition.}

For convenience, set
\[
    t_k:=\frac{\theta_k}{\sqrt{s}}.
\]
The discrete coefficient condition \eqref{eq_budget} is equivalent to
\begin{equation}\label{eq_t_inequality}
    t_{k+1}^2
    +
    \bigl(s\mu t_k^2-1\bigr)t_{k+1}
    -
    t_k^2
    \le0.
\end{equation}
Thus, the classical convex and strongly convex coefficients, as well as
the equality-branch transition \eqref{eq_canonical_cp} between them, are governed by the same
scalar inequality.

Suppose first that $\mu=0$ and $t_1=1$. Then
\eqref{eq_t_inequality} reduces to
\begin{equation}\label{eq_convex_budget}
    t_{k+1}^2
    \le
    t_k^2+t_{k+1}.
\end{equation}
Moreover,
$A_k=st_k^2$ and
$\beta_k=\frac{t_k-1}{t_{k+1}}$.
Consequently, the one-sequence scheme \eqref{eq_one_algorithm} becomes
\begin{equation}\label{eq_convex_one}
\left\{
\begin{aligned}
    y_k
    &=
    x_k+
    \frac{t_k-1}{t_{k+1}}
    (x_k-x_{k-1}),
    \\[4pt]
    x_{k+1}
    &=
    \prox_{sg}
    \bigl(y_k-s\nabla f(y_k)\bigr),
\end{aligned}
\right.
\end{equation}
whereas the two-sequence formulation \eqref{eq_two_affine} becomes
\begin{equation}\label{eq_convex_two}
\left\{
\begin{aligned}
    y_k
    &=
    \left(
        1-\frac{1}{t_{k+1}}
    \right)x_k
    +
    \frac{1}{t_{k+1}}z_k,
    \\[4pt]
    z_{k+1}
    &=
    \prox_{st_{k+1}g}
    \left(
        z_k-st_{k+1}\nabla f(y_k)
    \right),
    \\[4pt]
    x_{k+1}
    &=
    \left(
        1-\frac{1}{t_{k+1}}
    \right)x_k
    +
    \frac{1}{t_{k+1}}z_{k+1}.
\end{aligned}
\right.
\end{equation}
Hence, \eqref{eq_convex_budget} provides a common admissibility condition
for the one- and two-sequence accelerated forward-backward schemes.
Imposing equality in \eqref{eq_convex_budget} gives
$
    t_{k+1}^2=t_k^2+t_{k+1},
$
or, equivalently,
$
    t_{k+1}
    =
    \frac{1+\sqrt{1+4t_k^2}}{2}.
$
This is the classical Nesterov parameter update
\cite{Nesterov1983}. Under this choice, \eqref{eq_convex_one} reduces to
FISTA \cite{BeckFista}, while \eqref{eq_convex_two} reduces to the
 Tseng-type accelerated forward-backward scheme
\cite{Tseng2008}. Since
$
    t_k\ge(k+1)/2,
$
Theorem \ref{thm_dis} gives
$
    \Phi(x_k)-\Phi^*
    \le
    \frac{4\mathcal E_1}{s(k+1)^2}.
$
Thus, the framework recovers both classical convex schemes and their
$\bigO(k^{-2})$ convergence rate. Condition \eqref{eq_convex_budget} also admits the linearly growing
sequences
\[
    t_k
    =
    \frac{k+\gamma-2}{\gamma-1},
    \qquad
    \beta_k
    =
    \frac{k-1}{k+\gamma-1},
    \qquad
    \gamma\ge3,
\]
as well as broader classes of admissible inertial sequences studied in
\cite{SuJMLR2016,AttouchMP18,BotArxiv,AttouchC18}.

Suppose next that $\mu>0$. The equality associated with
\eqref{eq_t_inequality} is
\begin{equation}\label{eq_t_relation}
    t_{k+1}^2
    +
    \bigl(s\mu t_k^2-1\bigr)t_{k+1}
    -
    t_k^2
    =
    0.
\end{equation}
A positive fixed point $\bar t$ of this relation satisfies
$
    \bar t\bigl(s\mu\bar t^2-1\bigr)=0.
$
Hence, the unique positive fixed point is
$
    \bar t
    =
    \frac{1}{\sqrt{s\mu}}.
$
The corresponding constant sequence gives
$
    \theta_k
    =
    \sqrt{s}\,\bar t
    =
    \frac{1}{\sqrt{\mu}},
$
and therefore
\[
    A_k
    =
    \frac{1}{\mu}
    \bigl(1-\sqrt{s\mu}\bigr)^{-k},
    \qquad
    \beta_k
    =
    \frac{1-\sqrt{s\mu}}{1+\sqrt{s\mu}}.
\]
It follows from Theorem \ref{thm_dis} that
$
    \Phi(x_k)-\Phi^*
    \le
    \mu\mathcal E_1
    \bigl(1-\sqrt{s\mu}\bigr)^k.
$
For $s=1/L$, the one- and two-sequence formulations reduce,
respectively, to the first- and second-type Nesterov accelerated schemes
for strongly convex composite optimization, in the terminology of
\cite[Sections~2.2.1--2.2.2]{LinML}. In particular,
\[
    \Phi(x_k)-\Phi^*
    =
    \bigO\left(
        \left(
            1-\sqrt{\frac{\mu}{L}}
        \right)^k
    \right).
\]

The same equality relation also yields the known transition between the
convex and strongly convex coefficients. Starting from $t_1=1$ and
selecting the positive root of \eqref{eq_t_relation} at every iteration
gives
\begin{equation}\label{eq_t_explicit}
    t_{k+1}
    =
    \frac{
        1-s\mu t_k^2
        +
        \sqrt{
            \bigl(1-s\mu t_k^2\bigr)^2+4t_k^2
        }
    }{2}.
\end{equation}
The corresponding inertial coefficient is
\begin{equation}\label{eq_beta_transition}
    \beta_k
    =
    \frac{
        (t_k-1)\bigl(1-s\mu t_{k+1}\bigr)
    }{
        (1-s\mu)t_{k+1}
    }.
\end{equation}
This parameter update appears in \cite{ChambolleActa}, and the
one-sequence scheme \eqref{eq_one_algorithm} with
\eqref{eq_beta_transition} coincides with
\cite[Algorithm~5]{ChambolleActa}. For every fixed $k$, letting $\mu\downarrow0$ in
\eqref{eq_t_explicit} recovers the classical Nesterov-FISTA parameter
update \cite{Nesterov1983,BeckFista}. The resulting inertial coefficient is
asymptotically equivalent, up to an index shift, to the critical Nesterov
coefficient
$
    \frac{k-1}{k+2}.
$
Conversely, for every fixed $\mu>0$, the sequence generated by
\eqref{eq_t_explicit} is nondecreasing and bounded above by
$1/\sqrt{s\mu}$, and hence it converges. Passing to the limit in
\eqref{eq_t_relation} then gives
$
    t_k\to\frac{1}{\sqrt{s\mu}},
$
 and hence
\eqref{eq_beta_transition} gives
$
    \beta_k
    \to
    \frac{1-\sqrt{s\mu}}{1+\sqrt{s\mu}}$
as $k\to+\infty$;
see \cite{ChambolleActa}. Thus, the equality branch connects the classical
convex and strongly convex inertial coefficients. For a prescribed $t_k$, the admissibility condition
\eqref{eq_t_inequality} is equivalent to
\[
    0<t_{k+1}
    \le
    \frac{
        1-s\mu t_k^2
        +
        \sqrt{
            \bigl(1-s\mu t_k^2\bigr)^2+4t_k^2
        }
    }{2}.
\]
Hence, \eqref{eq_t_explicit} selects the largest admissible value of
$t_{k+1}$ at each iteration. Starting from $t_1=1$, the equality branch selects the largest admissible
value of $t_{k+1}$ at each iteration, thereby giving the fastest stepwise
growth toward the strongly convex fixed point. Under the parametrization above, the weight $A_k$ satisfies
\[
    A_k
    =
    \frac{s}{
        {1}/{t_k^2}-s\mu
    }.
\]
The estimates established in the proof of
\cite[Theorem~4.10]{ChambolleActa} give
\[
    \frac{1}{t_k^2}-s\mu
    \le
    \min\left\{
        \frac{4}{(k+1)^2},
        \bigl(1+\sqrt{s\mu}\bigr)
        \bigl(1-\sqrt{s\mu}\bigr)^k
    \right\}.
\]
Consequently, Theorem \ref{thm_dis} yields
\begin{equation}\label{eq_transition_rate}
    \Phi(x_k)-\Phi^*
    \le
    \frac{\mathcal E_1}{s}
    \min\left\{
        \frac{4}{(k+1)^2},
        \bigl(1+\sqrt{s\mu}\bigr)
        \bigl(1-\sqrt{s\mu}\bigr)^k
    \right\},
    \qquad
    k\ge1.
\end{equation}
Under the parameter rule
\eqref{eq_t_explicit}--\eqref{eq_beta_transition}, the one-sequence scheme
\eqref{eq_one_algorithm} coincides with
\cite[Algorithm~5]{ChambolleActa}. The estimate
\eqref{eq_transition_rate}, derived here from Theorem \ref{thm_dis}, matches the
convergence result in \cite[Theorem~4.10]{ChambolleActa}.

As shown in Section \ref{sec23} , the discrete coefficient condition
\eqref{eq_t_inequality} is the counterpart of the continuous condition
\[
    2\dot\theta(t)+\mu\theta(t)^2\le1.
\]
Accordingly, the equality rule \eqref{eq_t_explicit} considered in
\cite{ChambolleActa} corresponds to imposing equality in the continuous
coefficient condition:
\[
    2\dot\theta(t)+\mu\theta(t)^2=1.
\]
As discussed in Section~\ref{sec:dyna}, the latter generates the
hyperbolic coefficient
\[
    \theta_\mu(t)
    =
    \frac{1}{\sqrt{\mu}}
    \tanh\left(\frac{\sqrt{\mu}t}{2}\right).
\]
In both settings, imposing equality selects the largest admissible
coefficient growth and therefore the fastest admissible transition from
the convex regime to the strongly convex regime. Thus, the present
framework not only includes \cite[Algorithm~5]{ChambolleActa} and yields
the matching convergence estimate, but also provides a continuous-time
dynamical interpretation of its equality-based parameter rule.

\paragraph{A sampling principle for discrete transition coefficients.}

The equality branch gives an explicit parameter rule and a convergence
estimate combining $\bigO(k^{-2})$ and $\bigO\left(\bigl(1-\sqrt{s\mu}\bigr)^k\right) $ decay. Once $t_1=1$ is
fixed, however, equality in \eqref{eq_budget} determines the entire
parameter sequence. Its convex limit reduces to the standard Nesterov rule \cite{Nesterov1983},
but does not directly include the broader family
\[
    \frac{k-1}{k+\gamma-1},
    \qquad
    \gamma\ge3.
\]
Moreover, the equality branch contains no additional parameter for
controlling the transition from the convex coefficients to the limiting
strongly convex coefficient.

To introduce this flexibility, we construct discrete coefficient sequences
by sampling the transition functions used in the continuous analysis. The
resulting sequences satisfy the discrete coefficient condition, converge to
the above convex family as $\mu\downarrow0$, and approach the classical
strongly convex coefficient as $k\to+\infty$.

\begin{theorem}\label{thm_sampled}
 Let $\mu>0$ and
$S:[0,+\infty)\to[0,1)$ be continuously differentiable, nondecreasing,
and concave with $S(0)=0$. Suppose that \eqref{eq_S_cond} and \eqref{eq_S_asy} hold for
some $\kappa\in(0,1/2]$. Denote
\begin{equation}\label{eq_xi_fin}
	  \xi
    :=
    \int_0^{+\infty}
    \bigl(1-S(r)\bigr)\,\dd r<+\infty,
\end{equation}
  and define
\begin{equation}\label{eq_sample_theta}
    \theta_{k+1}
    :=
    \frac{1}{\sqrt{\mu}}
    \left[
        \sqrt{s\mu}
        +
        \bigl(1-\sqrt{s\mu}\bigr)
        S\bigl(k\sqrt{s\mu}\bigr)
    \right],
    \qquad k\ge0.
\end{equation}
Then the following statements hold.

\begin{enumerate}
\item[(i)]
The sequence $\{\theta_k\}$ satisfies Assumption \ref{ass_discrete}. Moreover, it is
nondecreasing and
\[
    \theta_k\to\frac{1}{\sqrt{\mu}}
    \qquad
    \text{as }k\to+\infty.
\]

\item[(ii)]
The corresponding inertial coefficient \eqref{eq_beta} satisfies
\begin{equation*}
\begin{aligned}
    \beta_k
    &\to
    \frac{\kappa(k-1)}{1+\kappa k}
    &&\text{as }\mu\downarrow0
    \quad\text{for each fixed }k\ge1,\\
    \beta_k
    &\to
    \frac{1-\sqrt{s\mu}}{1+\sqrt{s\mu}}
    &&\text{as }k\to+\infty
    \quad\text{for each fixed }\mu>0.
\end{aligned}
\end{equation*}
In particular, if
$
    \kappa=1/(\gamma-1)
$
for some $\gamma\ge3$, then
\[
    \beta_k
    \to
    \frac{k-1}{k+\gamma-1}
    \qquad
    \text{as }\mu\downarrow0
\]
for every fixed $k\ge1$.

\item[(iii)] The iterates generated by either \eqref{eq_one_algorithm} or
\eqref{eq_two_algorithm}/\eqref{eq_two_affine} satisfy
\begin{equation}\label{eq_sample_rate}
\begin{aligned}
    \Phi(x_k)-\Phi^*
    \le
    \frac{\mathcal E_1}{s}
    \min\Bigg\{&
        \frac{1
        }{
           C_{s,\mu}(k-1)^2
        }, e^{\sqrt{s\mu}+\xi}
        \bigl(1-\sqrt{s\mu}\bigr)^k
    \Bigg\},
    \qquad k>1.
\end{aligned}
\end{equation}
where $C_{s,\mu}$ is defined in \eqref{eq_Csu}.
Consequently, $\Phi(x_k)-\Phi^*
    =
    \bigO\left(
        \min\left\{
            \frac{1}{k^2},
            \bigl(1-\sqrt{s\mu}\bigr)^k
        \right\}
    \right).
$
\end{enumerate}
\end{theorem}
\begin{proof}
$(i)$ 
Since $S(0)=0$, \eqref{eq_sample_theta} gives
$
    \theta_1
    =
    \frac{\sqrt{s\mu}}{\sqrt{\mu}}
    =
    \sqrt{s}.
$
Because $0\le S(r)<1$,
$
    \sqrt{s}
    \le
    \theta_{k+1}
    <
    \frac{1}{\sqrt{\mu}}.
$
The monotonicity of $S$ and the limit $S(r)\to1$ further imply that
$\{\theta_k\}$ is nondecreasing and
\[
    \theta_k\to\frac{1}{\sqrt{\mu}}
    \qquad
    \text{as }k\to+\infty.
\]
In particular,
\[
    \mu\theta_k^2\le1,
    \qquad
    0\le\sqrt{s}\mu\theta_k<1.
\]

It remains to verify \eqref{eq_budget}. For $k\ge1$, set
$
    r:=(k-1)\sqrt{s\mu}.
$
The concavity of $S$ gives
\begin{equation}\label{eq_S_conc}
    \frac{
        S(r+\sqrt{s\mu})-S(r)
    }{
        \sqrt{s\mu}
    }
    \le
    S'(r).
\end{equation}
Since $S$ is nondecreasing, one has
$
    \theta_k\le\theta_{k+1}.
$
It follows from \eqref{eq_sample_theta} and \eqref{eq_S_conc} that
\[
\begin{aligned}
    &\frac{\theta_{k+1}-\theta_k}{\sqrt{s}}
    \frac{\theta_{k+1}+\theta_k}{\theta_{k+1}}
    +
    \mu\theta_k^2\\
    &\quad\le
    2\bigl(1-\sqrt{s\mu}\bigr)S'(r)
    +
    \left[
        \sqrt{s\mu}
        +
        \bigl(1-\sqrt{s\mu}\bigr)S(r)
    \right]^2\\
    &\quad\le
    \bigl(1-\sqrt{s\mu}\bigr)
    \bigl(1-S(r)^2\bigr)
    +
    \left[
        \sqrt{s\mu}
        +
        \bigl(1-\sqrt{s\mu}\bigr)S(r)
    \right]^2\\
    &\quad=
    1-
    \sqrt{s\mu}\bigl(1-\sqrt{s\mu}\bigr)
    \bigl(1-S(r)\bigr)^2\\
    &\quad\le1,
\end{aligned}
\]
where the second inequality follows from \eqref{eq_S_cond}. Thus,
\eqref{eq_eqvi}, and hence \eqref{eq_budget}, holds. Therefore,
Assumption \ref{ass_discrete} is satisfied.

$(ii)$
Since $S(0)=0$ and $S'(0)=\kappa$,
\[
    S(r)=\kappa r+o(r)
    \qquad
    \text{as }r\downarrow0.
\]
For every fixed $k$,
\[
    \frac{\theta_{k+1}}{\sqrt{s}}
    =
    1+
    \frac{1-\sqrt{s\mu}}{\sqrt{s\mu}}
    S\bigl(k\sqrt{s\mu}\bigr)=
    1+\kappa k+o(1)
    \qquad\text{as }\mu\downarrow0.
\]
Substituting these expansions into \eqref{eq_beta} gives
\[
    \beta_k
    \to
    \frac{\kappa(k-1)}{1+\kappa k}
    \qquad
    \text{as }\mu\downarrow0.
\]
For every fixed $\mu>0$, $(i)$ gives
$
    \theta_k\to1/\sqrt{\mu}.
$
It then follows from \eqref{eq_beta} that
\[
    \beta_k
    \to
    \frac{1-\sqrt{s\mu}}{1+\sqrt{s\mu}}
    \qquad
    \text{as }k\to+\infty.
\]
This proves $(ii)$.

$(iii)$ From
\eqref{eq_sample_theta}, we can compute
\[
\begin{aligned}
    1-\sqrt{s}\mu\theta_i
    =
    \bigl(1-\sqrt{s\mu}\bigr)
    \left(
        1+\sqrt{s\mu}
        \left(
            1-S\bigl((i-1)\sqrt{s\mu}\bigr)
        \right)
    \right).
\end{aligned}
\]
Consequently, from \eqref{eq_scales}
\begin{equation}\label{eq_sample_scale}
\begin{aligned}
    A_k
    =
    \frac{\theta_k^2}{
        \bigl(1-\sqrt{s\mu}\bigr)^k
    }
    \prod_{j=0}^{k-1}
    \frac{1}{
        1+\sqrt{s\mu}
        \left(
            1-S\bigl(j\sqrt{s\mu}\bigr)
        \right)
    }.
\end{aligned}
\end{equation}
Since $1-S$ is nonnegative and nonincreasing, from \eqref{eq_xi_fin}, we have
\[
\begin{aligned}
    &\sqrt{s\mu}
    \sum_{j=0}^{k-1}
    \left(
        1-S\bigl(j\sqrt{s\mu}\bigr)
    \right)\le
    \sqrt{s\mu}
    +
    \int_0^{(k-1)\sqrt{s\mu}}
    \bigl(1-S(r)\bigr)\,\dd r\le
    \sqrt{s\mu}+\xi.
\end{aligned}
\]
Together  with $\ln(1+u)\le u$ for $u\ge0$ yields
\[
\begin{aligned}
    &\prod_{j=0}^{k-1}
    \frac{1}{
        1+\sqrt{s\mu}
        \left(
            1-S\bigl(j\sqrt{s\mu}\bigr)
        \right)
    }=
    e^{\left(
        -\sum_{j=0}^{k-1}
        \ln\left[
            1+\sqrt{s\mu}
            \left(
                1-S\bigl(j\sqrt{s\mu}\bigr)
            \right)
        \right]
    \right)}\\
    &\quad\ge
    e^{\left(
        -\sqrt{s\mu}
        \sum_{j=0}^{k-1}
        \left(
            1-S\bigl(j\sqrt{s\mu}\bigr)
        \right)
    \right)}\\
    &\quad\ge
    e^{-\sqrt{s\mu}-\xi}.
\end{aligned}
\]
Combining this bound with \eqref{eq_sample_scale} and
$\theta_k\ge\sqrt{s}$ yields
\begin{equation}\label{eq_sample_geometric_scale}
    A_k
    \ge
    s e^{-\sqrt{s\mu}-\xi}
    \bigl(1-\sqrt{s\mu}\bigr)^{-k}.
\end{equation}

Next, we derive a global quadratic lower bound by using the
same two-region argument as in the continuous analysis.  By
\eqref{eq_S_low},
\begin{equation}\label{eq_S_low1}
	 S\bigl((k-1)\sqrt{s\mu}\bigr)
    \ge
    S(1)
    \min\left\{
        (k-1)\sqrt{s\mu},1
    \right\}.
\end{equation}

Suppose first that $
    (k-1)\sqrt{s\mu}\le1$. It follows from \eqref{eq_sample_theta} and $S(1)\in[0,1)$ that
\[
    \theta_k\ge
    \sqrt{s}
    \left[
        1+
        \bigl(1-\sqrt{s\mu}\bigr)S(1)(k-1)
    \right]\ge
    \sqrt{s}
    \bigl(1-\sqrt{s\mu}\bigr)S(1)k.
\]
Therefore, since
$0\le\sqrt{s}\mu\theta_i<1$ for $1\le i\le k$,
\eqref{eq_scales} gives
\begin{equation}\label{eq_sample_quad_small}
    A_k\ge \theta_k^2
    \ge
    s\bigl(1-\sqrt{s\mu}\bigr)^2S(1)^2k^2.
\end{equation}

Suppose next that $(k-1)\sqrt{s\mu}\ge1$. Set
$
    m:=\left\lceil\frac{1}{\sqrt{s\mu}}\right\rceil+1.
$, then $k\ge m$. For every $i\ge m$, \eqref{eq_sample_theta} and
\eqref{eq_S_low1} imply
\[
    \theta_i
    \ge
    \frac{
        \bigl(1-\sqrt{s\mu}\bigr)S(1)
    }{\sqrt{\mu}}.
\]
Moreover, using $-\ln(1-u)\ge u$ for $0\le u<1$, and
$\sqrt{s}\mu\theta_i<1$, we have
\[
\begin{aligned}
    \ln
    \prod_{i=1}^k
    \frac{1}{1-\sqrt{s}\mu\theta_i}
    &\ge
    \sum_{i=1}^k\sqrt{s}\mu\theta_i\ge
    \sum_{i=m}^k\sqrt{s}\mu\theta_i\\
    &\ge
    \bigl(1-\sqrt{s\mu}\bigr)S(1)
    \sqrt{s\mu}\,(k-m+1)\\
    &\ge
    \bigl(1-\sqrt{s\mu}\bigr)S(1)
    \left(
        (k-1)\sqrt{s\mu}-1
    \right),
\end{aligned}
\]
where the last inequality follows from
$
    \sqrt{s\mu}
    \left\lceil\frac{1}{\sqrt{s\mu}}\right\rceil
    \le 1+\sqrt{s\mu}
$ and $s\mu<1$.
Consequently, from \eqref{eq_scales}
\[
    A_k
    \ge
    \frac{
        \bigl(1-\sqrt{s\mu}\bigr)^2S(1)^2
    }{\mu}
    e^{(1-\sqrt{s\mu})S(1)
        \left(
            (k-1)\sqrt{s\mu}-1  \right)}.
    \]
As in the continuous analysis, the elementary bound
\[
    \frac{e^{ar-a}}{r^2}
    \ge
    \frac{a^2}{4}e^{2-a},
    \qquad
    r\ge1,
    \quad
    0<a<1,
\]
applied with
$
    a=\bigl(1-\sqrt{s\mu}\bigr)S(1)$ and $r=(k-1)\sqrt{s\mu}$
yields
\begin{equation}\label{eq_sample_quad_large}
\begin{aligned}
    A_k
    \ge{}&
    \frac{s}{4}
    \bigl(1-\sqrt{s\mu}\bigr)^4S(1)^4e^{
        2-\left(1-\sqrt{s\mu}\right)S(1)
    }
    (k-1)^2.
\end{aligned}
\end{equation}

Combining \eqref{eq_sample_quad_small} and
\eqref{eq_sample_quad_large}, we obtain
\begin{equation}\label{eq_sample_quad}
    A_k\ge sC_{s,\mu}(k-1)^2,
    \qquad
    k\ge1,
\end{equation}
where
\begin{equation}\label{eq_Csu}
\begin{aligned}
    C_{s,\mu}
    :=
    \min\Bigg\{
        \bigl(1-\sqrt{s\mu}\bigr)^2S(1)^2,        \frac{
            \bigl(1-\sqrt{s\mu}\bigr)^4S(1)^4
        }{4}
        e^{
            2-\left(1-\sqrt{s\mu}\right)S(1)
        }
    \Bigg\}.
\end{aligned}
\end{equation}
In particular, $C_{s,\mu}>0$, and it remains bounded away from zero as
$\mu\downarrow0$ for fixed $s$.

  The estimate \eqref{eq_sample_rate} now follows from
Theorem \ref{thm_dis}, \eqref{eq_sample_geometric_scale}, and
\eqref{eq_sample_quad}.
\end{proof}

We next apply the sampled construction to four representative transition
functions.

\begin{corollary}\label{cor_sampled}
Let $0<\kappa\le1/2$. For $r\ge0$,
consider the transition functions
\begin{equation}\label{eq_sampled_coefficients}
\begin{aligned}
    &\text{\rm Hyperbolic:}
    &\qquad
    S_{\rm hyp}(r)
    &:=
    \tanh(\kappa r),\\[3pt]
    &\text{\rm Exponential:}
    &\qquad
    S_{\rm exp}(r)
    &:=
    1-e^{-\kappa r},\\[3pt]
    &\text{\rm Algebraic:}
    &\qquad
    S_{\rm alg}(r)
    &:=
    \frac{\kappa r}{\sqrt{1+\kappa^2r^2}},\\[3pt]
    &\text{\rm Polynomial:}
    &\qquad
    S_{\rm poly}(r)
    &:=
    1-
    \left(
        1+\frac{\kappa r}{p}
    \right)^{-p},
    \qquad
    p>1.
\end{aligned}
\end{equation}
Each transition function satisfies the assumptions of
Theorem \ref{thm_sampled}. For any one of these functions, define
$\{\theta_k\}$ by \eqref{eq_sample_theta}, and generate $\{x_k\}$ by
either \eqref{eq_one_algorithm} or \eqref{eq_two_algorithm}, with the
corresponding inertial coefficient \eqref{eq_beta}.  Let $\gamma=1+\frac{1}{\kappa}\ge 3$. Then
\[
\begin{aligned}
    \beta_k
    &\to
    \frac{k-1}{k+\gamma-1}
    &&\text{as }\mu\downarrow0
    \quad\text{for each fixed }k\ge1,\\
    \beta_k
    &\to
    \frac{1-\sqrt{s\mu}}{1+\sqrt{s\mu}}
    &&\text{as }k\to+\infty
    \quad\text{for each fixed }\mu>0.
\end{aligned}
\]
Moreover,
\[
    \Phi(x_k)-\Phi^*
    =
    \bigO\left(
        \min\left\{
            \frac{1}{k^2},
            \bigl(1-\sqrt{s\mu}\bigr)^k
        \right\}
    \right).
\]
\end{corollary}

\begin{proof}
The required properties of the exponential, algebraic, and polynomial
transition functions follow from
Corollary \ref{cor_disT}. It remains to verify the
concavity condition for the hyperbolic function.

For
$
    S_{\rm hyp}(r)=\tanh(\kappa r),
$
one has
\[
    S_{\rm hyp}'(r)
    =
    \kappa\bigl(1-S_{\rm hyp}(r)^2\bigr)>0,
\]
and
\[
    S_{\rm hyp}''(r)
    =
    -2\kappa^2S_{\rm hyp}(r)
    \bigl(1-S_{\rm hyp}(r)^2\bigr)
    \le0.
\]
Thus, $S_{\rm hyp}$ is increasing and concave. Moreover,
\[
    2S_{\rm hyp}'(r)+S_{\rm hyp}(r)^2
    =
    2\kappa+
    (1-2\kappa)S_{\rm hyp}(r)^2
    \le1,
\]
and this quantity is strictly positive. Finally,
\[
    S_{\rm hyp}(0)=0,
    \qquad
    S_{\rm hyp}'(0)=\kappa,
    \qquad
    S_{\rm hyp}(r)\to1,
    \qquad
    S_{\rm hyp}'(r)\to0,
\]
and
\[
    \int_0^{+\infty}
    \bigl(1-\tanh(\kappa r)\bigr)\,\dd r
    =
    \frac{\log 2}{\kappa}.
\]
The conclusion follows from Theorem \ref{thm_sampled}.
\end{proof}

\begin{remark}\label{rem_sampled_transition}
The sampled coefficients provide a discrete transition between the
classical convex and strongly convex parameter regimes. Indeed, for every
fixed $k\ge1$, \[
    \beta_k
    \to
    \frac{k-1}{k+\gamma-1}\qquad \gamma\ge 3 \qquad\text{as }\mu\downarrow0
\]
For every fixed $\mu>0$, one instead has
\[
        \beta_k
    \to
    \frac{1-\sqrt{s\mu}}{1+\sqrt{s\mu}}
    \qquad
    \text{as }k\to+\infty.
\]
Thus, the sampled construction moves from the linearly growing convex
coefficients to the constant strongly convex coefficient.

This transition is consistent with the continuous-time construction in
Section~\ref{sec:dyna}. In both settings, the behavior near the origin,
determined by $S'(0)=\kappa$, specifies the convex endpoint, whereas the
limit $S(r)\to1$ specifies the strongly convex endpoint. Different choices
of $S$ produce different transition speeds while preserving the same two
endpoint regimes and the same quadratic and geometric convergence orders.
In the final numerical section, we compare these sampled transitions with
the equality branch. On the tested problems, this comparison examines
whether more gradual transitions can improve the transient objective
decrease, particularly when the strong convexity parameter is small.
\end{remark}

\section{Numerical experiments}
\label{sec_numerics}

This section illustrates the continuous and discrete transition mechanisms
developed in this paper. We first compare several damping
coefficients for the continuous-time dynamic \eqref{dy_main} on a
regularized logistic regression problem. We then examine the one-sequence
and two-sequence accelerated forward-backward schemes
\eqref{eq_one_algorithm} and \eqref{eq_two_algorithm}, equipped with
different inertial coefficients, on an elastic-net regularized inverse
problem and a total-variation image-denoising problem. The experiments
focus on how the transition from the convex regime to the strongly convex
regime affects convergence behavior. The numerical results demonstrate the
effectiveness of the proposed unified continuous-time dynamics and
discrete algorithms.

\subsection{Regularized logistic regression problem}
\label{subsec_continuous}

We first investigate the effect of the convex-limit damping and the choice
of transition function on a regularized logistic regression problem based
on the \texttt{a9a} dataset from the LIBSVM collection. The dataset contains
$N=32561$ training samples with $d=123$ features, and the labels are
encoded as elements of $\{-1,1\}$. We consider
\[
    \min_{x\in\mathbb R^d}
    \Phi(x)
    :=
    \frac{1}{N}
    \sum_{j=1}^N
    \log\left(
        1+e^{-y_j\langle a_j,x\rangle}
    \right)
    +
    \frac{\mu}{2}\norm{x}^2,
\]
where $\mu>0$ is the regularization parameter. A high-accuracy
approximation $x^*$ of the minimizer is computed using a Newton-type
method and is used to evaluate the objective gap.

All continuous-time dynamics are solved in MATLAB using
\texttt{ode45}, with $t_0=0.1$ and $\dot x(t_0)=0$. For each value of
$\mu$, we report
$
    \Phi(x(t))-\Phi(x^*).
$
The seven coefficient choices are summarized in
Table \ref{tab_dy}. The first two are the classical convex and
strongly convex endpoint dynamics, while the remaining five are generated
by transition functions. 
 
 \begin{table}[htbp]
\centering
\caption{Coefficient choices for the \texttt{a9a} logistic regression
experiment.}
\label{tab_dy}
\small

\begingroup
\renewcommand{\arraystretch}{1.55}
\setlength{\extrarowheight}{1.5pt}
\setlength{\tabcolsep}{5pt}

\resizebox{\textwidth}{!}{
\begin{tabular}{
    C{0.7cm}|
    C{2.8cm}|
    C{7.4cm}|
    C{2.0cm}|
    C{2.0cm}
}
\hline
\multirow[c]{2}{*}{ID}
&
\multirow[c]{2}{*}{Choice}
&
\multirow[c]{2}{*}{$\theta(t)$}
&
\multicolumn{2}{|c}{$\delta_\mu(t)$}
\\
\cline{4-5}
&
&
&
$\mu\downarrow0$
&
$t\to+\infty$
\\
\hline

$C_1$
&
NAG-C
&
$\displaystyle \frac{t}{3}$ 
&
$\displaystyle \frac{4}{t}$
&
$0$
\\
\hline

$C_2$
&
NAG-SC
&
$\displaystyle \frac{1}{\sqrt{\mu}}$
&
$0$
&
$\displaystyle 2\sqrt{\mu}$
\\
\hline

$C_3$
&
Equality
&
$\displaystyle
\frac{1}{\sqrt{\mu}}
\tanh\left(\frac{\sqrt{\mu}t}{2}\right)$
&
$\displaystyle \frac{3}{t}$
&
\multirow[c]{5}{*}{$\displaystyle 2\sqrt{\mu}$}
\\
\cline{1-4}

$C_4$
&
Hyperbolic
&
$\displaystyle
\frac{1}{\sqrt{\mu}}
\tanh\left(\frac{\sqrt{\mu}t}{3}\right)$
&
$\displaystyle \frac{4}{t}$
&
\\
\cline{1-4}

$C_5$
&
Exponential
&
$\displaystyle
\frac{1}{\sqrt{\mu}}
\left(
    1-e^{-\frac{\sqrt{\mu}t}{4}}
\right)$
&
$\displaystyle \frac{5}{t}$
&
\\
\cline{1-4}

$C_6$
&
Algebraic
&
$\displaystyle
\frac{1}{\sqrt{\mu}}
\frac{\frac{\sqrt{\mu}t}{4}}
{\sqrt{1+\left(\frac{\sqrt{\mu}t}{4}\right)^2}}$
&
$\displaystyle \frac{5}{t}$
&
\\
\cline{1-4}

$C_7$
&
Polynomial
&
$\displaystyle
\frac{1}{\sqrt{\mu}}
\left[
    1-
    \left(
        1+\frac{\sqrt{\mu}t}{25}
    \right)^{-5}
\right]$
&
$\displaystyle \frac{6}{t}$
&
\\
\hline
\end{tabular}
}
\endgroup
\end{table}

The choices in Table \ref{tab_dy} are designed to separate the
effects of the two endpoint damping regimes from those of the transition
between them. The first two choices are nontransitioning 
dynamics. Choice $C_1$ is the convex Nesterov accelerated gradient flow  (NAG-C)
\eqref{dy_NAG_C} with $\gamma=4$. Its damping coefficient is $4/t$, independently of the strong-convexity parameter of the objective, and
the corresponding objective residual satisfies
$
    \Phi(x(t))-\Phi^*=\bigO(t^{-2}).
$
Choice $C_2$ is the strongly convex Nesterov accelerated gradient flow (NAG-SC)
\eqref{dy_NAG_SC}, with constant damping $2\sqrt{\mu}$ and the
exponential estimate
$
    \Phi(x(t))-\Phi^*
    =
    \bigO\left(e^{-\sqrt{\mu}t}\right).
$
Thus, $C_1$ retains the convex time-dependent damping but has no positive
large-time damping level, whereas $C_2$ applies the strongly convex
damping from the initial time and has no convex-type transient regime.
Choice $C_3$ is the hyperbolic equality-branch transition corresponding
to $\kappa=1/2$. It connects the critical convex damping $3/t$ to the
strongly convex value $2\sqrt{\mu}$ and is related to the constructions
in \cite[Section~3.2]{LuoChen2022} and
\cite[Theorem~4.1]{KimYang2023}. This choice provides a direct transition
between the two classical endpoint dynamics, but it fixes both the
convex-limit damping and the form of the transition. Choice $C_4$ remains within the hyperbolic family but uses
$\kappa=1/3$. Its convex limit is therefore $4/t$, while its strongly
convex endpoint remains $2\sqrt{\mu}$. The comparison between $C_3$ and
$C_4$ shows the effect of replacing the critical convex damping $3/t$ by
the supercritical damping $4/t$. The comparison between $C_1$ and $C_4$
is also informative: both have the same convex-limit damping $4/t$, but
only $C_4$ transitions to a positive strongly convex damping level.
Choices $C_5$ and $C_6$ have the same convex limit $5/t$ but use
exponential and algebraic transition functions, respectively. Their
comparison therefore isolates the effect of the transition function while
keeping both endpoint damping coefficients fixed. Choice $C_7$ provides
an additional polynomial transition with convex limit $6/t$. Its
finite-time behavior is compared with those of the other transition
families below.

By the convergence results in Section \ref{sec23}, all transition choices
$C_3$--$C_7$ satisfy
\[
    \Phi(x(t))-\Phi^*
    =
    \bigO\left(
        \min\left\{
            \frac{1}{t^2},
            e^{-\sqrt{\mu}t}
        \right\}
    \right).
\]
Thus, they have the accelerated convex rate as $\mu\downarrow0$ and the
exponential strongly convex rate for every fixed $\mu>0$. Differences
between their numerical trajectories therefore reflect their transient
behavior rather than their asymptotic convergence orders.

  \begin{figure}[H]
\centering
\begin{subfigure}[t]{0.49\textwidth}
      \raggedright
    \includegraphics[width=\linewidth]{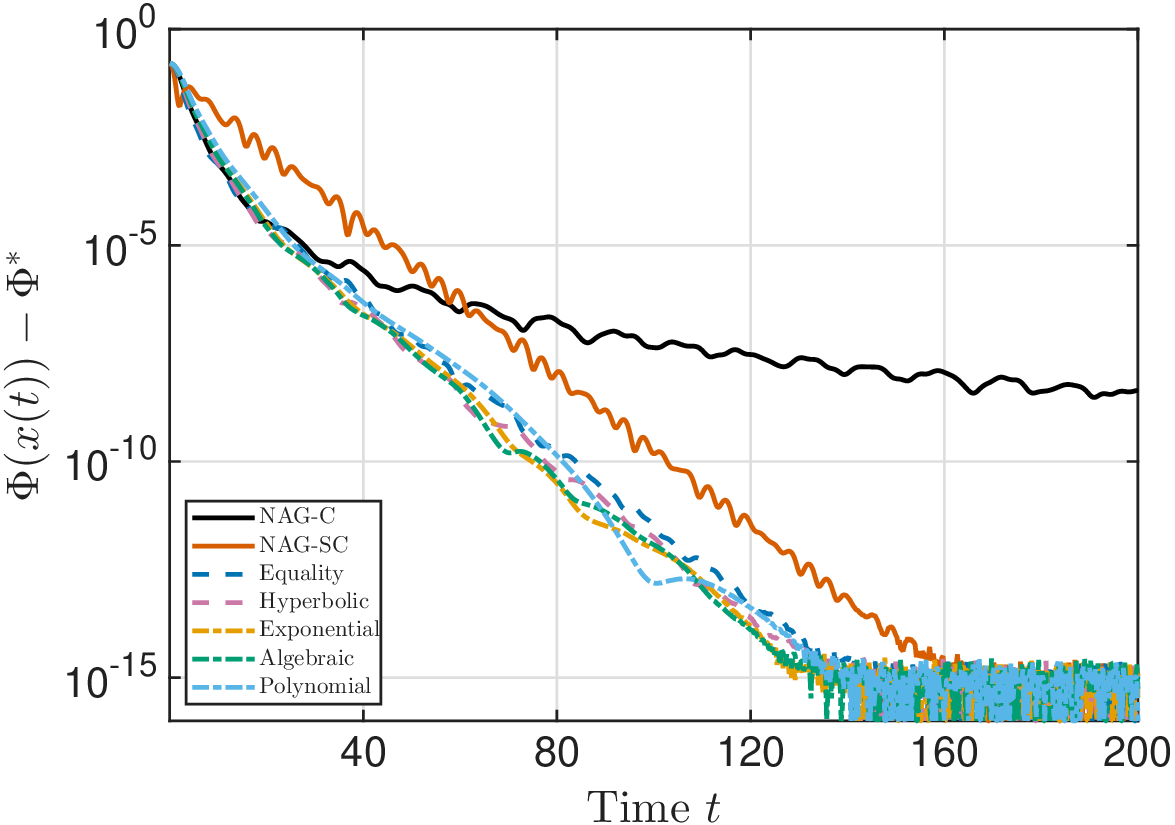}
\end{subfigure}
\hfill
\begin{subfigure}[t]{0.49\textwidth}
      \raggedright
    \includegraphics[width=\linewidth]{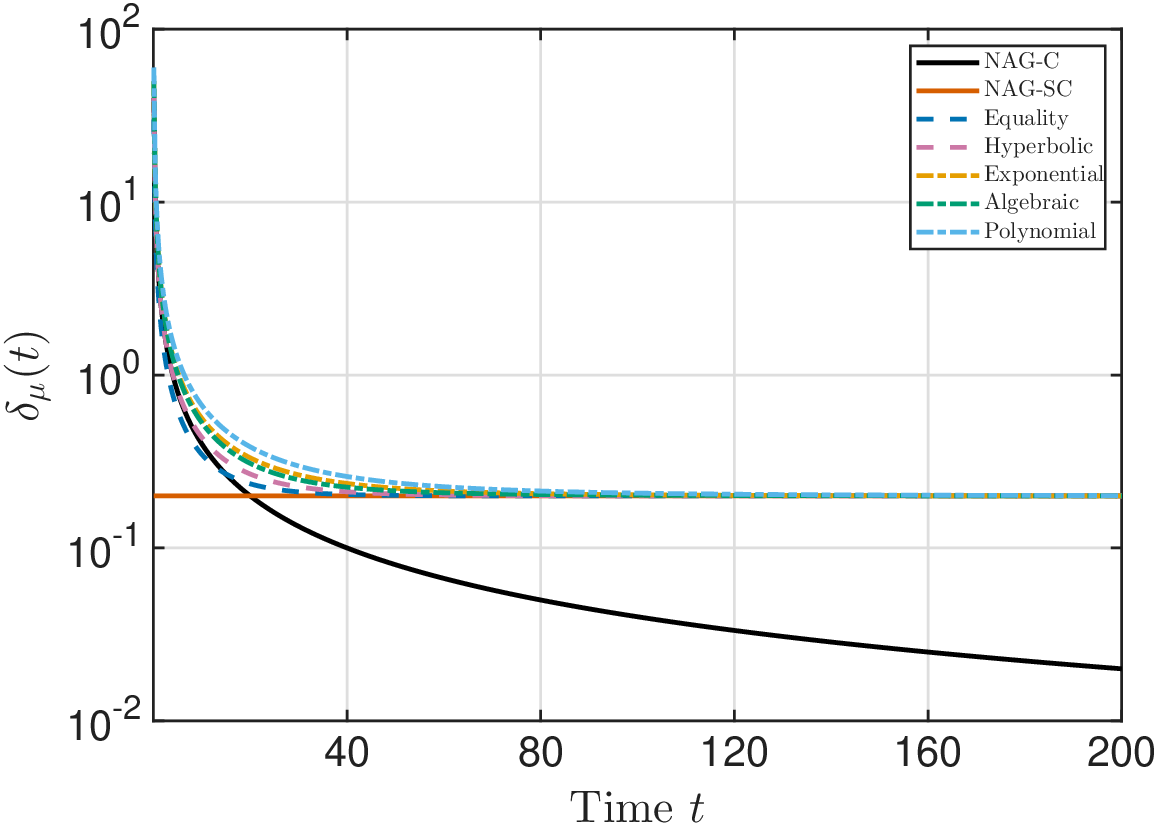}
\end{subfigure}
\caption{Objective residuals and damping coefficients for  $\mu=10^{-2}$.}
\label{fig_mu02}
\end{figure}

 \begin{figure}[H]
\centering
\begin{subfigure}[t]{0.49\textwidth}
      \raggedright
    \includegraphics[width=\linewidth]{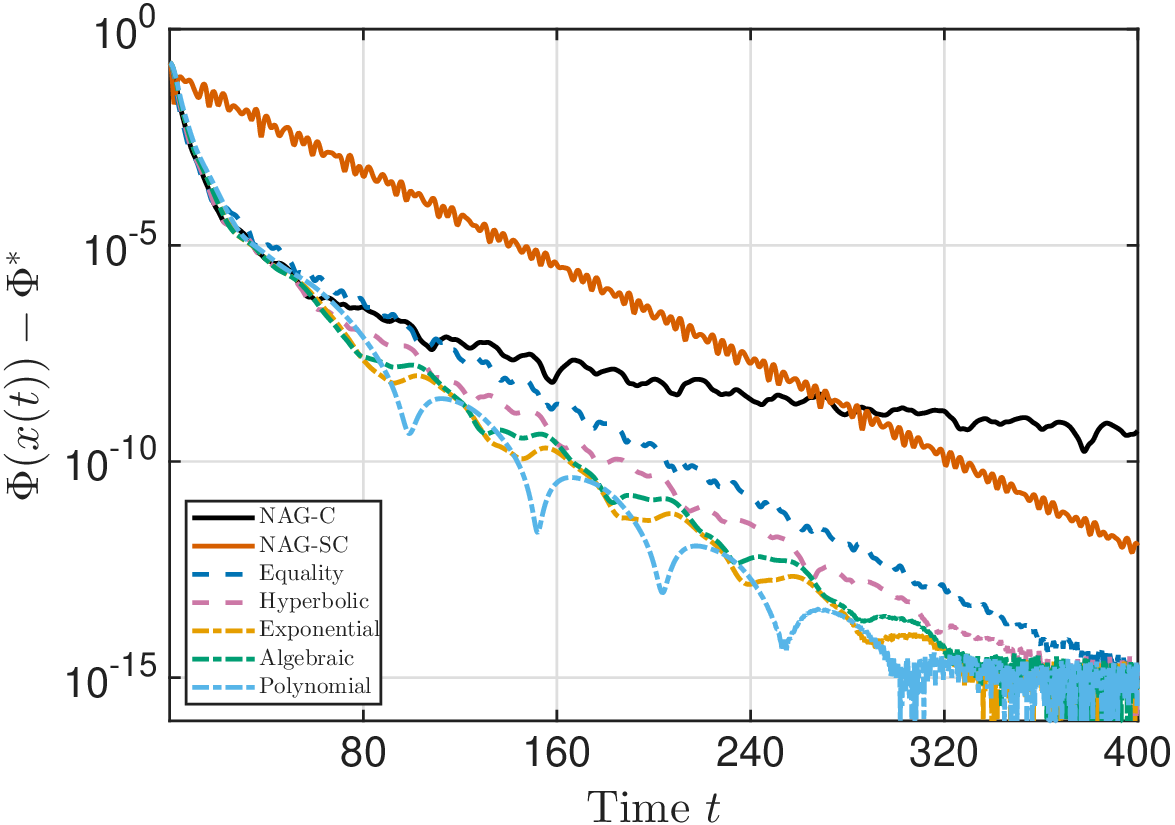}
\end{subfigure}
\hfill
\begin{subfigure}[t]{0.49\textwidth}
      \raggedright
    \includegraphics[width=\linewidth]{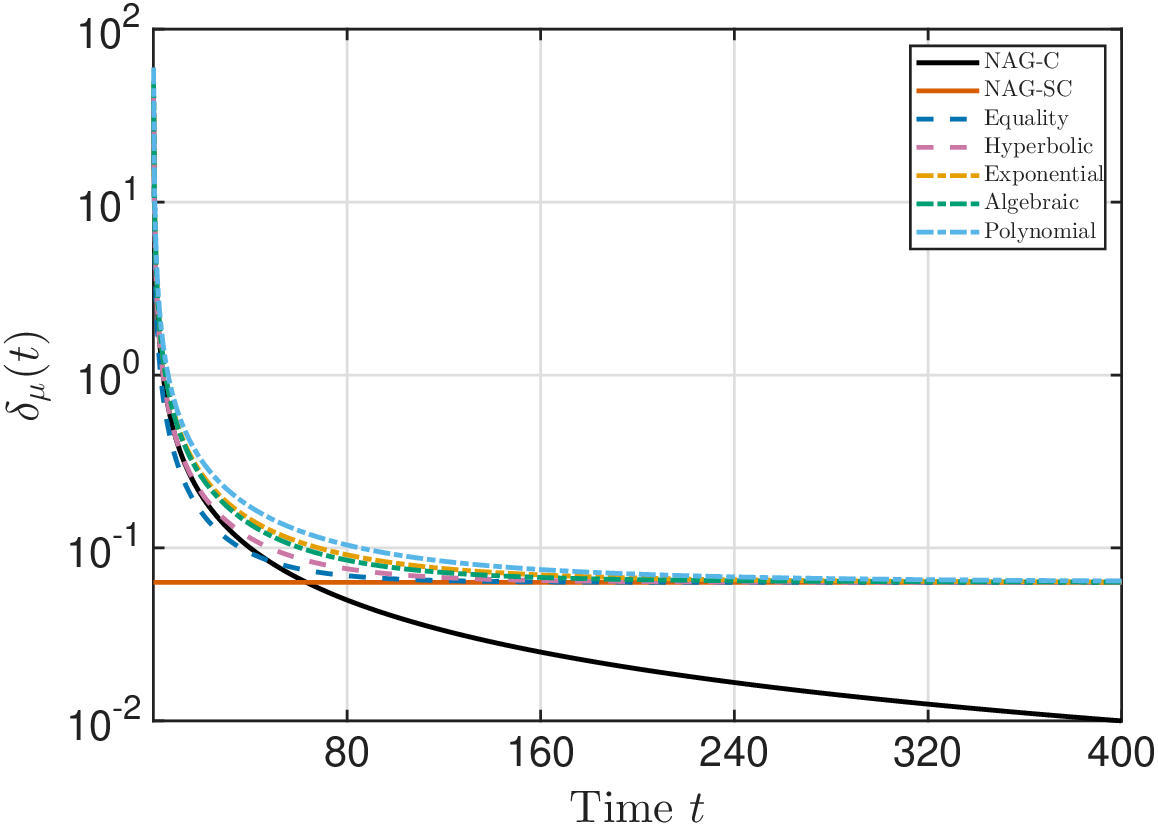}
\end{subfigure}
\caption{Objective residuals and damping coefficients for  $\mu=10^{-3}$.}
\label{fig_mu03}
\end{figure}

 \begin{figure}[H]
\centering
\begin{subfigure}[t]{0.49\textwidth}
      \raggedright
    \includegraphics[width=\linewidth]{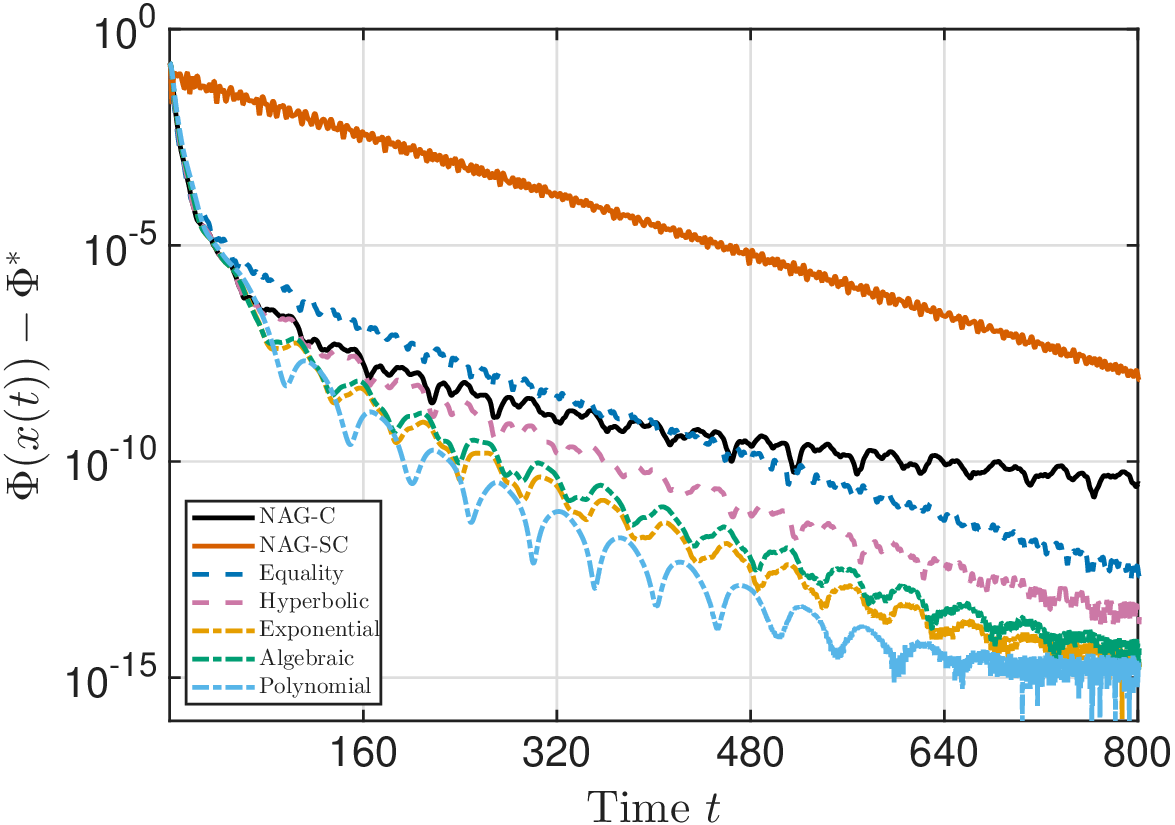}
\end{subfigure}
\hfill
\begin{subfigure}[t]{0.49\textwidth}
      \raggedright
    \includegraphics[width=\linewidth]{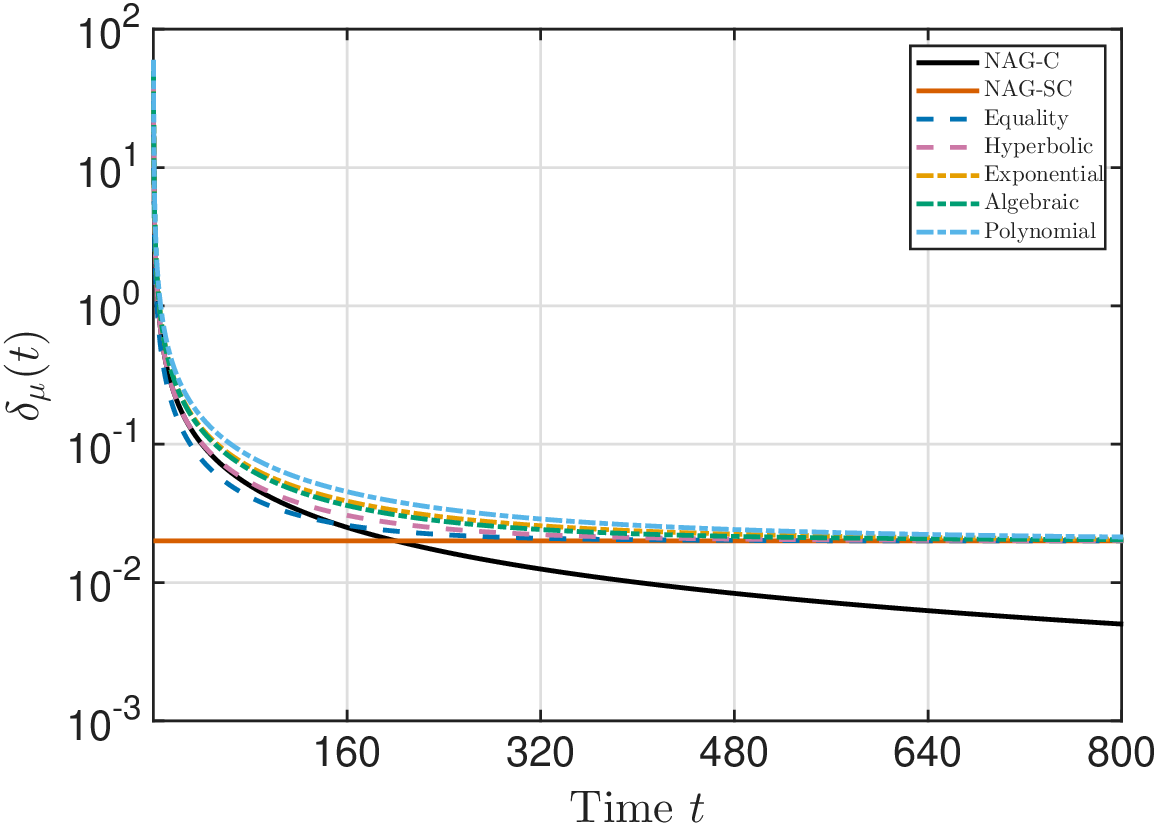}
\end{subfigure}
\caption{Objective residuals and damping coefficients for  $\mu=10^{-4}$.}
\label{fig_mu04}
\end{figure}

 \begin{figure}[H]
\centering
\begin{subfigure}[t]{0.49\textwidth}
      \raggedright
    \includegraphics[width=\linewidth]{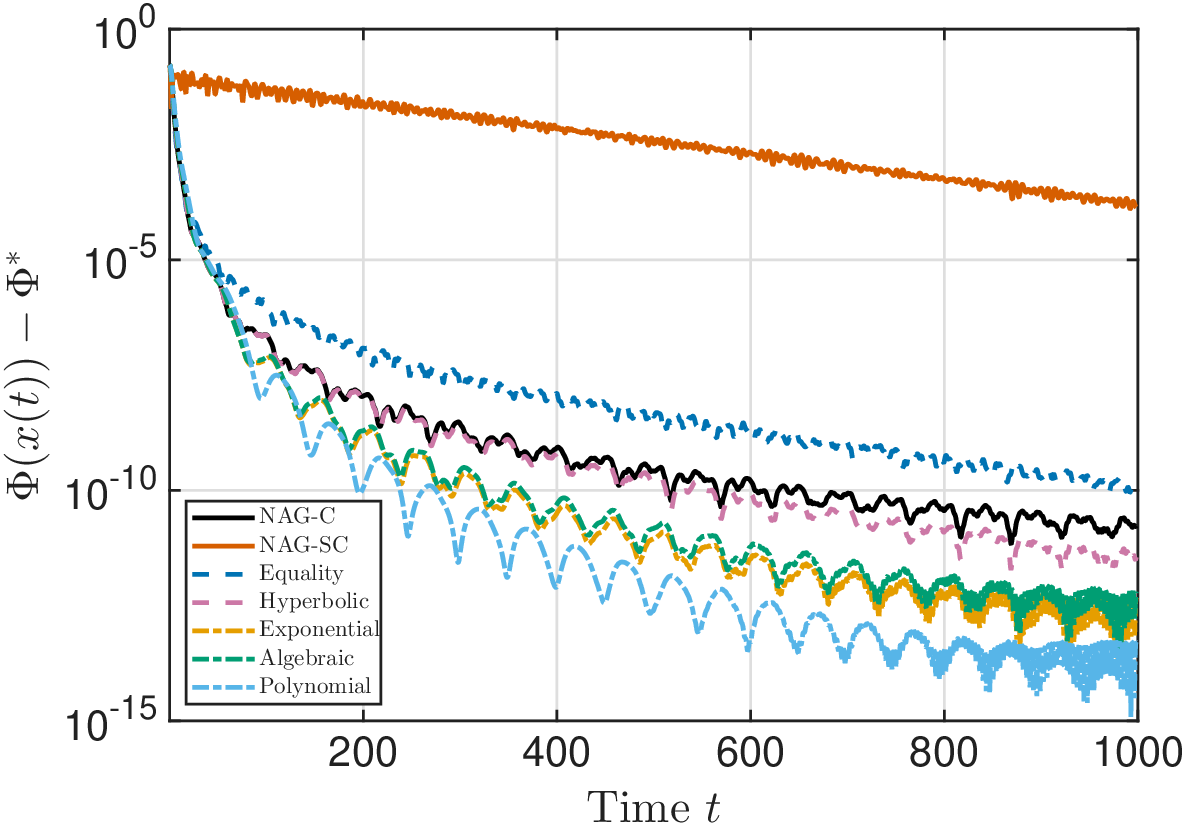}
\end{subfigure}
\hfill
\begin{subfigure}[t]{0.49\textwidth}
      \raggedright
    \includegraphics[width=\linewidth]{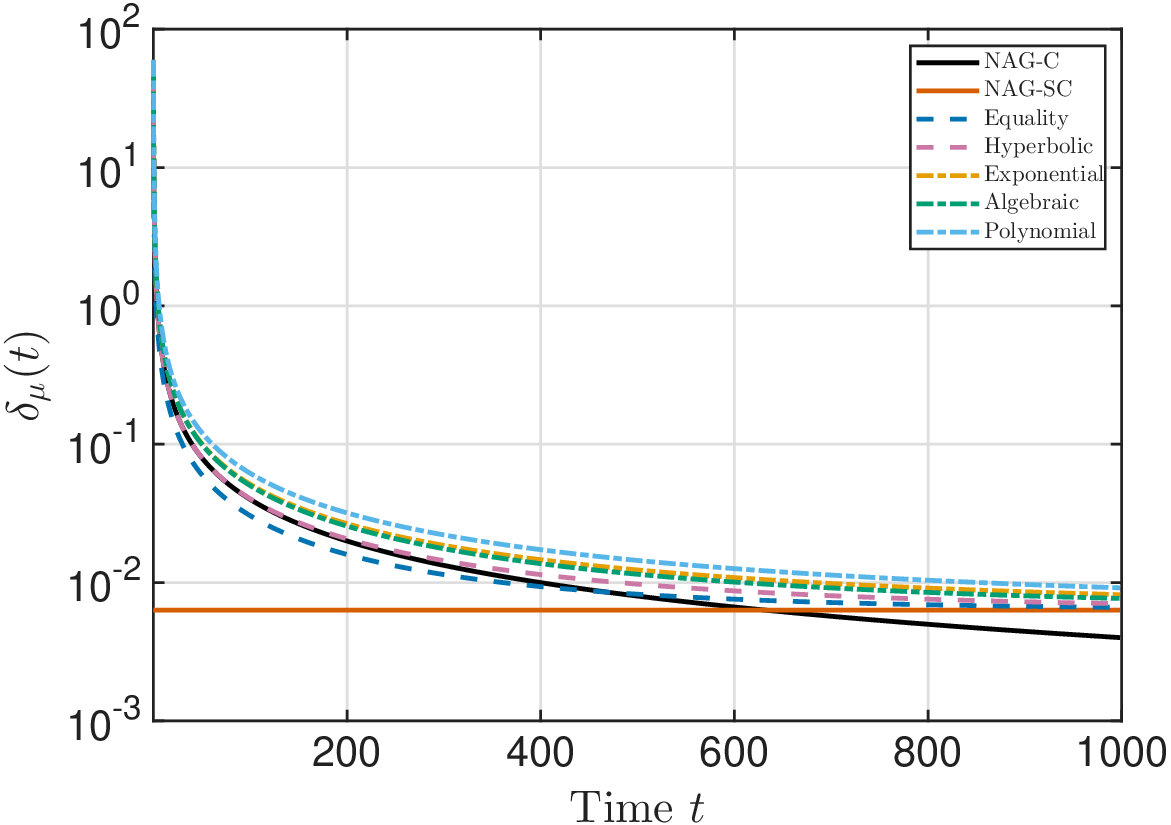}
\end{subfigure}
\caption{Objective residuals and damping coefficients for   $\mu=10^{-5}$.}
\label{fig_mu05}
\end{figure}

For $\mu=10^{-2}$, Figure \ref{fig_mu02} shows that, after the initial
transient, the transition dynamics $C_3$--$C_7$ produce smaller objective
residuals than both endpoint flows (NAG-C and NAG-SC) over the reported time interval. The
differences among the transition choices are relatively small, although
their damping coefficients follow visibly different paths.
As $\mu$ decreases, the differences among the transition choices become
larger. Figures~\ref{fig_mu03}--\ref{fig_mu05} show that transition
damping coefficients that approach the strongly convex limit
$2\sqrt{\mu}$ more gradually generally produce smaller objective residuals
over the reported time intervals. In particular, $C_4$ produces smaller
objective residuals than the equality-branch choice $C_3$, while
$C_5$--$C_7$ yield further reductions  over most of the reported time intervals. For $\mu=10^{-5}$, $C_7$ gives the smallest objective gap over
most of the simulation interval. Thus, for the problem and parameter
range considered here, a more gradual transition toward the strongly
convex damping is beneficial when $\mu$ is small.
The behavior of the equality-branch transition $C_3$, considered in
\cite{LuoChen2022,KimYang2023}, is particularly clear in
Figures \ref{fig_mu04} and \ref{fig_mu05}. For these small values of $\mu$, the objective
gap produced by $C_3$ remains larger than that of the convex endpoint flow
$C_1$ over most of the reported time interval. This indicates that, on the
tested problem, a rapid approach to the strongly convex damping
$2\sqrt{\mu}$ may lead to less favorable  finite-time behavior when $\mu$ is
small. The sampled construction avoids this restriction by allowing both
the convex-limit damping and the transition toward $2\sqrt{\mu}$ to be
adjusted.

For the tested values of $\mu$, the proposed transition dynamics produce
substantially smaller objective residuals than both endpoint flows when the
strong convexity parameter is small. Their improvement over the equality
branch becomes more evident as $\mu$ decreases. This numerical observation
indicates that the most effective transition depends on the structure of
the objective function,   the strong convexity parameter,
and the required solution accuracy.

 \subsection{Elastic-net regularized problem}
\label{subsec_discrete_numerics}
 
 We next evaluate the sampled discrete transition coefficients on an
elastic-net regularized  problem. Such models arise in sparse
recovery, imaging, and high-dimensional regression
\cite{Tibshirani1996,Daubechiesl04}. The $\ell_1$ term promotes sparse
solutions, while the quadratic regularization term supplies strong
convexity and stabilizes the problem when the forward operator is rank
deficient \cite{ZouHastie2005}.

We consider the elastic-net regularized problem
\begin{equation}\label{eq_discrete_test_problem}
    \min_{x\in\mathbb R^n}
    \Phi(x)
    :=
    \frac{1}{2}\norm{Ax-b}^2
    +
    \frac{\mu}{2}\norm{x}^2
    +
    \lambda\norm{x}_1,
\end{equation}
where $0<\mu\ll1$ is the quadratic regularization parameter. We write
\[
    f(x)
    :=
    \frac{1}{2}\norm{Ax-b}^2
    +
    \frac{\mu}{2}\norm{x}^2,
    \qquad
    g(x):=\lambda\norm{x}_1.
\]
Then $f$ is $\mu$-strongly convex, and its gradient is Lipschitz
continuous with constant
$
    L=\norm{A}^2+\mu.
$
The proximal mapping of $g$ is the componentwise soft-thresholding
operator
\[
    \prox_{\tau g}(v)
    =
    \operatorname{sign}(v)
    \odot
    \max\{|v|-\tau\lambda,0\}.
\]

For each experiment, we generate a Gaussian matrix
$A_0\in\mathbb R^{m\times n}$ with $m=2000$ and $n=4000$, whose entries
are independent and distributed according to $\mathcal N(0,1/m)$. We
then normalize it by its operator norm:
$
    A:={A_0}/{\norm{A_0}},
$
so that $\norm{A}=1$. Since $m<n$, the matrix $A^\top A$ is singular
with probability one. The data-fitting term alone is therefore not
strongly convex, and the strong convexity of
\eqref{eq_discrete_test_problem} is supplied by
$\mu\norm{x}^2/2$. We generate a reference vector $x^\dagger\in\mathbb R^n$ with
$q=0.05n$ nonzero entries. Its support is sampled uniformly among all
subsets of $\{1,\ldots,n\}$ of cardinality $q$. On this support, the
entries are generated according to
$
    x_j^\dagger=\varepsilon_j(1+u_j),
$
where $\varepsilon_j$ is uniformly distributed on $\{-1,1\}$ and
$u_j\sim\mathcal U(0,1)$. 
  We also draw
$\zeta\sim\mathcal N(0,I_m)$ and define the normalized noise direction
$
    d:=\zeta/\|\zeta\|.
$
For a prescribed relative noise level $\delta=10^{-2}$, the perturbation and the
corresponding observation are defined by $ b:=Ax^\dagger+\delta\|Ax^\dagger\|d$.
Since $\|d\|=1$, it follows that
$
    \frac{\|b-Ax^\dagger\|}
         {\|Ax^\dagger\|}
    =
    \delta.
$
The same realization of $A$, $x^{\dagger}$, and $d$ is used in both
experiments, and only the strong convexity parameter $\mu$ is varied.

 \begin{figure}[H]
\centering
\begin{subfigure}[t]{0.49\textwidth}
      \raggedright
    \includegraphics[width=\linewidth]{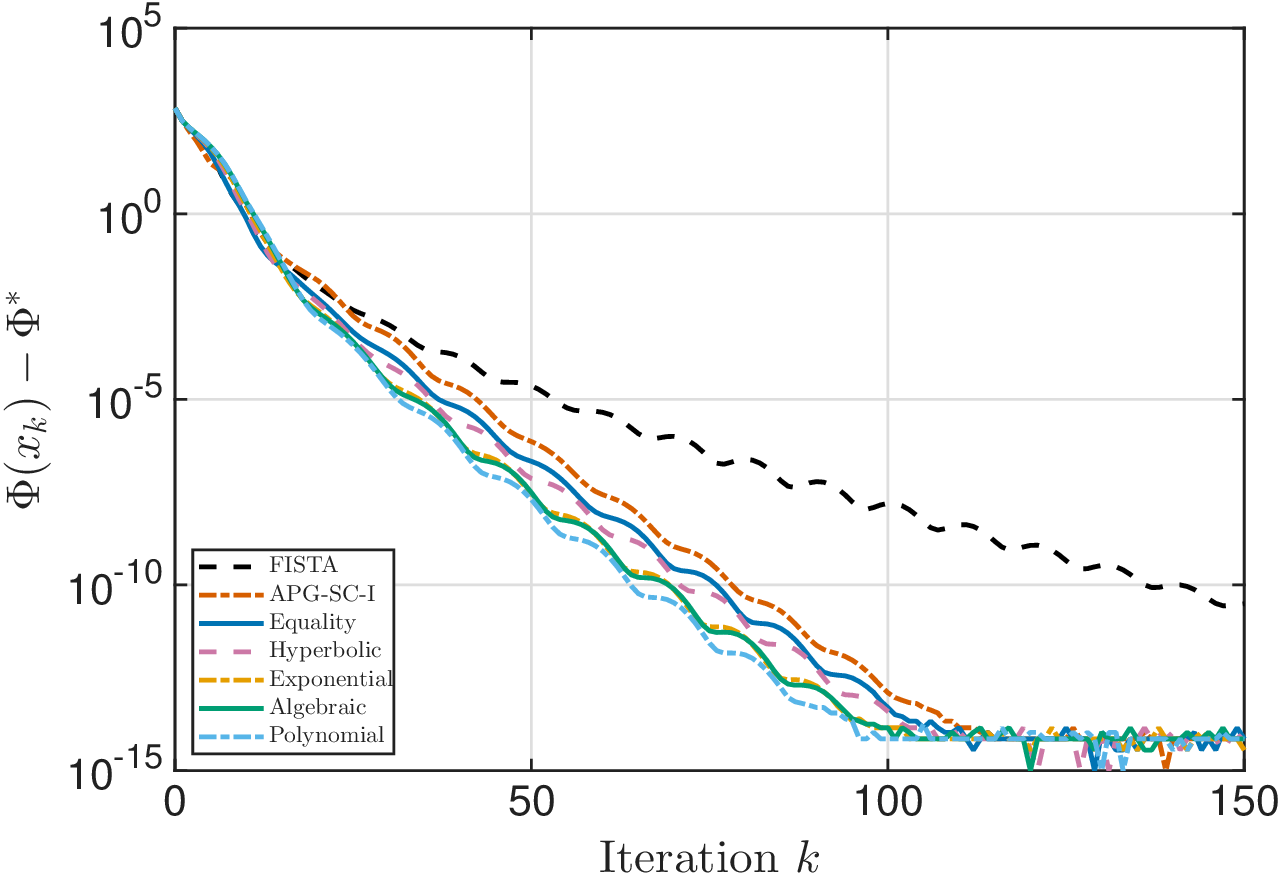}
\end{subfigure}
\hfill
\begin{subfigure}[t]{0.49\textwidth}
      \raggedright
    \includegraphics[width=\linewidth]{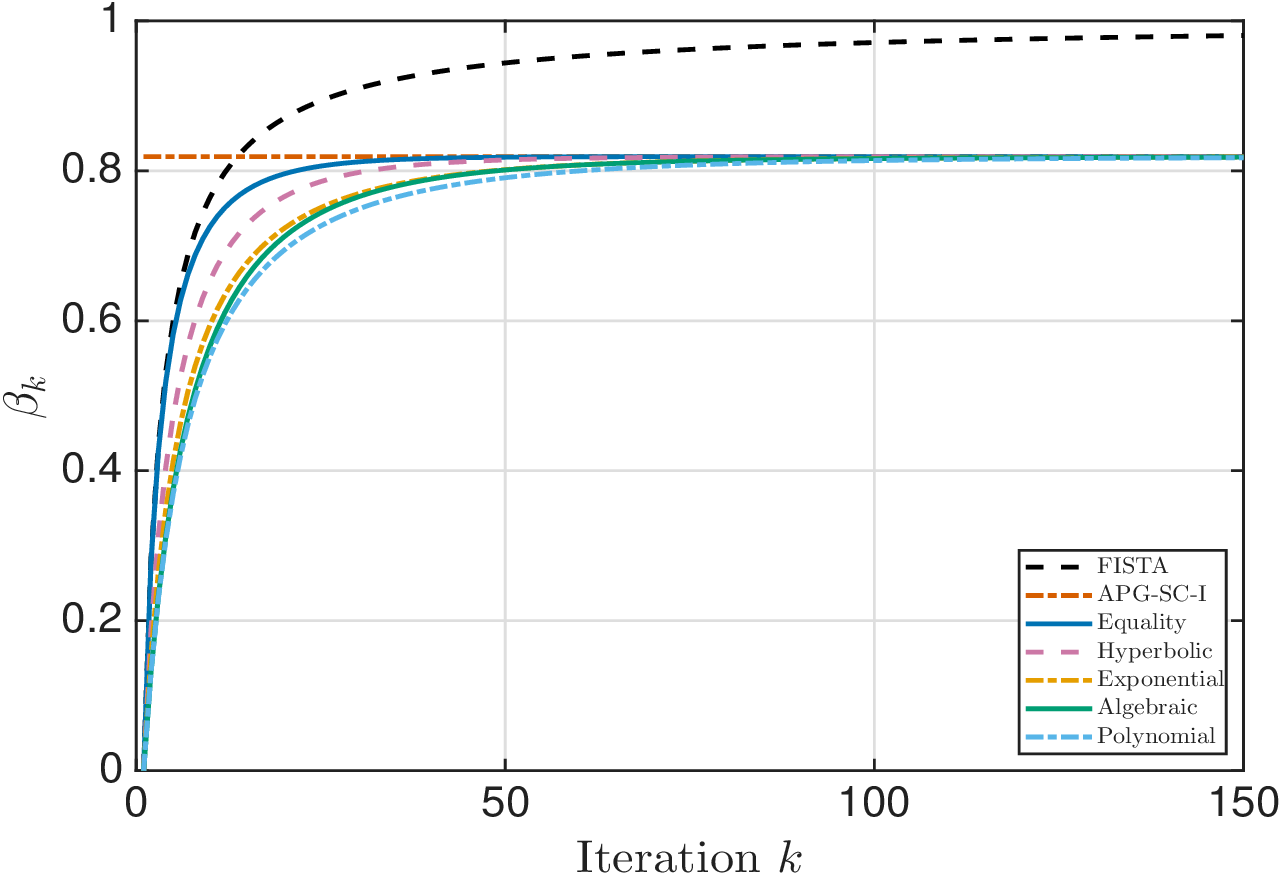}
\end{subfigure}
\caption{Objective residuals and inertial coefficients for   $\mu=10^{-2}$.}
\label{fig_oneu02}
\end{figure}

 \begin{figure}[H]
\centering
\begin{subfigure}[t]{0.49\textwidth}
      \raggedright
    \includegraphics[width=\linewidth]{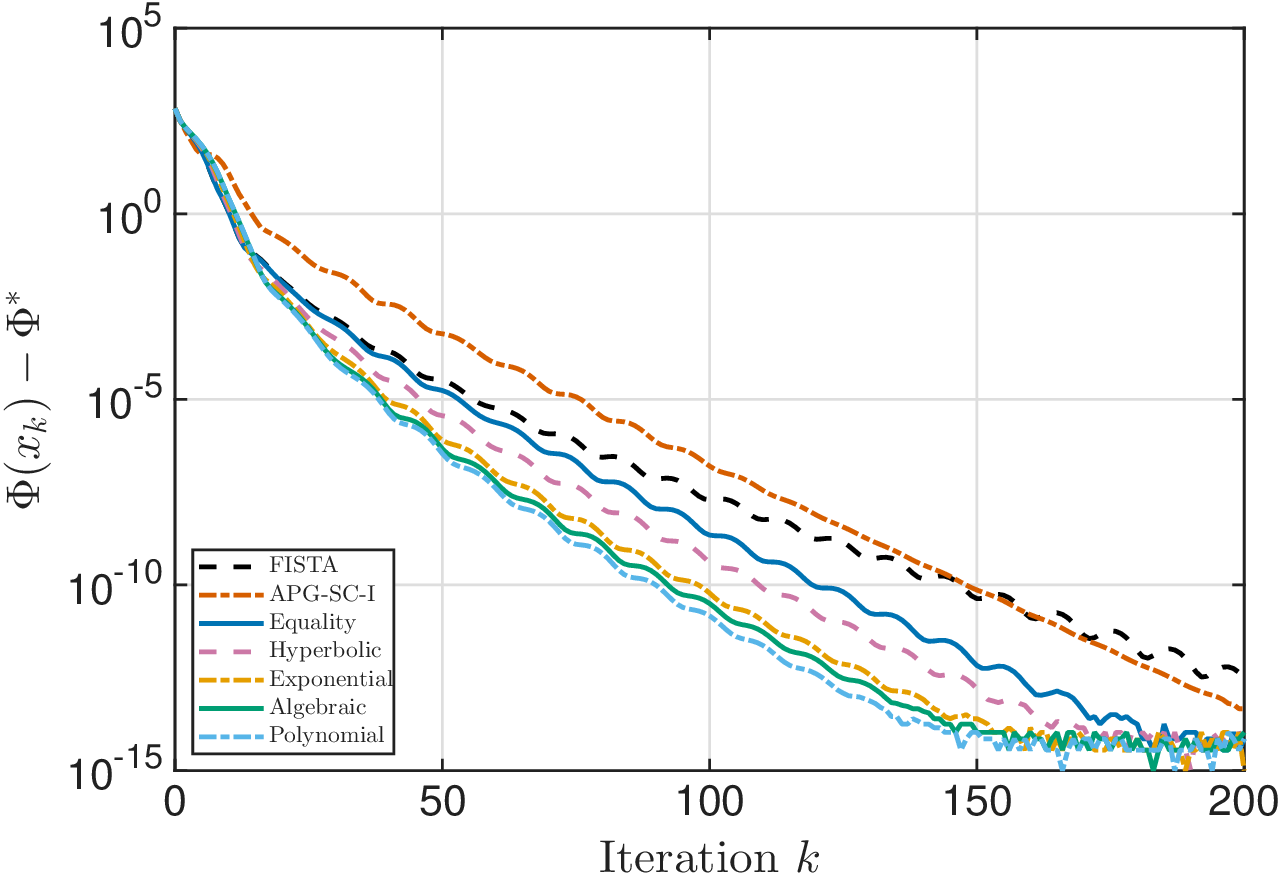}
\end{subfigure}
\hfill
\begin{subfigure}[t]{0.49\textwidth}
      \raggedright
    \includegraphics[width=\linewidth]{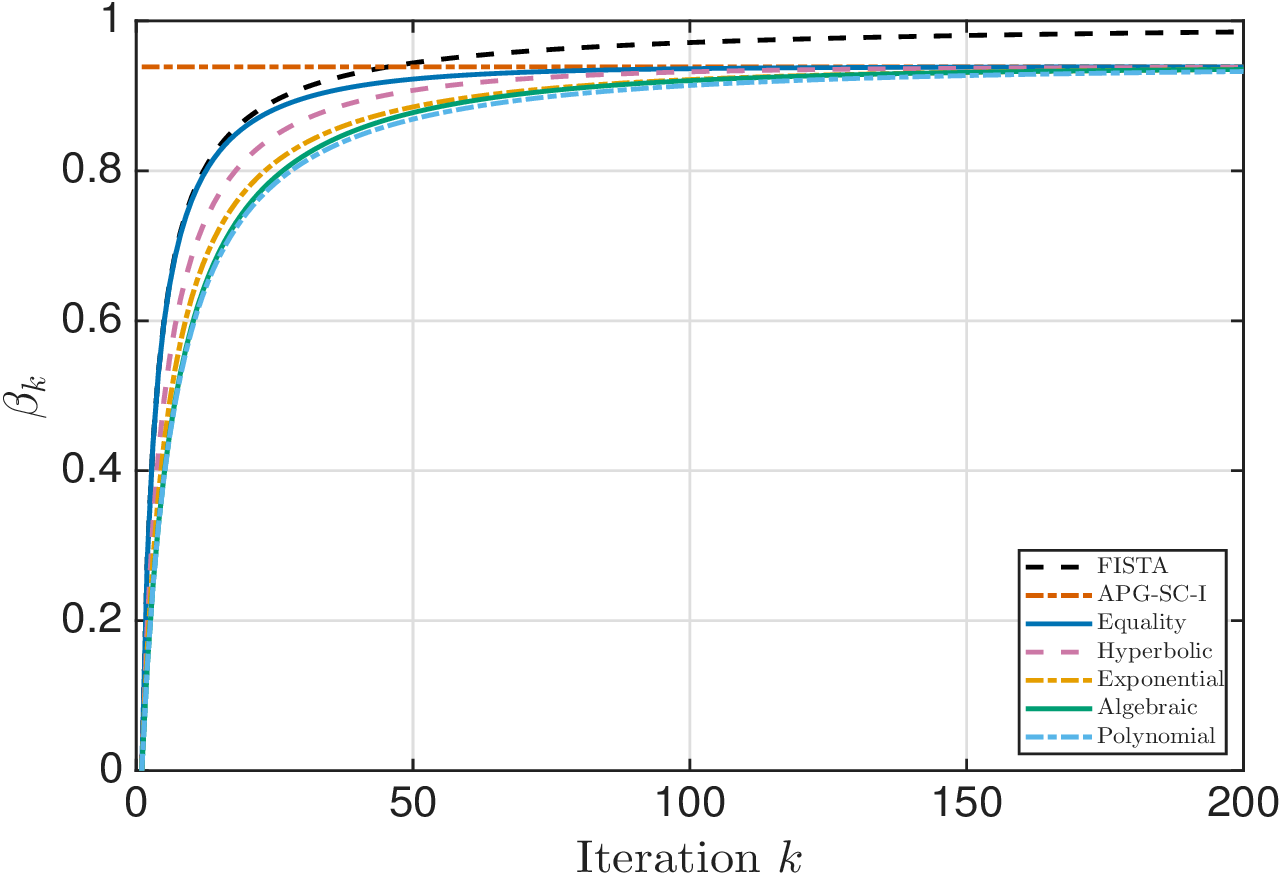}
\end{subfigure}
\caption{Objective residuals and inertial coefficients for   $\mu=10^{-3}$.}
\label{fig_oneu03}
\end{figure}

All methods use the one-sequence accelerated forward-backward scheme
\eqref{eq_one_algorithm} and differ only in their inertial coefficients.
The convex  method is the fast iterative
shrinkage-thresholding algorithm (FISTA) \cite{BeckFista}, with
$
    \beta_k
    =
    \frac{t_k-1}{t_{k+1}}$ with $t_{k+1}
    =
    \frac{1+\sqrt{1+4t_k^2}}{2}.
$
The strongly convex  method is the first-type strongly convex
Nesterov accelerated proximal-gradient scheme (APG-SC-I), in the terminology of
\cite{LinML}, with constant inertial coefficient
$
    \beta_k
    =
    \frac{1-\sqrt{\mu/L}}
         {1+\sqrt{\mu/L}}.
$
The equality-branch coefficient is defined by
\eqref{eq_beta_transition}. The resulting one-sequence scheme coincides
with \cite[Algorithm~5]{ChambolleActa}. The four sampled transition sequences are generated by \eqref{eq_beta}, with $\theta_k$ defined by \eqref{eq_sample_theta}. We use the
hyperbolic transition with $\kappa=1/3$, the exponential transition with
$\kappa=1/4$, the algebraic transition with $\kappa=1/5$, and the
polynomial transition with $\kappa=1/5$ and $p=9$. We fix $\lambda=0.1$. Since $\norm{A}=1$, the smoothness constant is
$
    L=1+\mu,
$
and all methods use
$
    s
    =
    \frac{1}{L}
    =
    \frac{1}{1+\mu}.
$
A high-accuracy numerical reference is obtained by running APG-SC-I for $5000$ iterations. The resulting
approximation is denoted by $x^\star$, and the convergence  report
$
    \Phi(x_k)-\Phi^\star.
$

 Figures~\ref{fig_oneu02} and \ref{fig_oneu03} compare the objective residuals
and inertial coefficients for $\mu=10^{-2}$ and $\mu=10^{-3}$,
respectively. The transition coefficients approach the common strongly
convex limit
$
    \frac{1-\sqrt{\mu/L}}
         {1+\sqrt{\mu/L}}
$
at different speeds. For both values of $\mu$, the sampled transition schemes \eqref{eq_sampled_coefficients} produce smaller
objective residuals than FISTA and the constant-coefficient strongly convex
Nesterov scheme over most of the reported iterations. They also reach the
numerical accuracy floor earlier. Although the equality branch selects the
largest admissible coefficient at each iteration, its rapid transition
toward the strongly convex limiting coefficient does not yield the
smallest objective residual in these experiments. The sampled sequences
approach this limit more gradually and produce smaller objective residuals than
the equality branch, with the algebraic and polynomial transitions giving
the smallest residuals among the tested choices. As $\mu$ decreases, the differences among the sampled inertial
coefficients become more apparent. In particular, the algebraic and
polynomial transitions produce smaller objective residuals than the equality
branch over most of the reported iterations. This ordering is consistent
with that observed for the continuous-time dynamics and indicates that the
effect of the transition toward the strongly convex regime is preserved
in the discrete schemes.

\subsection{TV-Huber ROF denoising}

We consider the TV-Huber ROF denoising model \cite{ChambolleActa,Rudin}:
\[
    \min_{u\in\mathbb{R}^{m\times n}}
    \left\{
        \lambda H_{\epsilon}(u)
        +\frac{1}{2}\lVert u-u^{0}\rVert^{2}
    \right\}.
\]
The ROF model balances fidelity to the observed noisy image $u^{0}$
with total-variation regularization, which promotes piecewise smooth
reconstructions while preserving sharp edges. The parameter
$\lambda>0$ controls the relative strength of the regularization.

Let
\[
    Du=(D_1u,D_2u)
\]
denote the discrete image gradient, where $D_1$ and $D_2$ are the
horizontal and vertical forward finite-difference operators. The
Huber-smoothed isotropic total variation is defined by
\[
    H_{\epsilon}(u)
    =
    \sum_{i=1}^{m}\sum_{j=1}^{n}
    h_{\epsilon}\bigl(\norm{(Du)_{i,j}}_2\bigr),
\]
where
\[
    h_{\epsilon}(t)
    =
    \begin{cases}
        \dfrac{t^{2}}{2\epsilon},
        & 0\leq t\leq\epsilon,\\[5pt]
        t-\dfrac{\epsilon}{2},
        & t>\epsilon.
    \end{cases}
\]
The parameter $\epsilon>0$ smooths the TV functional near the origin
while retaining its approximately linear growth for large gradients. The corresponding dual problem is:
\begin{equation}\label{eq_ques3}
	 \min_{x\in\mathbb{R}^{m\times n\times 2}}
    \Phi(x)
    :=
    \frac{1}{2}\lVert D^{*}x-u^{0}\rVert^{2}
    +\frac{\epsilon}{2\lambda}\lVert x\rVert^{2}
    +\delta_{C}(x),
\end{equation}
where $D^{*}$ denotes the adjoint of $D$, namely the discrete negative
divergence operator, and
\[
    C
    :=
    \left\{
        x\in\mathbb{R}^{m\times n\times 2}
        :
        \lVert x_{i,j}\rVert\leq\lambda
        \ \text{for all }(i,j)
    \right\}.
\]
The indicator function $\delta_{C}$ enforces the pointwise dual
constraint. Once a dual solution $x$ is obtained, the corresponding
denoised image is recovered by
$
    u=u^{0}-D^{*}x.
$

We use the  following equivalent composite decomposition of problem \eqref{eq_ques3}:
\[
    \Phi(x)=f(x)+g(x),
\]
where
\[
    f(x)
    :=
    \frac{1}{2}\lVert D^{*}x-u^{0}\rVert^{2}
    +\frac{\mu}{2}\lVert x\rVert^{2},
    \qquad
    g(x):=\delta_{C}(x),
    \qquad
    \mu:=\frac{\epsilon}{\lambda}.
\]
Then$
    \nabla f(p)
    =
    D\bigl(D^{*}p-u^{0}\bigr)+\mu p.
$
Since $\|D\|^2\le 8$ \cite{ChambolleJMIV}, the gradient of $f$ is
Lipschitz continuous with constant
$
    L=\|D\|^2+\mu\le 8+\mu.
$
Accordingly, we take $L=8+\mu$ and $s=1/L$.

The proximal mapping of $g$ is the pixelwise Euclidean projection
onto $C$:
\[
    \bigl[
        \operatorname{prox}_{\tau g}(q)
    \bigr]_{i,j}
    =
    \frac{q_{i,j}}
    {\max\left\{
        1,\,
        \lVert q_{i,j}\rVert_{2}/\lambda
    \right\}}.
\]

 \begin{figure}[H]
\centering
\begin{subfigure}[t]{0.49\textwidth}
      \raggedright
    \includegraphics[width=\linewidth]{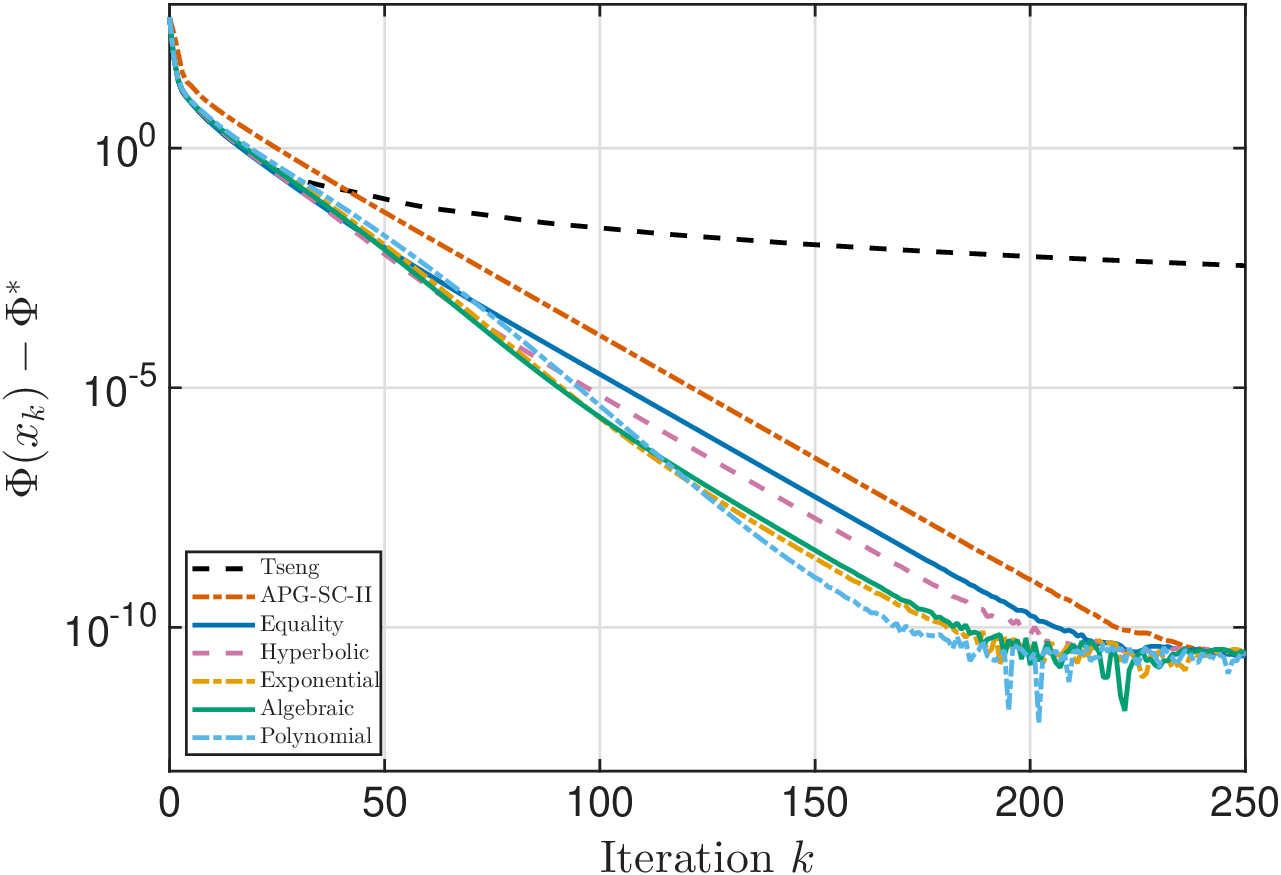}
\end{subfigure}
\hfill
\begin{subfigure}[t]{0.49\textwidth}
      \raggedright
    \includegraphics[width=\linewidth]{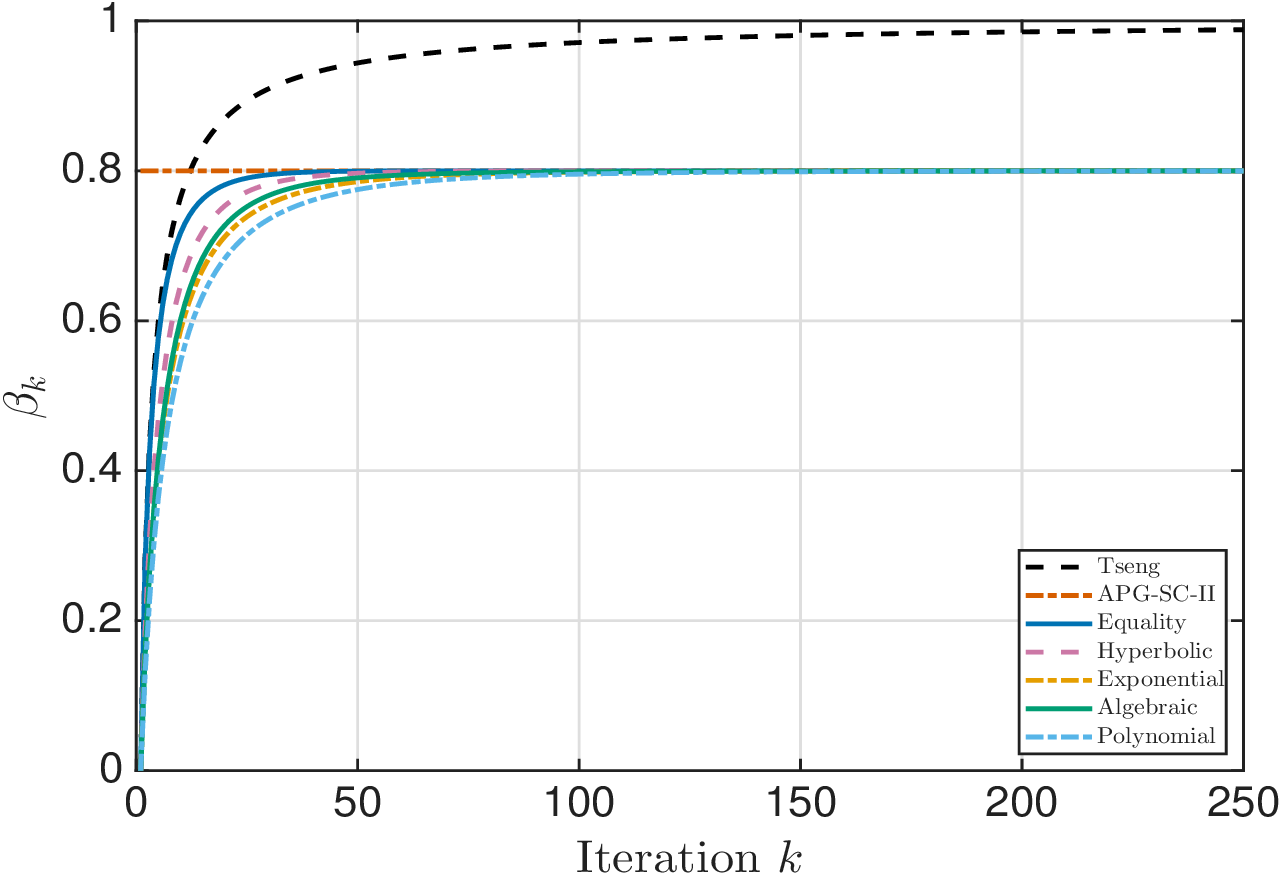}
\end{subfigure}
\caption{Objective residuals and inertial coefficients for solving problem \eqref{eq_ques3}.}
\label{fig_tv}
\end{figure}

\begin{figure}[!h]
\includegraphics[width=6.2in]{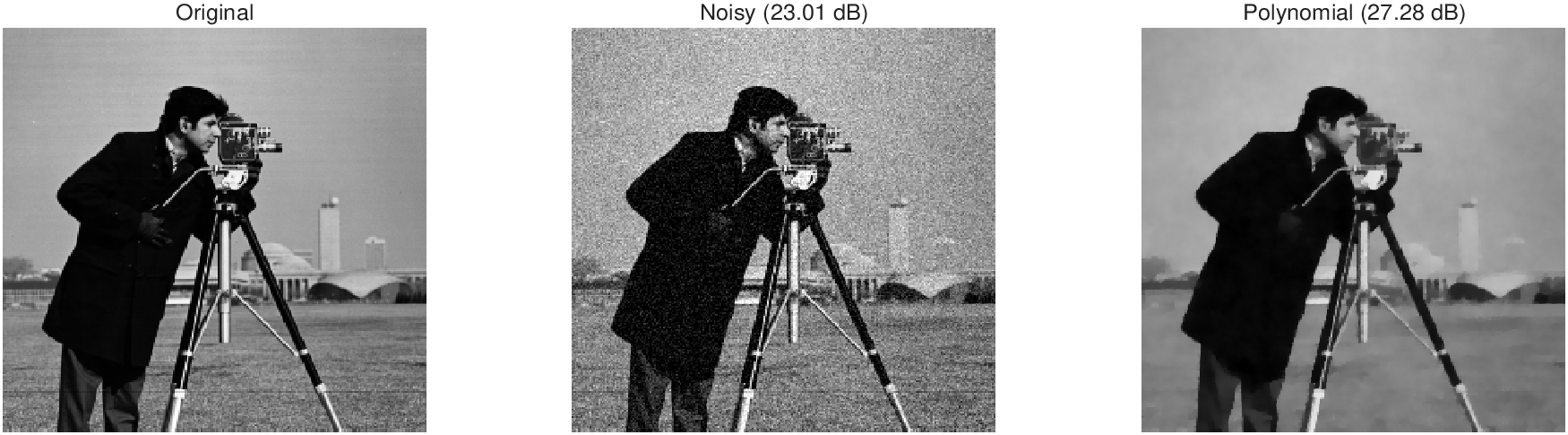}
\caption{Original, Noisy, and Denoised ``cameraman'' images}\label{figCame}
\end{figure}

All methods use the two-sequence accelerated forward-backward scheme
\eqref{eq_two_algorithm} and differ only through their coefficient sequence $\{\theta_k\}$, which also induces $\{\beta_k\}$.
The convex reference method is the Tseng-type scheme
\eqref{eq_convex_two}, with
$t_{k+1}
    =
    \frac{1+\sqrt{1+4t_k^2}}{2}$.
The strongly convex reference method is the second-type strongly convex
Nesterov accelerated proximal-gradient scheme (APG-SC-II), in the terminology of
\cite{LinML}, with constant inertial coefficient
$
    \beta_k
    =
    \frac{1-\sqrt{\mu/L}}
         {1+\sqrt{\mu/L}}.
$
We also include the equality-branch coefficient
\eqref{eq_beta_transition}. The sampled choices are the hyperbolic
transition with $\kappa=1/3$, the exponential transition with
$\kappa=1/4$, the algebraic transition with $\kappa=1/4$, and the
polynomial transition with $\kappa=1/5$ and $p=7$.

The experiments use a $256\times256$ \texttt{cameraman} image corrupted
by additive Gaussian noise with variance $\sigma^2=0.005$. We set
\[
    \lambda=0.1,
    \qquad
    \epsilon=0.01,
    \qquad
    \mu=\frac{\epsilon}{\lambda}=0.1,
\]
and run each method for $200$ iterations. To obtain a numerical reference, APG-SC-I is run for $5000$ iterations,
and the resulting approximation is denoted by $x^\star$. The convergence
curves report the objective residual
$
    \Phi(x_k)-\Phi^\star.
$

Figure~\ref{fig_tv} shows that the sampled inertial coefficients approach
the common strongly convex limit
$
    \frac{1-\sqrt{\mu/L}}
         {1+\sqrt{\mu/L}}
$
at different speeds. In the present experiment,
\[
    \mu=0.1,
    \qquad
    L=8.1,
    \qquad
    \frac{\mu}{L}=\frac{1}{81}.
\]
During the initial iterations, the transition coefficients that approach
the strongly convex limit more rapidly generally produce smaller objective
residuals than the more gradual choices. The constant-coefficient strongly
convex Nesterov scheme, however, gives larger objective residuals because it
uses the limiting coefficient from the first iteration. At higher
accuracy levels, the exponential, algebraic, and polynomial transitions
produce smaller objective residuals than the equality branch. At the final
iteration, all four sampled transitions give smaller objective residuals than
the convex Tseng-type scheme, the constant-coefficient strongly convex
Nesterov scheme, and the equality branch. Among the tested choices, the
polynomial transition gives the smallest objective gap. These results show that the objective decrease is not ordered solely by
the speed at which the inertial coefficient approaches its strongly
convex limit. In particular, a more gradual transition is not necessarily preferable
when the strong convexity constant is not very small or when the required
solution accuracy is relatively low. The most
effective transition depends on the structure of the objective function,
the strong convexity ratio $\mu/L$, and the required solution accuracy.
For the tested setting, the sampled transition coefficients produce
smaller finite-iteration objective residuals than the three reference choices.

Figure~\ref{figCame} shows the reconstruction obtained using the
two-sequence scheme with the polynomial transition. The peak
signal-to-noise ratio increases from $23.01\,\mathrm{dB}$ for the noisy
image to $27.28\,\mathrm{dB}$ for the reconstructed image, corresponding
to an improvement of $4.27\,\mathrm{dB}$. The reconstruction suppresses
the noise while retaining the principal image structures and edges.

\section{Conclusion}
\label{sec_conclusion}

This paper developed a unified continuous-discrete framework for
accelerated optimization across the convex and strongly convex regimes.
In continuous time, a common coefficient condition and Lyapunov analysis
yield both accelerated convex and exponential strongly convex convergence.
The equality case is connected with the hyperbolic constructions in
\cite{LuoChen2022,KimYang2023}, while the proposed transition families
allow the convex-limit damping and the approach to the strongly convex
damping to be adjusted separately. In discrete time, the same principle
leads to unified one-sequence and two-sequence accelerated
forward-backward schemes, which include FISTA, a Tseng-type method, and
the first- and second-type strongly convex Nesterov schemes. The discrete
equality branch coincides with \cite[Algorithm~5]{ChambolleActa}. The continuous-discrete coefficient
correspondence further provides a dynamical interpretation of this
equality-based parameter rule. Beyond the equality branch, we introduced sampled transition coefficients
that connect a broader family of convex inertial parameters to the
classical strongly convex coefficient while preserving the accelerated
 convergence rates. The numerical experiments show
that this additional flexibility can improve finite-time and
finite-iteration performance, particularly when the strong convexity
parameter is small.  These results identify the transition toward the strongly convex regime as an additional design parameter for accelerated optimization methods.

\end{document}